\documentclass[3p,times]{elsarticle}
\usepackage{graphicx}
\usepackage{sidecap}
\usepackage{amssymb,amsmath,amsthm,amsbsy}
\usepackage{calligra,calrsfs,mathrsfs}
\usepackage{multirow}
\usepackage{ulem,cancel}
\usepackage[usenames,dvipsnames]{color}
\usepackage{copyrightbox}
\usepackage{cases}
\usepackage{pgfplots}
\usepackage{subcaption}
\usepackage{tikz} 
\usepackage{tkz-euclide}
\usepackage{tikz-dimline}
\usetikzlibrary[patterns]
\usepackage{hyperref}
\usepackage[ruled,vlined]{algorithm2e}
\usepackage{algpseudocode}
\usepackage{dsfont}
\usepackage{caption}
\usepackage{epstopdf}
\usepackage{epsfig}
\usepackage{natbib}
\usepackage{grffile}
\usepackage{comment}
\usepackage{enumerate}
\usepackage{booktabs}
\usepackage{accents}
\usepackage{colortbl}
\usepackage{todonotes}
\usepackage{array}
\newcolumntype{P}[1]{>{\centering\arraybackslash}p{#1}}
\usepackage{float}
\usepackage[T1]{fontenc}
\usepackage[latin9]{inputenc}
\usepackage{geometry}
\usepgfplotslibrary{fillbetween}
\usetikzlibrary{patterns}
\usepackage{color}
\usepackage{stmaryrd}
\usepackage{setspace}
\usepackage{amsthm}
\usepackage{textcomp}
\usepackage{hyperref}
\usepackage{soul,xcolor}
\usepackage[inline]{enumitem}
\usepackage{textcomp}
\usepackage{adjustbox}
\DeclareMathAlphabet{\mathcal}{OMS}{cmsy}{m}{n}
\newtheorem{theorem}{Theorem}[section]
\newtheorem{proposition}[theorem]{Proposition}

\newcommand{\ifcomment}{\iffalse}

\newcommand{\bs}[1]{\boldsymbol{#1}}

\newdefinition{rem}{Remark}

\biboptions{sort&compress}

\begin{document}

\begin{frontmatter}

\title{Preconditioning and Reduced-Order Modeling of Navier-Stokes Equations in Complex Porous Microstructures}

\author[psu]{Kangan Li}
\ead{kbl5610@psu.edu}
\author[psu]{Yashar Mehmani}
\ead{yzm5192@psu.edu}
\address[psu]{Energy and Mineral Engineering Department, The Pennsylvania State University, University Park, Pennsylvania 16802}
\cortext[ca]{Corresponding author: Yashar Mehmani. Email: yzm5192@psu.edu}

\begin{abstract}
We aim to solve the incompressible Navier-Stokes equations within the complex microstructure of a porous material. Discretizing the equations on a fine grid using a staggered (e.g., marker-and-cell, mixed FEM) scheme results in a nonlinear residual. Adopting the Newton method, a linear system must be solved at each iteration, which is large, ill-conditioned, and has a saddle-point structure. This demands an iterative (e.g., Krylov) solver, that requires preconditioning to ensure rapid convergence. We propose two monolithic \textit{algebraic} preconditioners, $a\mathrm{PLMM_{NS}}$ and $a\mathrm{PNM_{NS}}$, that are generalizations of previously proposed forms by the authors for the Stokes equations ($a\mathrm{PLMM_{S}}$ and $a\mathrm{PNM_{S}}$). The former is based on the pore-level multiscale method (PLMM) and the latter on the pore network model (PNM), both successful approximate solvers. We also formulate faster-converging geometric preconditioners $g$PLMM and $g$PNM, which impose $\partial_n\bs{u}\!=\!0$ (zero normal-gradient of velocity) exactly at subdomain interfaces. Finally, we propose an accurate coarse-scale solver for the steady-state Navier-Stokes equations based on $g$PLMM, capable of computing approximate solutions orders of magnitude faster. We benchmark our preconditioners against state-of-the-art block preconditioners and show $g$PLMM is the best-performing one, followed closely by $a\mathrm{PLMM_{S}}$ for steady-state flow and $a\mathrm{PLMM_{NS}}$ for transient flow. All preconditioners can be built and applied on parallel machines.
\end{abstract}


\end{frontmatter}

\section{Introduction} \label{sec:intro}
Porous media span a wide range of natural (e.g., soil, rock) and engineered (e.g., foamed metal, ceramic) materials with extreme microstructural complexity \cite{bear2013dynamics, liu2014porous}. High-resolution X-ray $\mu$CT imaging \cite{wildenschild2013x} is often used to capture this complexity, for subsequent use in pore-scale modeling. Here, we concern ourselves with the Navier-Stokes equations:
\begin{subequations} \label{eq:governing_eqs}
\begin{align}
\rho (\partial_t\bs{u} + \bs{u} \cdot \nabla \bs{u}) - \mu\,\Delta\bs{u} + \nabla p &= \rho \bs{b} \label{eq:navstokes_mom} \\
\nabla\cdot\bs{u} &= 0 \label{eq:navstokes_mass}
\end{align}
\end{subequations}
describing the flow of a single-phase, incompressible, and Newtonian fluid through the void space $\Omega$ of a porous sample. Eqs.\ref{eq:navstokes_mom} and \ref{eq:navstokes_mass} denote momentum and mass conservation, respectively. Variables $\bs{u}$, $p$, $\rho$, $\mu$, and $\bs{b}$ represent velocity, pressure, density, viscosity, and body force per unit mass, respectively. The symbols $\nabla$, $\Delta$, and $\nabla \cdot$ are the gradient, Laplacian, and divergence operators. Direct numerical simulation (DNS) methods discretize Eq.\ref{eq:governing_eqs} on a fine grid over $\Omega$. Popular choices are staggered finite volume (FVM) \cite{perot2000stagFVM}, marker and cell (MAC) finite difference \cite{harlow1965MAC}, and mixed finite element (MFEM) \cite{bochev2006MFEMstab, loghin2002MpFpA}, due to their flux conservation properties essential to subsequent solute transport and two-phase flow simulations \cite{mehmani2019mult}. High-fidelity solutions obtained from these methods help characterize geologic sites for CO\textsubscript{2} sequestration \cite{bachu2008CO2}, underground H\textsubscript{2} storage \cite{hanson2022H2DOEreport}, and geothermal energy extraction \cite{barbier2002geothermal}, and design new porous components for fuel cells \cite{andersson2016fuelcell}, electrolyzers \cite{lee2020electrolyzer}, and heat exchangers and thermal insulators \cite{clyne2006heatexchange}.

Algebraic systems arising from the above DNS methods are nonlinear, due to the inertia term $\bs{u} \cdot \nabla \bs{u}$, and are often solved with Newton's method. Linearization leads to the following saddle-point structure for the Newton update:
\begin{equation} \label{eq:Ax=b} 
	J^k \delta^k
	\,
	=
	-r^k
	\qquad
	\Rightarrow
	\qquad
	\begin{bmatrix}
	 \mathrm{F}^k & \mathrm{G}^\top \\
	 \mathrm{G} & \mathrm{O}
	\end{bmatrix}
	\,
	\begin{bmatrix}
	 \delta_u^k \\
	 \delta_p^k
	\end{bmatrix}
	\,
	=
	\,
	-
	\begin{bmatrix}
	 r_u^{k} \\
	 r_p^{k}
	\end{bmatrix}
	\qquad
	\Rightarrow
	\qquad
	\hat{\mathrm{A}}\hat{x}=\hat{b}
\end{equation}
where $r^k$ is the residual vector at iteration $k$ and $J^k$ is the corresponding Jacobian. The Newton update $\delta^k$ consists of a velocity $\delta_u^k$ and a pressure $\delta_p^k$ component. The next iterate is obtained from $\bs{u}^{k+1}\! =\! \bs{u}^{k} + \delta_u^k$ and $p^{k+1}\! =\! p^{k} + \delta_p^k$. Residual vectors $\smash{r_u^k}$ and $\smash{r_p^k}$ correspond to discretized momentum and mass balance in Eq.\ref{eq:governing_eqs}, respectively. For simplicity, we denote the linearized system by $\mathrm{\hat{A}}\hat{x}\!=\!\hat{b}$, which is often very large. Solving it requires iterative (e.g., Krylov) solvers, that depend on effective preconditioning to converge rapidly \cite{saad2003book}. Our main goal is to develop such preconditioners.

The most common preconditioners are block-triangular \cite{benzi2005saddle} of the following form (dropping the index $k$):
\begin{equation} \label{eq:MB}
\mathrm{M_B} = 
     \begin{bmatrix}
     \mathrm{F} & \mathrm{G}^\top \\
     \mathrm{O} & \mathrm{X}
     \end{bmatrix}
\qquad\Rightarrow\qquad
\mathrm{M_B^{-1}} = 
     \begin{bmatrix}
     \mathrm{F^{-1}} & -\mathrm{F^{-1}}\mathrm{G}^\top\mathrm{X^{-1}} \\
     \mathrm{O} & \mathrm{X^{-1}}
     \end{bmatrix}
\qquad\Rightarrow\qquad
\mathrm{\hat{A}}\mathrm{M_B^{-1}}\hat{x} = \hat{b}
\end{equation}
applied here as a right-preconditioner, $\mathrm{M_B}$. If $\mathrm{X}$ equals the Schur complement matrix $-\mathrm{GF^{-1}}\mathrm{G}^\top$, convergence is in one step. Note $\mathrm{G}$ is the discrete divergence operator, and its transpose the discrete gradient. The idea behind all block preconditioners is to \textit{approximate} $\mathrm{X}^{-1}$, or more precisely its action on a given vector, because direct computation is equivalent to solving the system. Thus, left-multiplying a vector by $\mathrm{M_B^{-1}}$ requires solving two smaller systems, one involving the coefficient matrix $\mathrm{F}$ and another involving $\mathrm{X}$. Since these are elliptic, algebraic multigrid (AMG) \cite{ruge1987algebraic,notay2010agmg} is a preferred choice \cite{silvester2001StokesVcycle}. Approximate candidates for $\mathrm{X^{-1}}$ abound in the literature \cite{elman1996Grammian, elman1999BFBt, loghin2002MpFpA, kay2002GreenTensor, benzi2005saddle, elman2006BFBtScaled}, among which the \textit{scaled}-BFBt approximation of \cite{elman2006BFBtScaled} was found to exhibit robust performance in complex geometries \cite{mehmani2025multiscale}:
\begin{equation}
\label{eq:BFBt_scaled}
\mathrm{X^{-1}} = - (\mathrm{G}\mathrm{D_F^{-1}}\mathrm{G}^\top)^{-1}
                    \mathrm{G}\mathrm{D_F^{-1}}\mathrm{F}\mathrm{D_F^{-1}}\mathrm{G}^\top
                    (\mathrm{G}\mathrm{D_F^{-1}}\mathrm{G}^\top)^{-1}
\end{equation}
where $\mathrm{D_F}$ is the diagonal matrix comprised of the diagonal entries of $\mathrm{F}$.

Unfortunately, using AMG to solve block systems associated with $\mathrm{F}$ and $\mathrm{X}$, henceforth referred to as $b$AMG, was found by the authors \cite{mehmani2025multiscale} to result in GMRES stagnation for Stokes flow through complex porous geometries. The prefix $b$ here stands for ``block'' preconditioner. A far more robust alternative to AMG was proposed by \cite{mehmani2025multiscale} based on the pore-level multiscale method (PLMM) \cite{mehmani2018mult} and the pore-network model \cite{fatt1956network}, equipped with better prolongation/restriction matrices. We call these $b$PLMM and $b$PNM, which
exhibit faster convergence and no stagnation. Even with these improvements, block-preconditioners remain prohibitively slow for practical pore-scale simulation \cite{mehmani2025multiscale}. 

To address this gap, monolithic preconditioners based on PLMM and PNM were formulated in \cite{mehmani2025multiscale} for the Stokes equations, which are direct algebraic translations of the geometric models these algorithms represent \cite{mehmani2018mult,fatt1956network}. We denote them by $a\mathrm{PLMM_S}$ and $a\mathrm{PNM_S}$, where subscript ``S'' emphasizes the preconditioners are built from the discrete Stokes system. The term ``monolithic'' implies that velocity and pressure unknowns are solved simultaneously, without splitting the system into blocks. Both $a\mathrm{PLMM_S}$ and $a\mathrm{PNM_S}$ are two-level, consisting of a coarse preconditioner $\mathrm{M_G}$ and a smoother $\mathrm{M_L}$. The former encapsulates the physical insights behind the original PLMM and PNM, and the latter is an additive-Schwarz smoother to wipe out high-frequency errors. A key feature of $\mathrm{M_G}$ is that it doubles as an approximate Stokes solver, where $\hat{x}_{aprx}\!=\!\mathrm{M_G^{-1}}\hat{b}$ is usable in many applications. $L_2$-errors with the $\mathrm{M_G}$ of $a\mathrm{PLMM_S}$ are $<\!2$\% for velocity and $<\!1$\% for pressure, and with $a\mathrm{PNM_S}$ they are $<\!6$\% for velocity and $<\!5$\% for pressure \cite{mehmani2025multiscale}.

Our first goal here is to extend $a\mathrm{PLMM_S}$ and $a\mathrm{PNM_S}$ to the Navier-Stokes equations (Eq.\ref{eq:governing_eqs}), referring to the resulting preconditioners as $a\mathrm{PLMM_{NS}}$ and $a\mathrm{PNM_{NS}}$. In PLMM and PNM, $\Omega$ is decomposed into subdomains called \textit{primary} and \textit{dual grids}. In PLMM, primary grids coincide with geometric enlargements of $\Omega$ (or pores), and dual grids (covering primary-grid interfaces) coincide with geometric constrictions (or throats). In PNM, the reverse is true: primary grids are throats and dual grids (and interfaces) are pores. Decompositions are achieved via the watershed segmentation algorithm \cite{beucher1979water}. Shape and correction vectors are built on each primary grid by solving local systems subject to closure boundary conditions (BCs) at interfaces between primary grids. These BCs are $p\!=\!\mathrm{const}$ and $\partial_n\bs{u}\!=\!\bs{0}$, with $\partial_n$ denoting differentiation normal to the interface. For Stokes flow, the BCs are known to be very accurate for PLMM \cite{mehmani2018mult}, due to converging-diverging flow patterns near throats. In constructing $a\mathrm{PLMM_S}$ and $a\mathrm{PNM_S}$ for Stokes equations, closure BCs are easily imposed via purely algebraic operations \cite{mehmani2025multiscale}. We show this is not the case for the Navier-Stokes equations, because $\partial_n\bs{u}\!=\!\bs{0}$ can only be imposed approximately in building $a\mathrm{PLMM_{NS}}$ and $a\mathrm{PNM_{NS}}$. 

To impose $\partial_n\bs{u}\!=\!\bs{0}$ exactly in a preconditioner, we show the residual $r^k$ in Eq.\ref{eq:Ax=b} must be modified by removing the inertia term $\bs{u}\cdot\nabla\bs{u}$ in a small neighborhood of the primary-grid interfaces. This yields the geometric (not algebraic) preconditioners $g$PLMM and $g$PNM, where access to the code assembling $r^k$ is required. Despite this need for mildly intrusive code access, we find $g$PLMM and $g$PNM outperform $a\mathrm{PLMM_{NS}}$ and $a\mathrm{PNM_{NS}}$. We test all preconditioners ($a\mathrm{PLMM_{NS}}$, $a\mathrm{PNM_{NS}}$, $g$PLMM, $g$PNM) in GMRES for a range of 2D/3D porous geometries and benchmark them against state-of-the-art block ($b$AMG, $b$PLMM, $b$PNM) and Stokes-based ($a\mathrm{PLMM_S}$, and $a\mathrm{PNM_S}$) preconditioners. In all cases, $g$PLMM emerges as the best-performing preconditioner. If algebraic preconditioning is desired, $a\mathrm{PLMM_S}$ and $a\mathrm{PLMM_{NS}}$ are close seconds for the steady-state and unsteady-state Navier-Stokes equations, respectively.

Other monolithic preconditioners for Navier-Stokes equations exist, notably the generalized Dryja-Smith-Widlund (GDSW) \cite{heinlein2020reduced, heinlein2025NavStokes}, but place no emphasis on the importance of decomposition (by watershed) and closure BCs. We show, by comparing $a\mathrm{PLMM_{NS}}$ and $a\mathrm{PNM_{NS}}$ for example, that these choices play outsize roles in Krylov convergence. A further demonstration of this fact was made in \cite{khan2025hPLMM} for simpler elliptic equations arising from linear elasticity.

Our second contribution is a coarse-scale solver for the steady-state Navier-Stokes equations based on $g$PLMM. The algorithm yields an approximate solution iteratively, unlike the single-step $\hat{x}_{aprx}\!=\!\mathrm{M_G^{-1}}\hat{b}$ for Stokes flow. Iterations consist of Newton updates on coarse (pressure) unknowns only, with \textit{no} recourse to the fine grid. The coarse solution is mapped onto the fine grid using a prolongation operator, which is not feasible with competing reduced-order models based on PNM \cite{thauvin1998PNM, balhoff2009PNM, veyskarami2018PNM}. Computations are an order of magnitude faster than required for the exact solution.

The paper's outline follows: Section \ref{sec:problem} describes the problem we aim to solve and the fine-grid solver employed. Section \ref{sec:multi_precon} details the formulations of the algebraic preconditioners $a\mathrm{PLMM_{NS}}$ and $a\mathrm{PNM_{NS}}$. Section \ref{sec:mono_geom} formulates the geometric preconditioners $g$PLMM and $g$PNM. The coarse solver for steady-state Navier Stokes is detailed in Section \ref{sec:coarse_solver}. In Section \ref{sec:problem_set}, we describe the set of 2D/3D domains considered for testing the preconditioners and the coarse solver. Results are presented in Sections \ref{sec:test_coarse}--\ref{sec:test_prec} and discussed in Section \ref{sec:discussion}. The paper concludes with takeaways in Section \ref{sec:conclusion}.

\section{Problem and fine-grid solver} \label{sec:problem}

We aim to solve Eq.\ref{eq:governing_eqs} on the complex void space $\Omega$ of a porous sample (white region in Fig.\ref{fig:main_sketch}a), which is often captured by a pore-scale image (e.g., X-ray $\mu$CT). Eq.\ref{eq:governing_eqs} requires an initial condition (IC) $\bs{u}(t = 0)\! =\! \bs{u}_{0}$ and several boundary conditions (BCs). We impose no-slip ($\bs{u}\!=\!\bs{0}$ and $\partial p_n\!=\!0$) on the fluid-solid interface $\Gamma_w$ (white-black boundary in Fig.\ref{fig:main_sketch}a) and inlet/outlet BCs (constant pressure and $\partial_n\bs{u}\!=\!0$) on the left/right external boundaries of $\Omega$, respectively. The latter results in flow from left to right. Here, $\partial_n$ denotes differentiation in the direction normal to the corresponding boundary. At moderate to low Reynolds numbers $Re\!=\!\rho U l/\mu$ (where $U$ and $l$ are characteristic velocity and length), the equations admit a steady-state laminar solution without flow instabilities, allowing one to drop the $\bs{u}_{t}$ term. At even lower $Re$, flow becomes creeping, reducing Eq.\ref{eq:governing_eqs} to the Stokes equations by further dropping $\bs{u}\cdot\nabla \bs{u}$.

We discretize Eq.\ref{eq:governing_eqs} on a Cartesian \textit{fine grid} that conforms to, or is an integer fraction of, the image pixels comprising $\Omega$. We use a marker and cell (MAC) finite difference method, with velocity unknowns defined at cell faces and pressure unknowns at cell centers, as shown in Fig.\ref{fig:unknown_sketch}c. We remark that our proposed preconditioners in later sections are by no means limited to Cartesian grids or the MAC scheme. Other discretizations discussed in Section \ref{sec:intro} (e.g., FVM, MFEM), where velocity and pressure are arranged in a spatially staggered fashion, are equally possible and would benefit from our preconditioners. The resulting discrete system, given by Eq.\ref{eq:Ax=b}, is nonlinear and is solved with Newton's method. At each iteration of Newton, a linearized system of the form $\mathrm{\hat{A}}\hat{x}\!=\!\hat{b}$ is solved.

The blocks $\mathrm{F}^k$, $\mathrm{G}^\top$, $\mathrm{G}$ of $\hat{\mathrm{A}}$ in Eq.\ref{eq:Ax=b} are discrete forms of the $\rho\partial_{t} + \rho\mathcal{L}(\bs{u}\cdot\nabla \bs{u}) - \mu\,\Delta$, gradient ($\nabla$), and divergence ($\nabla\cdot$) operators, respectively. The symbol $\mathcal{L}(\cdot)$ denotes the argument is linearized around the previous iterate, which is why only $\mathrm{F}^k$ has superscript $k$. For Stokes flow, $\mathrm{F}$ is symmetric and the residual $r^k$ is linear, allowing Newton to converge in one step. For Navier-Stokes flow, neither symmetry nor linearity are present. For practical domain sizes of real porous microstructures, $\hat{\mathrm{A}}$ is large and poorly conditioned, caused by the geometric complexity of $\Omega$ and the small fine-grids needed to resolve it. This demands iterative solvers like GMRES, which require preconditioning.

\begin{figure} [t!]
  \centering
  \centerline{\includegraphics[scale=0.52,trim={0 0 0 0},clip]{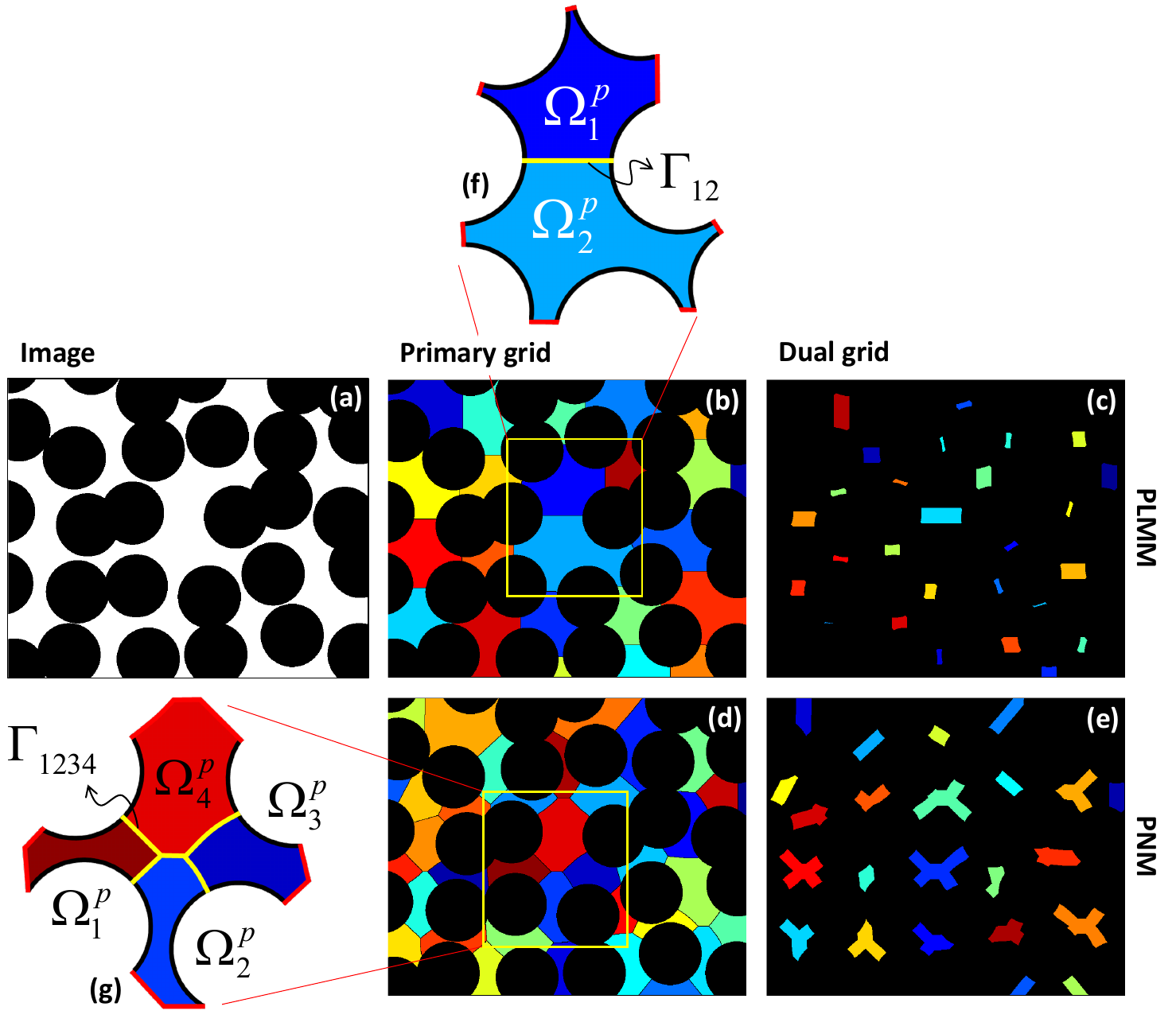}}
  \caption{Schematic of a pore-scale image and its decomposition into primary grids and dual grids in the PLMM and PNM preconditioner. (a) The white regions of the image represent the void space and black regions represent the solid phase. The Navier-Stokes equations are solved in the void space. (b, d) Primary grids (or subdomains) obtained from PLMM (b) and PNM (d) decompositions. (c, e) Corresponding dual grids that overlap the contact interfaces shared between primary grids. (f, g) Each basis function is associated with one contact interface (yellow). In PLMM, each contact interface is shared by only two adjacent primary grids. While in PNM, each contact interface is shared among multiple primary grids.}
\label{fig:main_sketch}
\end{figure}

\begin{figure} [t!]
  \centering
  \centerline{\includegraphics[scale=0.56,trim={0 0 0 0},clip]{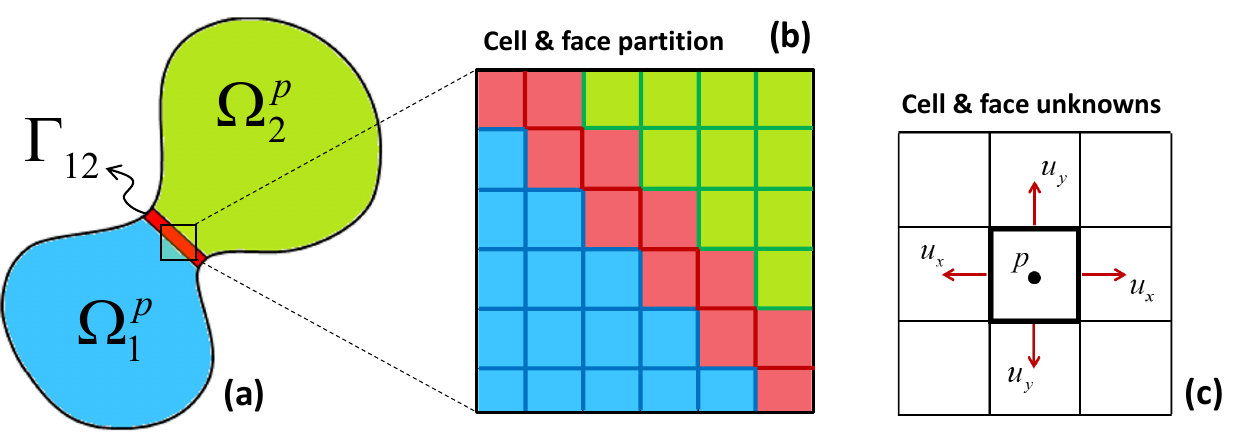}}
  \caption{(a) Two primary grids, $\smash{\Omega^p_1}$ and $\smash{\Omega^p_2}$, separated by a contact interface, $\smash{\Gamma_{12}}$. (b) Partitioning of the cells and faces of the fine grid among the two primary grids and the contact interface is shown. Light/dark blue cells/faces belong to $\smash{\Omega^p_1}$, light/dark green cells/faces belong to $\smash{\Omega^p_2}$, and light/dark red cells/faces belong to $\smash{\Gamma_{12}}$. (c) The cell unknowns correspond to pressure, $p$, and the face unknowns correspond to velocity, $\bs{u}$.}
\label{fig:unknown_sketch}
\end{figure}
\section{Algebraic monolithic preconditioner} \label{sec:multi_precon}

We extend the fully-algebraic monolithic preconditioners $a\mathrm{PLMM_S}$ and $a\mathrm{PNM_S}$ originally designed for the Stokes equations \cite{mehmani2025multiscale} to the Navier-Stokes equations. They are based on the pore-level multiscale method (PLMM) and the the pore network model (PNM). We refer to the extended preconditioners as $a\mathrm{PLMM_{NS}}$ and $a\mathrm{PNM_{NS}}$, where prefix ``$a$'' denotes ``algebraic.'' Both are two-level with a global (coarse) preconditioner $\mathrm{M_G}$ and a local (fine) smoother $\mathrm{M_L}$. Section \ref{sec:mono_struct} outlines the overall structure of the two-level preconditioners, combining $\mathrm{M_G}$ and $\mathrm{M_L}$. Since the mathematical structures of $a\mathrm{PLMM_{NS}}$ and $a\mathrm{PNM_{NS}}$ are identical, we detail $a\mathrm{PLMM_{NS}}$ first in Section \ref{sec:mono_PLMM} then highlight key differences of $a\mathrm{PNM_{NS}}$ in Section \ref{sec:mono_PNM}. In Section \ref{sec:dom_decomp}, we describe the domain decomposition central to $\mathrm{M_G}$ and $\mathrm{M_L}$. Section \ref{sec:mono_PLMM_MG} formulates $\mathrm{M_G}$, and Section \ref{sec:mono_PLMM_ML} reviews the additive-Schwarz smoother $\mathrm{M_L}$ by \cite{mehmani2025multiscale}. Since $a\mathrm{PLMM_{NS}}$ and $a\mathrm{PNM_{NS}}$ reduce to $a\mathrm{PLMM_S}$ and $a\mathrm{PNM_S}$ for Stokes flow, we drop subscripts ``NS'' and ``S'' below.

\subsection{Overall structure} \label{sec:mono_struct}
The monolithic $a$PLMM and $a$PNM preconditioners $\mathrm{M}$ are formulated as:
\begin{subequations} \label{eq:mono_struct}
\begin{align} 
	&\mathrm{M}^{-1} = \mathrm{M_G^{-1}} + \mathrm{M_L^{-1}}(\mathrm{I} - \hat{\mathrm{A}} \mathrm{M_G^{-1}}) \label{eq:mono_struct_a}
	\\
	&\mathrm{M_L^{-1}} = \sum_{i=1}^{n_{st}} \mathrm{M}_l^{-1} (\mathrm{I} - \hat{\mathrm{A}}\mathrm{M}_l^{-1})^{\,i-1} \label{eq:mono_struct_b}
\end{align}
\end{subequations}
where the global preconditioner $\mathrm{M_G}$ and the smoother $\mathrm{M_L}$ are combined multiplicatively, with the former attenuating low-frequency errors and the latter high-frequency errors. The local smoother itself is a multiplicative ($n_{st}$-stage) application of a \textit{base smoother} $\mathrm{M}_l$, which can be any black-box preconditioner (e.g., block Gauss-Seidel) or a custom smoother compatible with $\mathrm{M_G}$. Our numerical tests show that the additive-Schwarz smoother $\mathrm{M_{pd}}$ from \cite{mehmani2025multiscale} (reviewed in Section \ref{sec:mono_PLMM_ML}) provides superior performance. Therefore, we adopt $\mathrm{M}_l = \mathrm{M_{pd}}$ with $n_{st}\!=\!1$ throughout this work.

\subsection{Monolithic PLMM preconditioner}
\label{sec:mono_PLMM}

\subsubsection{Domain decomposition}\label{sec:dom_decomp}

To build $\mathrm{M_G}$ and $\mathrm{M_L}$, the domain $\Omega$ is first decomposed into $N^{p}$ non-overlapping subdomains $\Omega^{p}_{i}$, called \textit{primary grids}. The decomposition uses a modified watershed segmentation algorithm \cite{mehmani2021multiscale} that operates on a binary-image representation of $\Omega$, as shown in Fig.\ref{fig:main_sketch}b for the pore-scale image in Fig.\ref{fig:main_sketch}a. Primary grids appear as randomly colored regions. The interfaces between adjacent primary grids, $\Gamma_{ij}$, are called \textit{contact interfaces} (thick yellow line in Fig.\ref{fig:main_sketch}f). Watershed segmentation ensures that $\Omega^{p}_i$ correspond to local geometric enlargements of $\Omega$ and $\Gamma_{ij}$ to local constrictions. We also construct $N^{d}$ complementary \textit{dual grids} $\Omega^{d}_{i}$ that overlap contact interfaces and cover a thin region around them (see Fig.\ref{fig:main_sketch}c). Dual grids are built by successive morphological dilations of the pixels comprising each contact interface \cite{mehmani2018mult, mehmani2025multiscale}. By construction, $N^{d}$ equals the number of contact interfaces $N^{c}$. Dual grids are used only by $\mathrm{M_L}$ to remove high-frequency errors concentrated near $\Gamma_{ij}$ after each application of $\mathrm{M_G}$. Notice the union of $\Omega^{d}_{i}$ is much smaller than $\Omega$. Unlike the PLMM algorithm of \cite{mehmani2018mult}, where overlapping $\Omega^{d}_{i}$ are merged (sometimes creating large sample-spanning regions with costly local problems), here $\Omega^{d}_{i}$ are allowed to overlap, avoiding this computational drawback. Primary and dual grids are collectively termed \textit{coarse grids} when the distinction is unimportant.

\subsubsection{Global preconditioner} \label{sec:mono_PLMM_MG}
In iterative solvers, preconditioning means solving systems like $\mathrm{\hat{A}}\hat{x}\!=\!\hat{b}$ \textit{approximately} for $\hat{x}$ given $\hat{b}$. The proposed $\mathrm{M_G}$ is presented first as a recipe for computing an approximate solution, involving three operations called \textit{permutation}, \textit{reduction}, and \textit{prolongation}. These operations parallel those of \cite{mehmani2025multiscale} for the Stokes equations, with key differences highlighted below. We then provide a compact representation of $\mathrm{M_G^{-1}}$ as the composition of matrices and operators.
\\

\textbf{Permutation.} The goal here is to reshuffle the rows and columns of $\hat{\mathrm{A}}$ such that the fine-grid entires associated with each primary grid, $\Omega^{p}_{i}$, and each contact interface $\Gamma_{ij}$, are grouped together by applying a permutation matrix $\mathrm{W}$ to the linear system in Eq.\ref{eq:Ax=b}.
Consider the simple domain $\Omega$ depicted in Fig.\ref{fig:unknown_sketch}a, which is decomposed into two primary grids $\Omega^p_1$ and $\Omega^p_2$, sharing the contact interface $\Gamma_{12}$. With reference to Fig.\ref{fig:unknown_sketch}b, fine-grid \textit{cells} comprising $\Omega$ are partitioned into those that belong to: (1) $\Omega^p_1$ (light blue); (2) $\Omega^p_2$ (light green); and (3) $\Gamma_{12}$ (light red). Here, $\Gamma_{12}$ is defined by the thinnest 4-connected collection of pixels in 2D (and 6-connected in 3D) that topologically separate the two primary grids. Other definitions of $\Gamma_{12}$ may be used without loss of generality. 
A similar partitioning of the fine-grid \textit{faces} is performed as follows: (1) faces flanked by at least one cell from $\Omega^p_1$ belong to $\Omega^p_1$ (dark blue); (2) faces flanked by at least one cell from $\Omega^p_2$ belong to $\Omega^p_2$ (dark green); and (3) faces flanked by cells neither in $\Omega^p_1$ nor $\Omega^p_2$ belong to $\Gamma_{12}$ (dark red).
Given this labeling of cells and faces, the matrix $\mathrm{W}$ is constructed as follows:
\begin{subequations} \label{eq:permute}
\begin{equation} \label{eq:permute_act}
\underbrace{\mathrm{W}^\top \hat{\mathrm{A}} \mathrm{W}}_{\mathrm{A}}
\underbrace{\mathrm{W}^\top \hat{x}}_{x} = 
\underbrace{\mathrm{W}^\top \hat{b}}_{b}
\qquad \Rightarrow \qquad
\mathrm{A}x=b
\end{equation}
with the block structure below for $\mathrm{A}$, $b$, and $x$:
\begin{equation} \label{eq:permute_blk}
	\mathrm{A} = 
	\begin{bmatrix}
	\mathrm{A}^p_p & \mathrm{A}^p_c & \mathrm{A}^p_f \\
	\mathrm{A}^c_p & \mathrm{A}^c_c & \mathrm{A}^c_f \\
	\mathrm{A}^f_p & \mathrm{A}^f_c & \mathrm{A}^f_f
	\end{bmatrix}
	\qquad
	b =
	\begin{bmatrix}
	b^p \\
	b^c \\
	b^f
	\end{bmatrix}
	\qquad
	x =
	\begin{bmatrix}
	x^p \\
	x^c \\
	x^f
	\end{bmatrix}
\end{equation}
\end{subequations}
We remark that the matrix $\mathrm{W}$ is unitary and satisfies $\mathrm{W}\mathrm{W}^\top\!=\!\mathrm{I}$. 

The super/subscript $p$ in each block denotes the \textit{cell and face} unknowns/residuals belonging to the primary grids. The super/subscript $c$ and $f$ denote the \textit{cell only} and $\textit{face only}$ unknowns/residuals belonging to the contact interfaces, respectively. Recall that the cell and face unknowns correspond to pressure and velocity, respectively (Fig.\ref{fig:unknown_sketch}c).
\\

\textbf{Reduction.} The goal of reduction is to introduce \textit{closure} or \textit{localization assumptions} that simplify and effectively decouple the permuted system in Eq.\ref{eq:permute}. For the Stokes equations, the steps outlined below enforce the closure BCs $p\!=\!\mathrm{const}$ and $\partial_n\bs{u}\!=\!\bs{0}$ on each contact interface. In \cite{mehmani2018mult}, a local analysis was performed, demonstrating these BCs are nearly optimal for Stokes flow based on arguments that streamlines exhibit a converging-diverging pattern at geometric constrictions ($\Gamma_{12}$ in Fig.\ref{fig:main_sketch}b). For the Navier-Stokes equations, we adopt the same reduction steps but with two caveats. The velocity closure no longer corresponds to $\partial_n\bs{u}\!=\!\bs{0}$ due to the asymmetric upwind discretization of the inertia term in Eq.\ref{eq:navstokes_mom}. Additionally, the suitability of these closures for Navier-Stokes flow is diminished given the more complex streamline patterns near geometric constrictions compared to Stokes flow, as we illustrate later.

To impose $\partial\bs{u}/\partial n\!\approx\!0$ (where ``$\approx$'' means ``inexact'') at contact interfaces, we first approximate $\mathrm{A}$ as follows:
\begin{equation} \label{eq:reduce_1}
	\mathrm{A} \approx 
	\begin{bmatrix}
	\mathrm{\tilde{A}}^p_p & \mathrm{A}^p_c & \mathrm{O} \\
	\mathrm{A}^c_p & \mathrm{\tilde{A}}^c_c & \mathrm{O} \\
	\mathrm{A}^f_p & \mathrm{A}^f_c & \mathrm{A}^f_f
	\end{bmatrix}
	\,
	\qquad
	\qquad
	\begin{matrix}
	\mathrm{\tilde{A}}^p_p = \mathrm{A}^p_p + \mathrm{diag} \left(\mathrm{csum}(\mathrm{A}^p_f)\right) \\
	\mathrm{\tilde{A}}^c_c = \mathrm{A}^c_c + \mathrm{diag} \left(\mathrm{csum}(\mathrm{A}^c_f)\right)
	\end{matrix}
\end{equation}
where the operation $\mathrm{csum(\cdot)}$ sums all columns of its argument matrix to yield a single column vector. The operation $\mathrm{diag(\cdot)}$ takes an input vector and creates a diagonal matrix whose diagonal entries coincide with those of the vector.
Notice by imposing $\partial_n\bs{u}\!\approx\!\bs{0}$ in Eq.\ref{eq:reduce_1}, the face-only unknowns $x^{f}$ defined at contact interfaces are now decoupled from the system. Therefore, we may henceforth focus on solving the smaller system $\tilde{\mathrm{A}}\tilde{x}\!=\!\tilde{b}$ with:
\begin{equation} \label{eq:Atilde}
	\mathrm{\tilde{A}} =
	\begin{bmatrix}
	\mathrm{\tilde{A}}^p_p & \mathrm{A}^p_c \\
	\mathrm{A}^c_p & \mathrm{\tilde{A}}^c_c
	\end{bmatrix}
	\qquad
	\tilde{b} =
	\begin{bmatrix}
	 b^p \\
	 b^c
	\end{bmatrix}
	\qquad
	\tilde{x} =
	\begin{bmatrix}
	 x^p \\
	 x^c
	\end{bmatrix}
\end{equation}
after which $x^f$ can be obtained by back substitution (this is fast because $\smash{\mathrm{A}^f_f}$ is a very small):
\begin{equation} \label{eq:xf}
x^f = \left(\mathrm{A}^f_f \right)^{-1}\left(b^f - \mathrm{A}^f_p x^p - \mathrm{A}^f_c x^c \right)
\end{equation}

Focusing next on Eq.\ref{eq:Atilde}, each block consists of the following sub-blocks:
\begin{subequations} \label{eq:Atilde_blk}
\begin{equation}
	\mathrm{\tilde{A}}^p_p =
	\begin{bmatrix}
	\mathrm{\tilde{A}}^{p_1}_{p_1} & \cdots & \mathrm{\tilde{A}}^{p_1}_{p_n} \\
	\vdots & \ddots & \vdots \\	
	\mathrm{\tilde{A}}^{p_n}_{p_1} & \cdots & \mathrm{\tilde{A}}^{p_n}_{p_n}
	\end{bmatrix}
	\quad
	\mathrm{\tilde{A}}^c_c =
	\begin{bmatrix}
	\mathrm{\tilde{A}}^{c_1}_{c_1} & \cdots & \mathrm{\tilde{A}}^{c_1}_{c_m} \\
	\vdots & \ddots & \vdots \\	
	\mathrm{\tilde{A}}^{c_m}_{c_1} & \cdots & \mathrm{\tilde{A}}^{c_m}_{c_m}
	\end{bmatrix}
	\quad
	\mathrm{A}^p_c =
	\begin{bmatrix}
	\mathrm{A}^{p_1}_{c_1} & \cdots & \mathrm{A}^{p_1}_{c_m} \\
	\vdots & \ddots & \vdots \\	
	\mathrm{A}^{p_n}_{c_1} & \cdots & \mathrm{A}^{p_n}_{c_m}
	\end{bmatrix}
	\quad
	\mathrm{A}^c_p =
	\begin{bmatrix}
	\mathrm{A}^{c_1}_{p_1} & \cdots & \mathrm{A}^{c_1}_{p_n} \\
	\vdots & \ddots & \vdots \\	
	\mathrm{A}^{c_m}_{p_1} & \cdots & \mathrm{A}^{c_m}_{p_n}
	\end{bmatrix}
\end{equation}
\begin{equation} \label{eq:Atilde_blk_bx}
	b^p =
	\begin{bmatrix}
	 b^{p_1} \\
	 \vdots \\
	 b^{p_n}
	\end{bmatrix}
	\qquad
	b^c =
	\begin{bmatrix}
	 b^{c_1} \\
	 \vdots \\
	 b^{c_m}
	\end{bmatrix}
	\qquad
	x^p =
	\begin{bmatrix}
	 x^{p_1} \\
	 \vdots \\
	 x^{p_n}
	\end{bmatrix}
	\qquad
	x^c =
	\begin{bmatrix}
	 x^{c_1} \\
	 \vdots \\
	 x^{c_m}
	\end{bmatrix}
	\qquad
\end{equation}
\end{subequations}
where $n\!=\!N^p$ and $m\!=\!N^c$, with $N^p$ and $N^c$ denoting the total number of primary grids and contact interfaces in $\Omega$. Each sub-block corresponds to entries associated with either a primary grid, a contact interface, or a coupling between the two. To complete the imposition of $\partial_n\bs{u}\!\approx\!\bs{0}$, we further simplify the $\tilde{\mathrm{A}}^p_p$ sub-block as follows:
\begin{equation} \label{eq:reduce_2}
	\mathrm{\tilde{A}}^p_p \approx
	\mathrm{\bar{A}}^p_p =
	\begin{bmatrix}
	\mathrm{\bar{A}}^{p_1}_{p_1} &  & \mathrm{O} \\
	 & \ddots & \\	
	\mathrm{O} &  & \mathrm{\bar{A}}^{p_n}_{p_n}
	\end{bmatrix}
	\qquad
	\qquad
	\bar{\mathrm{A}}^{p_i}_{p_i} = \tilde{\mathrm{A}}^{p_i}_{p_i} +
	     \mathrm{diag}\left( \sum_{\forall j\neq i} 
	     \mathrm{csum}(\mathrm{\tilde{A}}^{p_i}_{p_j}) \right)
\end{equation}
where $\bar{\mathrm{A}}^{p}_{p}$ is block-diagonal matrix with the entries of each primary grid fully decoupled from those of other primary grids.
Similar to Eq.\ref{eq:reduce_1}, the off-diagonal blocks of each block-row in $\mathrm{\tilde{A}}^p_p$ are column-summed and added to its main diagonal entries.
We are now left with solving the simpler system $\bar{\mathrm{A}}\tilde{x}\!=\!\tilde{b}$, with:
\begin{equation} \label{eq:Abar}
	\mathrm{\bar{A}} = 
	\begin{bmatrix}
	   \mathrm{\bar{A}}^p_p & \mathrm{A}^p_c \\
	   \mathrm{A}^c_p & \mathrm{\tilde{A}}^c_c
	\end{bmatrix}
\end{equation}

To enforce $p\!=\!\mathrm{const}$ at contact interfaces, we introduce a \textit{reduction matrix}, $\mathrm{Q}$:
\begin{subequations} \label{eq:reduce_3}
\begin{equation} \label{eq:reduced_defs}
	\mathrm{Q} =
	\begin{bmatrix}
		\mathrm{I}_{N^p_f\times N^p_f} & \mathrm{O} \\
		\mathrm{O} & \mathrm{Q^o}
	\end{bmatrix}
	\qquad
	\mathrm{Q^o} =
	\begin{bmatrix}
		\bs{1}^{c_1} & & \mathrm{O} \\
		& \ddots & \\
		\mathrm{O} & & \bs{1}^{c_m}
	\end{bmatrix}_{N^c_f\times N^c}
	\;
	\bs{1}^{c_i} =
	\begin{bmatrix}
		1 \\
		\vdots \\
		1
	\end{bmatrix}_{N^{c_i}_f\times 1}
\end{equation}
such that when applied, yields the following \textit{reduced system}:
\begin{equation} \label{eq:reduced_sys}
	\mathrm{\bar{A}}\tilde{x} = \tilde{b} \;,\quad \tilde{x}\simeq \mathrm{Q}x_M 
	\quad \Rightarrow \quad 
	\underbrace{\mathrm{Q}^\top \mathrm{\bar{A}}\mathrm{Q}}_{\mathrm{A_M}} x_M
	     = \underbrace{\mathrm{Q}^\top \tilde{b}}_{b_M}
	\quad \Rightarrow \quad
	\mathrm{A_M} x_M  = b_M
\end{equation}
\end{subequations}
In Eq.\ref{eq:reduced_defs}, $\smash{N^p_f}$, $\smash{N^c_f}$, and $\smash{N^{c_i}_f}$ denote the number of fine-scale unknowns belonging to all primary grids, all contact interfaces, and the contact interface labeled $c_i$, respectively. We remark that $\smash{N^p_f}$ includes both velocity and pressure unknowns, whereas $\smash{N^c_f}$ and $\smash{N^{c_i}_f}$ include only pressure unknowns. 
The reduction matrix $\mathrm{Q}$ is block-diagonal and is comprised of only $0$ and $1$ entries. Left-multiplying (or right-multiplying) a matrix by $\smash{\mathrm{Q}^\top}$ (or $\mathrm{Q}$) has the effect of performing a row-sum (or column-sum) of all fine-scale entries associated with each contact interface separately. As the entries belonging to the interface $c_i$ in $\bar{\mathrm{A}}$ correspond to only pressure unknowns, the closure $p\!=\!\mathrm{const}$ over $c_i$ is imposed by the column-sum. The row-sum ensures overall (``weak'') mass conservation across each interface.

The reduced system has now the following block structure:
\begin{equation} \label{eq:reduced_blk}
	\mathrm{A_M} =
	\begin{bmatrix}
		\bar{\mathrm{A}}_p^p & \bar{\mathrm{A}}_c^p \\
		\bar{\mathrm{A}}_p^c & \bar{\mathrm{A}}_c^c
	\end{bmatrix}
	\qquad
	b_M =
	\begin{bmatrix}
	b^p \\
	\bar{b}^c
	\end{bmatrix}
	\qquad
	x_M =
	\begin{bmatrix}
	x^p \\
	x^o
	\end{bmatrix}
\end{equation}
where $x^o$ (with length $N^c\!\times\!1$) contains \textit{coarse-scale} pressure unknowns defined at contact interfaces (one per interface). 
\\

\textbf{Prolongation.} The goal here is to build a \textit{prolongation matrix}, $\mathrm{P}$, and a \textit{restriction matrix}, $\mathrm{R}$, such that a coarse-scale problem can be formulated to solve for the interface pressures $x^o$. This is done as follows:
\begin{equation} \label{eq:coarse_prob}
\mathrm{A_M} x_M = b_M, \quad  x_M = \mathrm{P} 
	\begin{bmatrix}
	x^o \\
	y^o
	\end{bmatrix}
	\quad \Rightarrow \quad 
	\underbrace{\mathrm{R}\mathrm{A_M}\mathrm{P}}_{\mathrm{A^o}}
	\begin{bmatrix}
	x^o \\
	y^o
	\end{bmatrix}
	= \underbrace{\mathrm{R} b_M}_{b^o}
\end{equation}
where $y^o$ are auxiliary coarse-scale unknowns associated with the primary grids.

The prolongation matrix $\mathrm{P}$ has the following block structure:
\begin{subequations} \label{eq:prolong}
\begin{equation} \label{eq:prolong_1}
	\mathrm{P} =
 	\begin{bmatrix}
		\mathrm{B} & \mathrm{C} \\
		\mathrm{I} & \mathrm{O}
	\end{bmatrix}
	\qquad\qquad
	\mathrm{B} =
	\begin{bmatrix}
		p_{c_1}^{p_1} & p_{c_2}^{p_1} & \cdots & p_{c_m}^{p_1} \\
		p_{c_1}^{p_2} & p_{c_2}^{p_2} & \cdots & p_{c_m}^{p_2} \\
		\vdots & \vdots & \ddots & \vdots \\
		p_{c_1}^{p_n} & p_{c_2}^{p_n} & \cdots & p_{c_m}^{p_n}
	\end{bmatrix}_{N_f^p \times N^c}
	\qquad 
	\mathrm{C} = 
	\begin{bmatrix}
		c^{p_1} & & & \mathrm{O}\\
		 & c^{p_2} & &\\
		 & & \ddots &\\
		\mathrm{O} & & & c^{p_n}
	\end{bmatrix}_{N^p_f \times N^p}
\end{equation}
where $p^{p_i}_{c_j}$ and $c^{p_i}$ are called \textit{shape} and \textit{correction} vectors, respectively. 

The shape vectors are defined as follows:
\begin{align} 
	p^{p_i}_{c_j} &= 
	\begin{cases}
		- \left(\bar{\mathrm{A}}^{p_i}_{p_i}\right)^{-1}  R^{p_i}_p \bar{\mathrm{A}}^p_c e^{c_j}_c\,, \qquad c_j \in C^{p_i}
	        \\
    		\mathrm{O}\,,  \hspace{3.05cm} c_j \notin C^{p_i}
	\end{cases} \label{eq:prolong_3}
	\\
	e_c^{c_j} &=
	\begin{bmatrix}
		\delta_{1j}, \delta_{2j}, \cdots, \delta_{mj}
	\end{bmatrix}^\top_{N^c \times 1} \label{eq:prolong_4}
\end{align}
The index set $C^{p_i}$ contains all contact interfaces that intersect the boundary of the primary grid $p_i$. The Kronecker delta $\delta_{ab}$ is used to define the unit vector $e^{c_j}_c$, with 1 occupying its $j^{th}$ entry and 0 elsewhere. The matrix $R^{p_i}_p$ is called a \textit{contraction matrix} that restricts a fine-scale $\smash{N^p_f}\!\times 1$ vector defined over the union of all primary grids (excluding the contact interfaces) to a fine-scale $\smash{N^{p_i}_f}\!\times 1$ vector defined over the primary grid $p_i$ only. More precisely:
\begin{equation} \label{eq:shape_func}
	\mathrm{R}_p^{p_i} =
	\begin{bmatrix}
		\Delta^{p_i}_{p_1}, \Delta^{p_i}_{p_2}, \cdots, \Delta^{p_i}_{p_n}
	\end{bmatrix}_{N^{p_i}_f\times N^p_f}
	\qquad\qquad
	\Delta^{p_i}_{p_j} =
	\begin{cases}
		\mathrm{I}_{N^{p_i}_f \times N^{p_i}_f} \quad \text{\; if }\; i = j \\
		\mathrm{O}_{N^{p_i}_f \times N^{p_j}_f} \quad \text{if }\; i \neq j
	\end{cases}
\end{equation}
Similar to $\smash{N^p_f}$ defined earlier, $\smash{N^{p_i}_f}$ is the number of fine-scale unknowns that belong to primary grid $p_i$. From Eq.\ref{eq:prolong_3}, it is obvious that the matrix $\mathrm{B}$ is very sparse, because only two shape vectors per column contain non-zero entries. These two shape vectors belong to the two primary grids sandwiching the contact interface associated with each column (Fig.\ref{fig:main_sketch}f). As we shall discuss in Section \ref{sec:mono_PNM}, this does not hold for the $a$PNM preconditioner (Fig.\ref{fig:main_sketch}g). 

Returning back to Eq.\ref{eq:prolong_1}, the correction vectors are defined as follows:
\begin{equation}
	c^{p_i} = \left(\bar{\mathrm{A}}^{p_i}_{p_i}\right)^{-1}b_{B}^{p_i} \label{eq:prolong_2}
\end{equation}
where $b_{B}^{p_i}$ is a right-hand side (RHS) vector defined on primary grid $p_i$ that accounts for any non-homogeneous BCs and source terms ($\bs{b}\!\neq\!0$) in Eq.\ref{eq:governing_eqs}. This does \textit{not} equal $b^{p_i}$ in Eq.\ref{eq:Atilde_blk_bx}, which also includes non-zero entries associated with the PDE residual in the interior of $\Omega^p_i$, specifically $b^{p_i} = b_{B}^{p_i} + b_{I}^{p_i}$. To compute $b_{B}^{p_i}$, we evaluate the residual of the steady-state form of Eq.\ref{eq:governing_eqs} (dropping the time derivative, or equivalently using zero velocity at the previous time step) at the origin (i.e., $r(\bs{u} = \bs{0},p = 0)$), then subject the result to the same permutation in Eq.\ref{eq:permute}. We remark that for the Stokes equations, where Eq.\ref{eq:governing_eqs} is linear, the discrete system $\mathrm{\hat{A}}\hat{x}\!=\!\hat{b}$ yields a RHS vector $\hat{b}$ that contains only the effect of non-homogeneous BCs, resulting in $\smash{b^{p_i}\!=\!b_B^{p_i}}$. Hence, no distinction between these two vectors was made in \cite{mehmani2025multiscale}. Here, we use $\smash{b_B^{pi}}$ instead of $\smash{b^{pi}}$ in Eq.\ref{eq:prolong_2} because extensive numerical tests revealed the latter leads to poor GMRES convergence by introducing high-frequency errors in $c^{p_i}$ over primary grids.
\end{subequations}

Having completed the definition of $\mathrm{P}$, we now define the restriction matrix in Eq.\ref{eq:coarse_prob} as:
\begin{align}
\mathrm{R} = 
    \begin{bmatrix}
    \mathrm{O} & \mathrm{I} \\
    \Pi(\mathrm{C}^\top) & \mathrm{O}
    \end{bmatrix}  \label{eq:restrict_FVM}
\end{align}
where $\mathrm{I}$ is the $N^c\!\times N^c$ identity matrix, and $\Pi(\mathrm{C}^\top)$ is a matrix with 0 and 1 entries that reflects the sparsity pattern of $\mathrm{C}^\top$ in Eq.\ref{eq:prolong_1} (i.e., entries are 1 where $\mathrm{C}^\top$ is non-zero, and 0 elsewhere).
Left-multiplying $\mathrm{A_M}$ by $\mathrm{R}$ performs a row-sum over each primary grid while leaving the rows associated with contact interfaces intact. Combined with the earlier row-sum performed through the application of $\mathrm{Q}$ in Eq.\ref{eq:reduced_sys} over entries of each contact interface, mass conservation is imposed in a weak (integrated) sense within each primary grid and across each contact interface. Note that $\mathrm{R}\!\neq\!\mathrm{P}^\top$. The choice of the restriction matrix in Eq.\ref{eq:restrict_FVM} is motivated by the fact that a Galerkin restriction ($\mathrm{R}\!=\!\mathrm{P}^\top$) leads to poor solver performance for (Navier-)Stokes equations discretized with MAC finite differences or similar schemes \cite{mehmani2025multiscale}.

We note that constructing $\mathrm{R}$ is trivial, and constructing $\mathrm{P}$ is embarrassingly parallelizable. The latter follows because computing shape and correction vectors in Eqs.\ref{eq:prolong_3} and \ref{eq:prolong_2} involves solving fully decoupled local systems on primary grids. Since only RHS vectors differ between systems, local LU-decomposition can accelerate computations.
\\

\textbf{Summary.} The above steps constitute a recipe for computing the action of the global preconditioner, $\mathrm{M_G^{-1}}$, on a RHS vector $\hat{b}$. For notational simplicity, and ease of reference later, we express this action as:
\begin{equation} \label{eq:iMG}
		\mathrm{M_G^{-1}} = \hat{\mathrm{P}}\, (\hat{\mathrm{R}} \check{\mathrm{A}} \hat{\mathrm{P}})^{-1} \hat{\mathrm{R}}
\end{equation}
	where $\hat{\mathrm{P}}$ and $\hat{\mathrm{R}}$ are \textit{effective} prolongation and restriction matrices, respectively, and $\check{\mathrm{A}}$ a transformation of $\hat{\mathrm{A}}$:
\begin{subequations} \label{eq:iMG_PR}
	\begin{align}
		&\hat{\mathrm{P}} = \mathrm{W} \mathrm{Q} \mathrm{P} \label{eq:iMG_PR_a}\\
		&\hat{\mathrm{R}} = \mathrm{R} \mathrm{Q}^\top \mathrm{W}^\top \label{eq:iMG_PR_b}\\
		&\check{\mathrm{A}} = \mathrm{W}\mathcal{C}(\mathrm{W}^\top
		                      \hat{\mathrm{A}}\mathrm{W})\mathrm{W}^\top
		                      \label{eq:iMG_PR_c}
	\end{align}
\end{subequations}
Recall $\mathrm{W}$, $\mathrm{Q}$, $\mathrm{P}$, and $\mathrm{R}$ are the permutation (Eq.\ref{eq:permute_act}), reduction (Eq.\ref{eq:reduced_defs}), prolongation (Eq.\ref{eq:prolong_1}), and restriction (Eq.\ref{eq:restrict_FVM}) matrices introduced earlier. The operator $\mathcal{C}(\cdot)$ is \textit{linear} and encapsulates the column-sums performed in Eqs.\ref{eq:reduce_1} and \ref{eq:reduce_2} to impose the velocity closure BC. The $\mathrm{csum(\cdot)}$ operation in Eqs.\ref{eq:reduce_1} and \ref{eq:reduce_2} does not admit a simple matrix representation. Left-multiplying a coarse-grid vector, like $[x^o, y^o]^\top$ in Eq.\ref{eq:coarse_prob}, by $\hat{\mathrm{P}}$ maps it onto the fine grid. Left-multiplication by $\hat{\mathrm{R}}$ maps in the inverse direction. The above representation of $\mathrm{M_G}$ makes one of its key properties immediately clear:

\begin{proposition} \label{prp:project}
The matrix $M_G^{-1}\hat{A}$ is not a projection, where $M_G^{-1}$ is given by Eq.\ref{eq:iMG} for the saddle-point Stokes or Navier-Stokes equations. This lies in contrast to $M_G^{-1}\!=\!\hat{P}(\hat{R}\hat{A}\hat{P})^{-1}\hat{R}$ for elliptic equations, where $M_G^{-1}\hat{A}$ is a projection.
\end{proposition}
\begin{proof}
A simple substitution of Eqs.\ref{eq:iMG}-\ref{eq:iMG_PR} into $(\mathrm{M_G^{-1}\hat{A}})^2$ shows that the product does not equal $\mathrm{M_G^{-1}\hat{A}}$, unless $\mathcal{C}(\cdot)$ is the identity matrix. For elliptic equations, $\mathcal{C}(\cdot)$ is identity \cite{li2024machine}, but not for saddle-point equations studied herein.
\end{proof}

\noindent\textbf{Remark 1.} We found empirically that $\mathrm{M_G^{-1}}\!=\!\mathrm{\hat{P}(\hat{
R}\hat{A}\hat{P})^{-1}\hat{R}}$ performs almost identically to Eq.\ref{eq:iMG} (notice the difference in using $\mathrm{\hat{A}}$ versus $\mathrm{\check{A}}$). In other words, the column-sum operator $\mathcal{C}(\cdot)$ in Eq.\ref{eq:iMG_PR_c} can be replaced by the identity matrix. We adopt this formulation to produce all results in Section \ref{sec:test_prec}, because it saves the extra column-sum operation.

\subsubsection{Local smoother} \label{sec:mono_PLMM_ML}
Application of the global preconditioner $\mathrm{M_G}$ leaves behind high-frequency errors concentrated near contact interfaces \cite{mehmani2025multiscale}. Therefore, a compatible local smoother $\mathrm{M_L}$ is needed to attenuate these errors. We adopt the \textit{dual-primary} smoother\footnote{A term we coin here, mirroring the \textit{contact-grain} smoother defined in \cite{li2023smooth} for linear elasticity.} $\mathrm{M_L}\!=\!\mathrm{M_{dp}}$ of \cite{mehmani2025multiscale}, which has exhibited superior performance over other black-box alternatives. The dual-primary smoother combines two additive-Schwarz (block-Jacobi) preconditioners multiplicatively:
\begin{equation} \label{eq:Mdp}
\mathrm{M_{dp}^{-1}} = \mathrm{M_d^{-1}} + 
                     \mathrm{M_p^{-1}}
                     \left(\mathrm{I} - \mathrm{\hat{A}} \mathrm{M_d^{-1}} \right)
\end{equation}
where $\mathrm{M_d^{-1}}$ is applied first to eliminate errors within dual grids, then $\mathrm{M_p^{-1}}$ is applied to remove errors within primary grids. Both $\mathrm{M_d}$ and $\mathrm{M_p}$ have standard additive-Schwarz formulations \cite{saad2003book}:
\begin{equation} \label{eq:Mp_Md}
	\mathrm{M_p^{-1}}
	= \sum_{i=1}^{N^p}
	\mathrm{E}^{p_i}_f \,
	(\underbrace{\mathrm{R}^{p_i}_f \mathrm{\hat{A}}\, \mathrm{E}^{p_i}_f}
	      _{\hat{\mathrm{A}}_{p_i}})^{-1} \,
	\mathrm{R}^{p_i}_f
	\qquad
	\mathrm{M_d^{-1}}
	= \sum_{i=1}^{N^c} 
	\mathrm{E}^{d_i}_f  \,
	(\underbrace{\mathrm{R}^{d_i}_f \mathrm{\hat{A}}\, \mathrm{E}^{d_i}_f}
	      _{\hat{\mathrm{A}}_{d_i}})^{-1} \,
	\mathrm{R}^{d_i}_f\
\end{equation}
The matrices $\smash{\mathrm{R}^{p_i}_f}$ and $\smash{\mathrm{E}^{p_i}_f}$ restrict and extend vectors between $\Omega^{p}_{i}$ and $\Omega$, respectively, while $\smash{\mathrm{R}^{d_i}_f}$ and $\smash{\mathrm{E}^{d_i}_f}$ operate between $\Omega^{d}_{i}$ and $\Omega$. These restriction and extension matrices are transposes of each other and contain only 0 and 1 entries. For details, see \cite{mehmani2025multiscale}. Applying $\mathrm{M_d}$ requires solving $N^d$ ($=\! N^c$) decoupled local problems on dual grids, and applying $\mathrm{M_p}$ requires solving $N^p$ decoupled problems on primary grids. Both operations are embarrassingly parallelizable.

\subsection{Monolithic PNM preconditioner} \label{sec:mono_PNM}
The overall structures of $\mathrm{M_G}$ and $\mathrm{M_L}$ for the monolithic $a$PNM preconditioner are identical to those of the monolithic $a$PLMM. The only difference lies in the definitions of ``primary grid'' and ``dual grid,'' which is reflected in the permutation matrix $\mathrm{W}$ in Eq.\ref{eq:permute}. In $a$PNM, the domain $\Omega$ is decomposed into primary grids $\Omega^p_i$ in a different way, still utilizing watershed segmentation. Unlike $a$PLMM, primary grids in $a$PNM correspond to local geometric constrictions of $\Omega$, and contact interfaces to geometric enlargements. In pore-network modeling parlance, these are the \textit{throats} and \textit{pores} of the void space, respectively. Details of how this decomposition is performed can be found in \cite{mehmani2025multiscale}, but briefly, it involves providing interface pixels identified by $a$PLMM as ``seed points'' to the watershed algorithm. The result is a decomposition like Fig.\ref{fig:main_sketch}d, where contact interfaces are shared by more than two primary grids and are located within geometric enlargements. For example, the contact interface in Fig.\ref{fig:main_sketch}g (yellow line denoted by $\Gamma_{1234}$) is shared among four primary grids: $\Omega^p_1$, $\Omega^p_2$, $\Omega^p_3$, and $\Omega^p_4$. Dual grids are constructed in exactly the same way as before, namely, by morphologically dilating the pixels comprising each contact interface as illustrated by Fig.\ref{fig:main_sketch}e. Unlike $a$PLMM, each column of the matrix $\mathrm{B}$ in Eq.\ref{eq:prolong_1} for $a$PNM can have more than two non-zero shape vectors $\smash{p^{p_i}_{c_j}}$.

\section{Geometric monolithic preconditioner}\label{sec:mono_geom}
As noted in Section \ref{sec:mono_PLMM_MG}, the closure BC $\partial_n\bs{u}\!=\!\bs{0}$ is not imposed exactly at contact interfaces during the reduction step of $\mathrm{M_G}$, except for the Stokes equations. The deviation grows further with increasing Reynolds number, $Re$, resulting in large errors near interfaces \cite{mehmani2025multiscale} after each application of $\mathrm{M_G}$. These errors then convect downstream due to the upwind scheme used for the inertia term in the Navier-Stokes equations. Since the inertia and viscous terms cannot be separated algebraically in the discretized system, exact closure cannot be imposed. In Stokes flow, inertia is absent, so no such separation is needed. In Navier-Stokes flow, exact closure demands geometric information, i.e., access to the code used to assemble the Jacobian matrix in Eq.\ref{eq:Ax=b}. This would allow defining a modified residual $\tilde{r}_w^k$ (instead of $r^k$ in Eq.\ref{eq:Ax=b}) whose Jacobian $\tilde{J}$ is used (in lieu of $J^k$) to build $\mathrm{M_G}$. The modified residual we propose is:
\begin{equation} \label{eq:mod_res}
	\tilde{r}_w^k
	\,
	=
	\,
	\begin{bmatrix}
	\rho\,(\partial_t\bs{u}_k + 
	\bs{w} \cdot \nabla \bs{u}_{k}) - \mu \Delta \bs{u}_{k} + \nabla p_{k} - \rho \bs{b} \\
	\nabla \cdot \bs{u}_{k}
	\end{bmatrix}
\end{equation}
where the inertia term $\bs{u} \cdot \nabla \bs{u}$ is linearized around a given velocity $\bs{w}$, called ``wind.'' To compute $\bs{u}$ (and $p$) at some $Re$, we set $\bs{w}$ to the velocity at a nearby Reynolds number. The closer the $Re$ corresponding to $\bs{w}$ is to the $Re$ at which the solution is sought, the more accurate $\mathrm{M_G}$ will be. If no adjacent $\bs{w}$ is available, the Stokes velocity serves as fallback.

Aside from the linearization in Eq.\ref{eq:mod_res}, we also remove the inertia term entirely from the modified residual of those fine-grid faces that lie within a few layers of each contact interface. Fig.\ref{fig:coarse_sketch} illustrates such faces by tagging them with black dots. The layer-thickness of tagged faces depends on the size of the upwind stencil used for the inertia term (here, three-point). Removing the inertia term ensures that the reduction steps used to build $\mathrm{M_G}$ in Section \ref{sec:mono_PLMM_MG} impose $\partial_n\bs{u}\!=\!\bs{0}$ exactly at interfaces. 
This is highlighted by the Navier-Stokes basis functions depicted in Fig.\ref{fig:basis_sol}, derived from the Jacobian of Eq.\ref{eq:mod_res}. Each basis is a column of the matrix $\mathrm{B}$ in Eq.\ref{eq:prolong_1}, comprised of two (in PLMM) or four (in PNM) shape vectors associated with the yellow-highlighted interfaces. They belong to primary grids in Figs.\ref{fig:main_sketch}f-g and correspond to a high $Re$. For comparison, Fig.\ref{fig:basis_sol} also shows Stokes bases. Both satisfy $\partial_n\bs{u}\!=\!\bs{0}$ at interfaces (yellow).

To summarize, the resulting geometric PLMM and PNM preconditioners, denoted hereafter by $g$PLMM and $g$PNM, are constructed from $\mathrm{\hat{A}}$ set equal to the modified Jacobian $\tilde{J}_w$ of Eq.\ref{eq:mod_res} (where inertia is disabled for near-interface faces). The RHS vector $\hat{b}$ is computed as before from the original (not modified) residual in Eq.\ref{eq:Ax=b} (i.e., $-r^k$). We show in Section \ref{sec:test_prec} that $g$PLMM and $g$PNM perform better than $a$PLMM and $a$PNM. Notice $\tilde{J}_w$ does not depend on the Newton iteration index $k$, meaning $\mathrm{M_G}$ is build once and reused thereafter. While $J^k$, used to build $a$PLMM and $a$PNM, does depend on $k$, we found empirically that building $\mathrm{M_G}$ once and reusing it for all $k$ is more efficient.

\begin{figure} [t!]
  \centering
  \centerline{\includegraphics[scale=0.56,trim={0 0 0 0},clip]{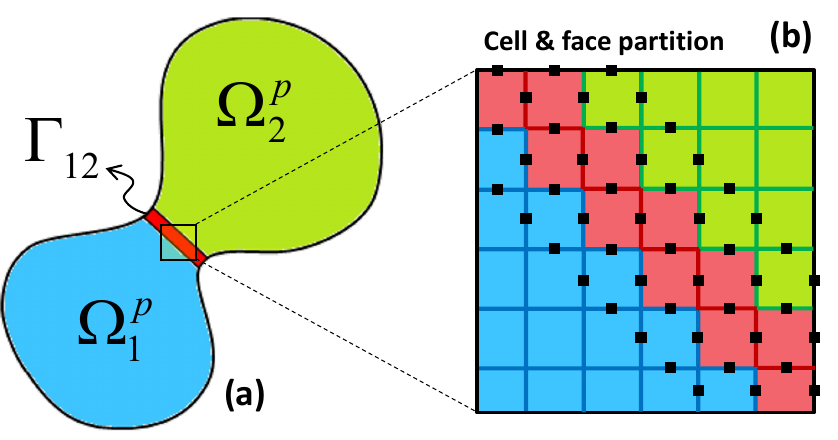}}
  \caption{The inertia ($\bs{u}\cdot\nabla\bs{u}$) term is removed from the modified residual in Eq.\ref{eq:mod_res} of the fine-grid faces that lie within a few layers of each contact interface ($\smash{\Gamma_{12}}$). These faces are tagged with black dots. The operation is required to impose the closure BC $\partial_n\bs{u}\!=\!\bs{0}$ exactly at contact interfaces.}
\label{fig:coarse_sketch}
\end{figure}

\begin{figure} [h!]
  \centering
  \centerline{\includegraphics[scale=0.55,trim={0 0 0 0},clip]{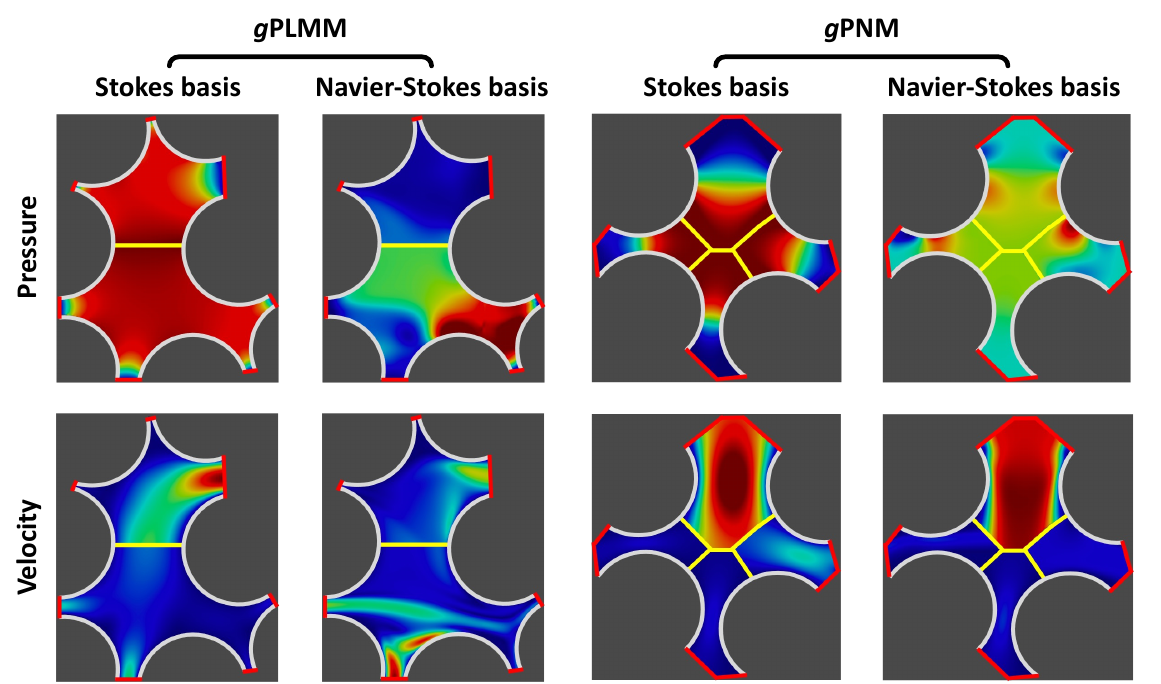}}
  \caption{Basis functions (pressure and velocity magnitude) corresponding to the $\mathrm{M_G}$ of $g$PLMM and $g$PNM. Each basis is a column of the matrix $\mathrm{B}$ in Eq.\ref{eq:prolong_1}, comprised of two (in $g$PLMM) or four (in $g$PNM) shape vectors associated with the yellow-highlighted interface. Navier-Stokes bases are built from the Jacobian of the modified residual in Eq.\ref{eq:mod_res} at $Re\!=\!1.6\times10^4$. The bases belong to primary grids highlighted in Figs.\ref{fig:main_sketch}f-g, with BCs $p\!=\!0$ and $\partial\bs{u}/\partial n\!=\!0$ on all red lines, and $p\!=\!1$ and $\partial\bs{u}/\partial n\!=\!0$ on yellow lines. No-slip BCs hold on the void-solid interface (gray lines).}
\label{fig:basis_sol}
\end{figure}

\section{Coarse-scale Navier-Stokes solver} \label{sec:coarse_solver}
Here, we demonstrate that the global preconditioner $\mathrm{M_G}$ for $g$PLMM and $g$PNM can also be used as standalone coarse-scale solvers to compute approximate solutions to the \textit{steady-state} Navier-Stokes equations (i.e., pre-turbulent regime). This approach results in significant computational speedup over exact solvers while maintaining acceptable accuracy in many applications, provided $\mathrm{M_G}$ corresponds to $g$PLMM (not $g$PNM). Extension to dynamic (i.e., turbulent) flow is beyond our scope and left as future work. For Stokes flow, governed by the linear system $\mathrm{\hat{A}}\hat{x}\!=\!\hat{b}$, we showed in \cite{mehmani2025multiscale} that an accurate approximate solution can be obtained from a single application of $\mathrm{M_G^{-1}}$ in Eq.\ref{eq:iMG}:
\begin{equation} \label{eq:approx_formula}
	\hat{x}_{aprx} = \mathrm{M_G^{-1}} \hat{b}
\end{equation}

However, a single application is insufficient for the steady-state Navier-Stokes equations due to its nonlinearity, requiring an iterative approach. To achieve speedup, we devise a Newton method that performs iterations only on the coarse-scale (pressure) unknowns, without recourse to the fine grid or any fine-scale problems (local or global):
\begin{equation} \label{eq:coarse_scale_sys}
	J_{c}^k \delta_{c}^{k} = - r_{c}^{k}
\end{equation}
The coarse-scale residual and Jacobian are defined as follows:
\begin{subequations} \label{eq:coarse_rJ}
\begin{align}
	r_{c}^{k} &= \hat{\mathrm{R}} \tilde{r}^k \label{eq:coarse_scale_def1} \\
	J_{c}^k &= \hat{\mathrm{R}} 
	           \tilde{J}^k
	           \hat{\mathrm{P}}
	 \label{eq:coarse_scale_def2}
\end{align}
\end{subequations}
where $\tilde{r}^k$ and $\tilde{J}^k$ are the modified residual and Jacobian computed from Eq.\ref{eq:mod_res}, respectively (setting $\partial_t\bs{u}_k\!=\!\bs{0}$). However, the omission of subscript $w$ from both implies that we set $\bs{w}\!=\!\bs{u}_k$ in Eq.\ref{eq:mod_res}. In other words, $\tilde{r}^k$ and $\tilde{J}^k$ are the same as $r^k$ and $J^k$ in Eq.\ref{eq:Ax=b}, except the inertia term is removed from fine-grid faces near contact interfaces (Fig.\ref{fig:coarse_sketch}). $\hat{\mathrm{R}}$ and $\hat{\mathrm{P}}$ are the effective restriction and effective prolongation matrices, respectively; defined in Eq.\ref{eq:iMG_PR}. We build $\hat{\mathrm{R}}$ and $\hat{\mathrm{P}}$ by setting $\hat{\mathrm{A}}$ in Section \ref{sec:mono_PLMM_MG} to $\tilde{J}_w$ for a given wind $\bs{w}$ ($\neq\!\bs{u}_k$) at a nearby $Re$. If no such $\bs{w}$ is available, the Stokes velocity serves as fallback.
To map the course-scale update $\smash{\delta_{c}^{k}}$ onto the fine grid and obtain the approximate fine-scale solution at the next Newton iteration, we execute:
\begin{equation} \label{eq:coarse_scale_up}
	\hat{x}_{aprx}^{k+1} = \hat{x}_{aprx}^{k} + \hat{\mathrm{P}}\, \delta_{c}^k
\end{equation}

Note that $\hat{\mathrm{R}}$ and $\hat{\mathrm{P}}$ do \textit{not} depend on Newton iterations and are built only once. If they were built from $J^k$ in Eq.\ref{eq:Ax=b} instead of $\tilde{J}_w$, coarse iterations would diverge or yield poor approximate solutions. This is because the closure $\partial_n\bs{u}\!=\!\bs{0}$ is only strictly valid at interfaces when $\tilde{J}_w$ is used to construct $\mathrm{M_G}$. Thus, only $\mathrm{M_G}$ from $g$PLMM and $g$PNM can serve as coarse solvers, not $a$PLMM or $a$PNM. We show later that the $\mathrm{M_G}$ of $g$PLMM is far more accurate than $g$PNM.

Algorithm \ref{alg:coarse_scale} summarizes the above steps for the proposed coarse-scale solver of the steady-state Navier-Stokes equations. The algorithm consists of an outer loop that gradually increments $Re$ from zero to a desired maximum, enveloping the coarse-scale Newton solver as an inner loop. Here, $Re$ is incremented by increasing the inlet-boundary pressure while keeping the outlet pressure constant. The outer loop's job is to provide a good initial guess for Newton to converge (iterations diverge if too far from $\hat{x}_{aprx}$). In computing $\tilde{J}_w$, we set $\bs{w}$ to the velocity from the previous $Re$.

\begin{minipage}{.8\linewidth} 
\vspace{1.5em} 
\begin{algorithm}[H] 
\small
\caption{Coarse-scale solver for steady-state Navier-Stokes equations}
\label{alg:coarse_scale}
\begin{algorithmic}
	\State Initialize $\hat{x}_{aprx}=0$ and $\bs{w}=0$
	\State \textbf{Do} $Re = 0$ to $Re_{max}$ with increment $\Delta Re$ 
	                              \qquad\quad\qquad\qquad\qquad\qquad$\;\:$ $\triangleright$ Increment Reynolds number
	\State \quad Set wind $\bs{w}$ in Eq.\ref{eq:mod_res} to velocity from previous $Re$ increment
	\State \quad Build coarse preconditioner $\mathrm{M_G}$ from Jacobian $\tilde{J}_w$ of Eq.\ref{eq:mod_res} \qquad\quad $\triangleright$ This yields $\hat{\mathrm{R}}$ and $\hat{\mathrm{P}}$
	\State \quad Compute coarse-scale residual $r_c^0$ and Jacobian $J_c^0$ from Eq.\ref{eq:coarse_rJ}
	
	\State \quad Initialize Newton to $\hat{x}^0_{aprx} = \hat{x}_{aprx}$ from previous $Re$ increment 
	\State \quad \textbf{Do} $k = 0$ to $k_{max}$
	                                        \hfill $\triangleright$ Coarse-scale Newton iteration
	
	\State \quad \quad Compute coarse-scale Newton update $\delta_{c}^{k}$ by solving Eq.\ref{eq:coarse_scale_sys}
	\State \quad \quad Update fine-scale approximate solution $\hat{x}^{k+1}_{aprx}$ using Eq.\ref{eq:coarse_scale_up}
	\State \quad \quad Update coarse-scale residual $r_c^{k+1}$ and Jacobian $J_c^{k+1}$ from Eq.\ref{eq:coarse_rJ}
	\State \quad \quad \textbf{If} $\|r_c^{k+1}\| \,/\, \|r_c^{0}\| < 10^{-6}$, \textbf{exit}
	
	\State \quad \textbf{End do}
	\State \quad Set $\hat{x}_{aprx}$ to converged approximate solution
	\State \textbf{End do}
\end{algorithmic}
\end{algorithm}
\end{minipage}
\\
\\

\noindent\textbf{Remark 2}. For the linear Stokes equations, Algorithm \ref{alg:coarse_scale} converges in only one Newton iteration and the solution equals Eq.\ref{eq:approx_formula} but with $\mathrm{M_G^{-1}}\!=\!\mathrm{\hat{P}(\hat{R}\hat{A}\hat{P})^{-1}\hat{R}}$ (not Eq.\ref{eq:iMG}). Another way to formulate the coarse-scale Jacobian given by Eq.\ref{eq:coarse_scale_def2} would have been $J_c^k\!=\!\mathrm{W}\mathcal{C}(\mathrm{W}^\top\tilde{J}^k\mathrm{W})\mathrm{W}^\top$, but the presence of the column-sum operator $\mathcal{C}(\cdot)$ prevents Newton from converging in only one step, which is why we opted for the form given by Eq.\ref{eq:coarse_scale_def2}. The simple proof of this fact involves tracing one Newton step of Algorithm \ref{alg:coarse_scale}, revealing that $(\mathrm{\hat{R}}\tilde{J}^k\mathrm{\hat{P}})^{-1}J_c^k\!=\!\mathrm{I}$ must hold for one-step convergence.

\section{Problem set} \label{sec:problem_set}
In the following sections, we first test the coarse-scale solver in Algorithm \ref{alg:coarse_scale} against the exact solution of the steady-state Navier-Stokes equations for the 2D and 3D porous geometries in Fig.\ref{fig:domains}. Both $g$PLMM and $g$PNM are considered as the basis for building $\mathrm{M_G}$ for the solver. The domains include a monodisperse and spatially disordered disk pack \cite{mehmani2018mult} (GLD4), a polydisperse and spatially disordered disk pack \cite{mehmani2018mult} (PD4), a 2D representation of Berea sandstone \cite{boek2010lberea} (Berea), and a 3D Gyroid whose void-solid interface $\Gamma_w$ is described by the following equation \cite{schoen1970gyroid}:

\begin{equation}
	\sin(12x)\cos(12y) + \sin(12y)\cos(12z) + \sin(12z)\cos(12x) = 0 
\end{equation}

After gaining some intuition into the accuracy of $g$PLMM and $g$PNM as coarse-scale solvers, we test the algebraic ($a$PLMM and $a$PNM) and geometric ($g$PLMM and $g$PNM) monolithic preconditioners proposed for both steady-state and unsteady forms of the Navier-Stokes equations. The test domains include GLD4, PD4, Berea, and a 3D packing of non-spherical grains \cite{mehmani2025multiscale} (Granular). Disk positions in PD4 and GLD4 are identical, differing only in the disk sizes. We benchmark $a$PLMM, $a$PNM, $g$PLMM and $g$PNM against each other and block-triangular preconditioners $b$AMG, $b$PLMM, and $b$PNM. The latter were described in Section \ref{sec:intro} and detailed in \cite{mehmani2025multiscale}. We utilize the \textit{scaled}-BFBt approximation of the Schur-complement matrix in Eq.\ref{eq:BFBt_scaled}, which exhibits robust performance in complex geometries \cite{mehmani2025multiscale}. The difference between $b$AMG, $b$PLMM, and $b$PNM lies in the choice of prolongation (transposition yields restriction) matrices used to precondition the elliptic sub-systems arising from the Schur-complement approximation. $b$PLMM and $b$PNM are built from the same shape and correction vectors used to formulate $a$PLMM and $a$PNM. To build $b$AMG, we use the \texttt{amg.m} code from the iFEM GitHub repository \cite{Chen2008ifem}, with settings identical to those in \cite{mehmani2025multiscale}.

Each domain is decomposed into primary and dual grids corresponding to the randomly colored regions in Fig.\ref{fig:domain_decomp}. These are used by $a$PLMM, $a$PNM, $g$PLMM, $g$PNM, $b$PLMM, and $b$PNM. Table \ref{tab:domains} summarizes each domain's dimensions and fine/coarse-grid information, including the number of cell-pressure unknowns ($n_c$), face-velocity unknowns ($n_f$), primary grids ($N^p$), and dual grids ($N^d\!=\!N^c$) for both PLMM and PNM. The $\smash{N^p}$ for PLMM roughly equals the $\smash{N^c}$ for PNM because primary grids in PLMM correspond to contact interfaces (or dual grids) in PNM and vice versa. 

To compute Reynolds number $Re\!=\!\rho U l/\mu$, we define the characteristic velocity as the interstitial velocity $U\!=\!U_D/\phi$, where $U_D$ is the Darcy velocity ($=$ flowrate divided by cross-sectional area) and $\phi$ is porosity. The characteristic length is $l\! =\! V_s / A_w$, where $V_s$ is the total solid volume and $A_w$ is the void-solid interfacial area $|\Gamma_w|$. In all simulations, we set the fluid density to $\rho\!=\!1$ g/cc and the dynamic viscosity to $\mu\!=\!2 \times 10^{-4}$ Poise for PD4, $1.2 \times 10^{-4}$ Poise for GLD4, $1 \times 10^{-4}$ Poise for Berea, and $1 \times 10^{-3}$ Poise for Granular and Gyroid. The BCs imposed were described in Section \ref{sec:problem}, which include setting the inlet pressure $p_{in}$ at the left boundary and the outlet pressure $p_{out}$ at the right. 

To validate our fine-grid Navier-Stokes solver, we simulate the classic problem of flow around a circular or square obstacles at various Reynolds numbers in \ref{sec:appendix_a}. These simulations show that our solver produces accurate drag coefficients compared to literature values, with comparable pressure, velocity-magnitude, and vorticity fields.

\begin{figure} [t!]
  \centering
  \centerline{\includegraphics[scale=0.4, trim={0cm 0.0cm 0cm 0cm},clip]{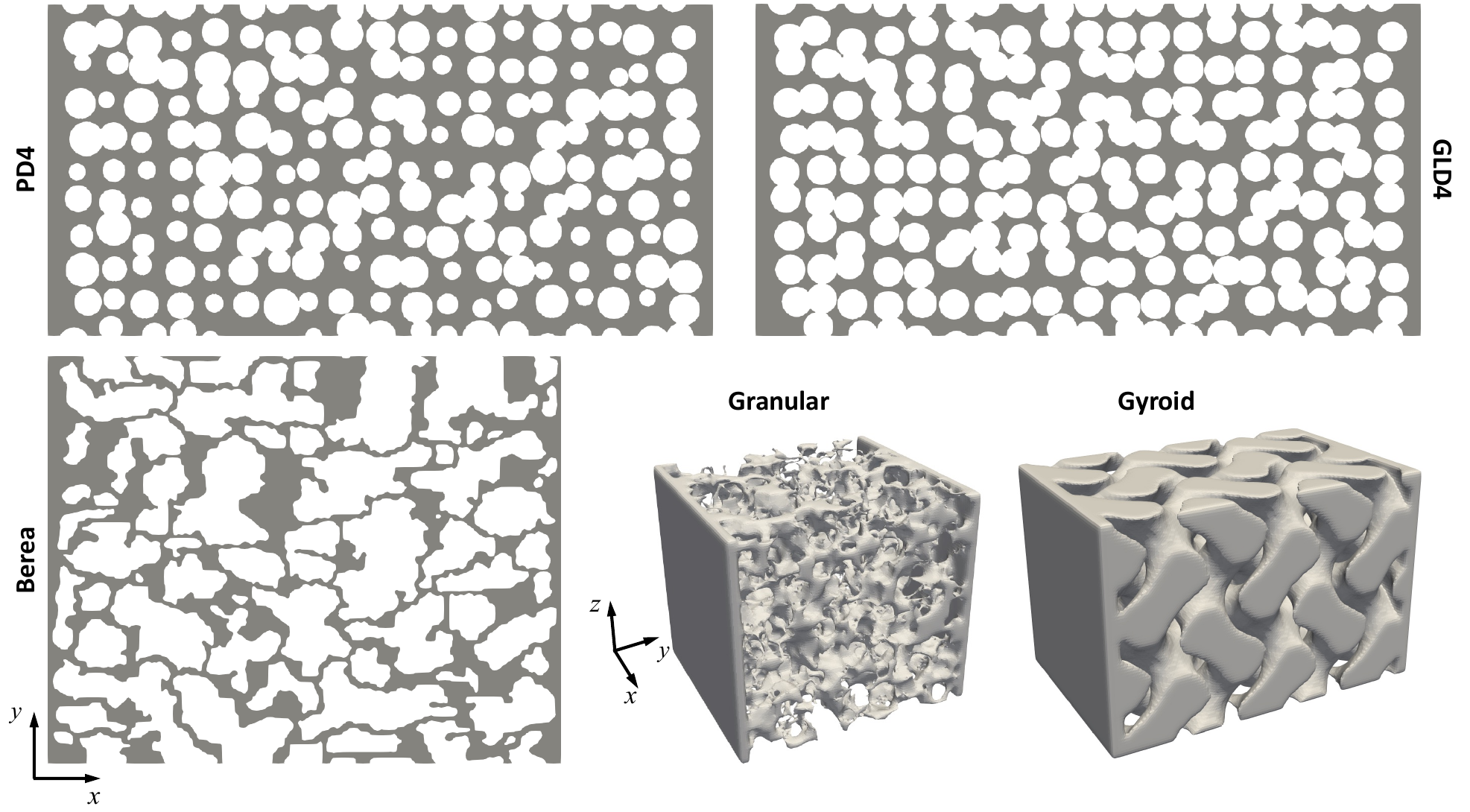}}
  \caption{Schematic of 2D and 3D porous geometries considered for testing the proposed preconditioners and the coarse-scale solver.}
\label{fig:domains}
\end{figure}

\begin{figure} [t!]
  \centering
  \centerline{\includegraphics[scale=0.38, trim={0cm 0.0cm 0cm 0cm},clip]{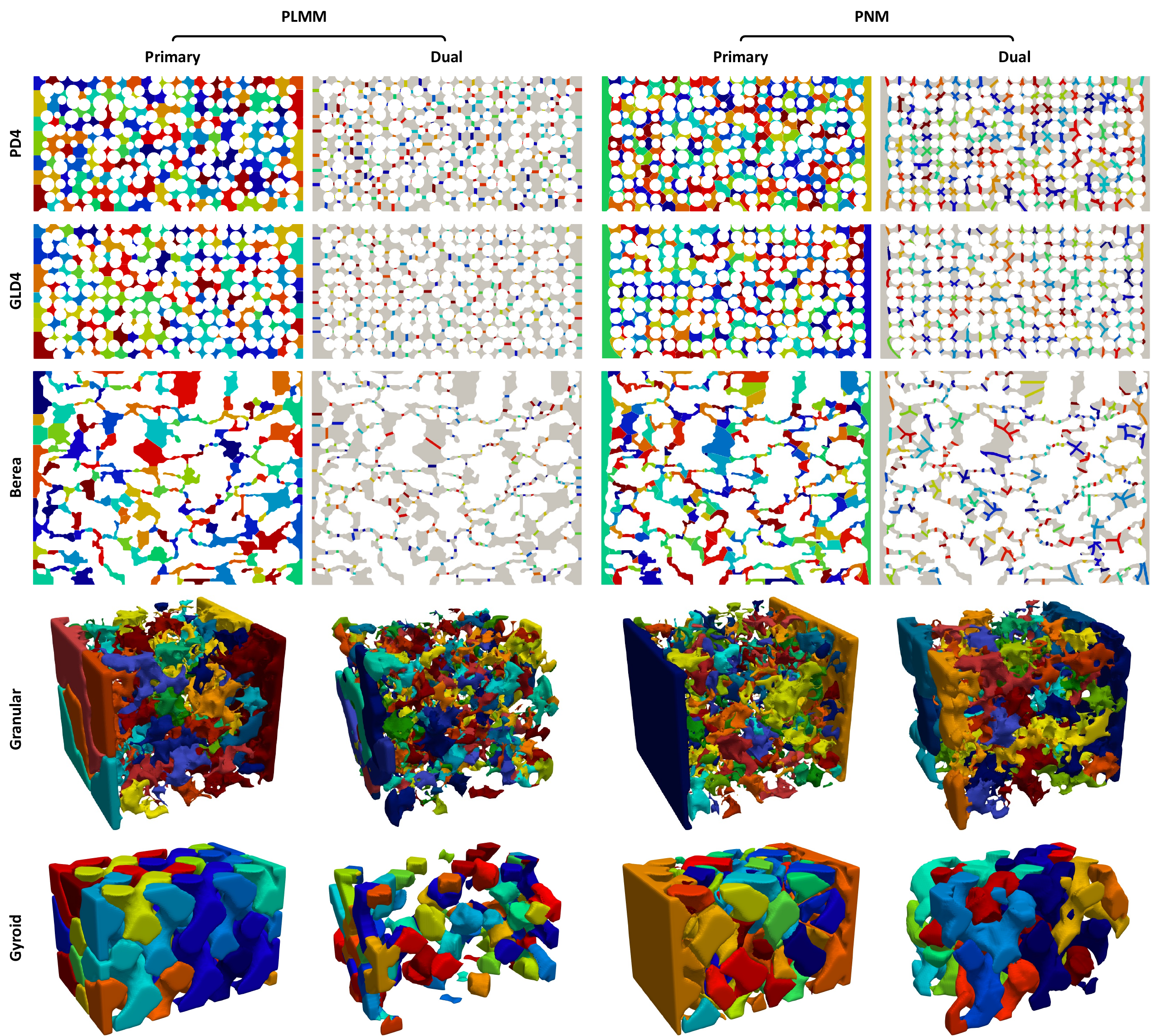}}
  \caption{Decomposition of the porous geometries shown in Fig.\ref{fig:domains} into primary and dual grids used by the coarse-scale solver and the monolithic and block preconditioners considered herein. Each primary grid and dual grid is illustrated by a randomly colored region.}
\label{fig:domain_decomp}
\end{figure}

\begin{table}[b!]\centering
	\caption{Geometric and fine/coarse-grid information of the domains in Fig.\ref{fig:domains}. They include the domain size, number of cell-pressure unknowns ($n_c$), number of face-velocity unknowns ($n_f$), and number of primary ($N^p$) and dual ($N^d\!=\!N^c$) grids for PLMM- and PNM-based preconditioners.}
	{\small
  	\begin{tabular}{ P{1.5cm} | P{2.4cm} | P{1.8cm} | P{1.8cm} | P{1.8cm} | P{1.8cm} | P{1.8cm} | P{1.8cm}}
  	\cline{2-8}
{} & Domain size (cm) & $n_c$ & $n_f$ & $N^p$, PLMM & $N^c$, PLMM & $N^p$, PNM & \multicolumn{1}{|c|}{$\;\;\;$ $N^c$, PNM$\;\;\;$}  \\
	\hline
	\rowcolor[gray]{0.9}\multicolumn{1}{|c|}{PD4} 
	     & 2 $\times$ 1 &  610,040 & 1,230,924 & 174 & 294 &  291 & \multicolumn{1}{|c|}{172} \\
	\hline
	\multicolumn{1}{|c|}{GLD4} 
	     & 2 $\times$ 1 &  717,097 & 1,447,556 & 186 & 257 &  260 & \multicolumn{1}{|c|}{178} \\
	\hline
	\rowcolor[gray]{0.9}\multicolumn{1}{|c|}{Berea}
	     & 1.45 $\times$ 1.15 &  931,396 & 1,877,760 & 203 & 255 &  302 & \multicolumn{1}{|c|}{192} \\
	\hline
	\multicolumn{1}{|c|}{Granular} 
	     & 0.9 $\times$ 1 $\times$ 0.9 & 375,953 & 1,265,199 & 99 & 690 &  647 & \multicolumn{1}{|c|}{96} \\
	\hline
	\rowcolor[gray]{0.9}\multicolumn{1}{|c|}{Gyroid}
	     & 1$\times$ 1.5 $\times$ 1 &  616,400 & 1,947,468 & 21 & 103 &  109 & \multicolumn{1}{|c|}{16} \\
	\hline
 	\end{tabular}
	}
	\label{tab:domains}
\end{table}

\section{Testing the coarse-scale solver} \label{sec:test_coarse}
We test the accuracy of the coarse-scale solver outlined in Algorithm \ref{alg:coarse_scale} for the steady-state Navier-Stokes equations by comparing the pressure and velocity magnitudes associated with the approximate solution $\hat{x}_{aprx}$ produced to that of the exact solution. To increment $Re$ in Algorithm \ref{alg:coarse_scale}, we progressively increase the inlet pressure $p_{in}$ by 10\% of its previous value while keeping the outlet pressure $p_{out}$ fixed at zero. The initial value of $p_{in}$ is set to $10^{-3}$ dynes/cm\textsuperscript{2}. We increment $Re$ in this way until the coarse-scale Newton loop ceases to converge within 50 iterations, which we interpret as the onset (or vicinity) of turbulent instabilities beyond which a steady-state solution is difficult to compute. Recall Algorithm \ref{alg:coarse_scale} uses the converged solution from the previous $Re$ increment as the initial guess for Newton in the current $Re$ increment. In addition to pressure and velocity fields, which are fine-scale information, we compute the permeability $K$ of each domain versus $Re$, which is a bulk property:
\begin{equation} \label{eq:perm}
	K = \frac{\mu L U_D}{\Delta p} 
\end{equation}
The $U_D$, $\Delta p$, and $L$ denote Darcy velocity, $p_{in}-p_{out}$, and inlet-to-outlet distance of the domain, respectively.

\begin{figure} [b!]
  \hspace{-0.5cm}
  \centerline{\includegraphics[scale=0.41, trim={0cm 0.0cm 0cm 0cm},clip]{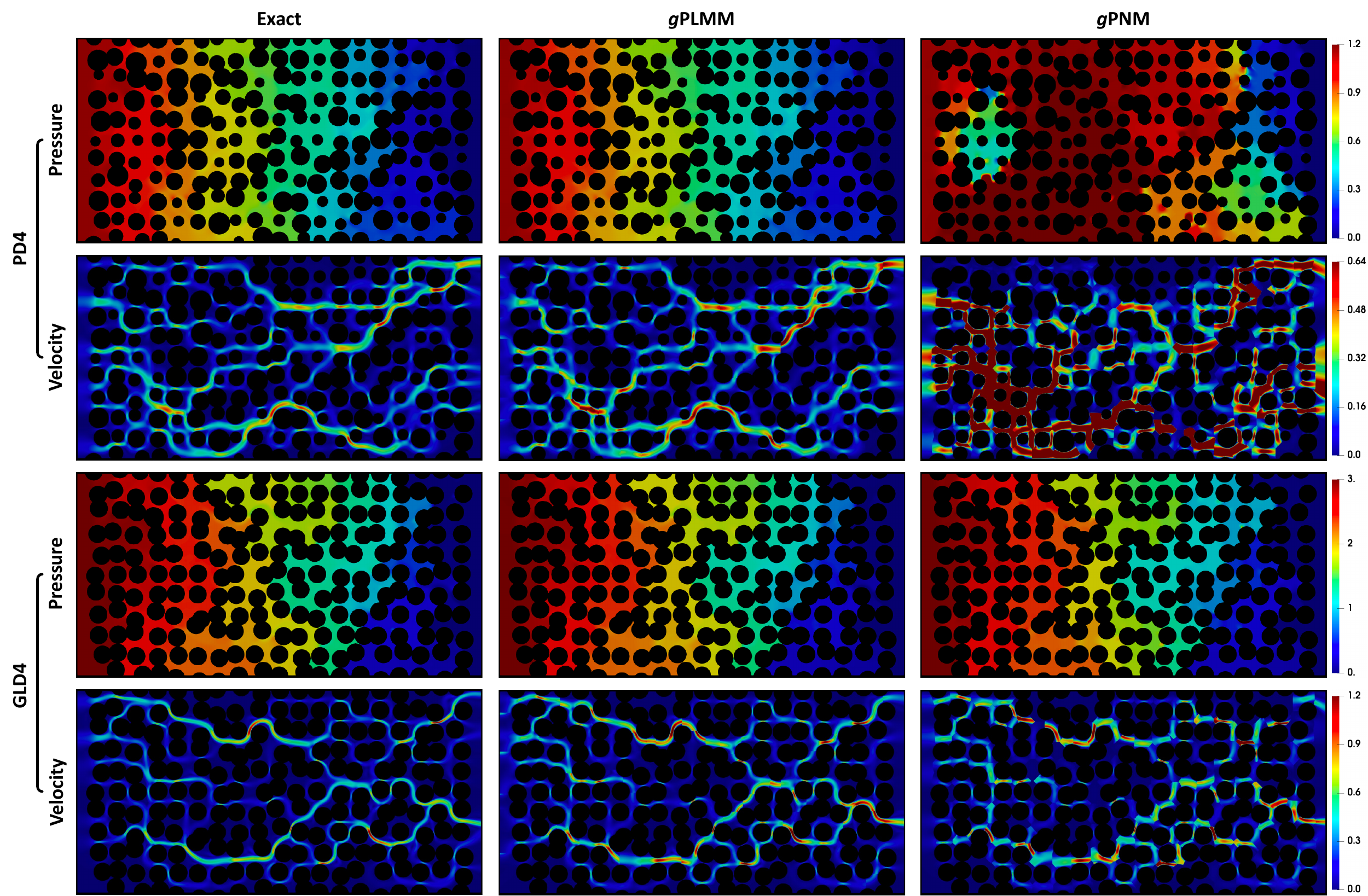}}
  \caption{Comparison of approximate pressure and velocity-magnitude fields obtained from the coarse-scale solver in Algorithm \ref{alg:coarse_scale} against exact solutions. Results correspond to the \textit{steady-state Navier-Stokes equations} on PD4 at $Re$ = $3.1\times 10^3$ and GLD4 at $Re$ = $1.1\times10^4$.}
\label{fig:coarse_sol1}
\end{figure}

\begin{figure} [t!]
  \hspace{-0.5cm}
  \centerline{\includegraphics[scale=0.42, trim={0cm 0.0cm 0cm 0cm},clip]{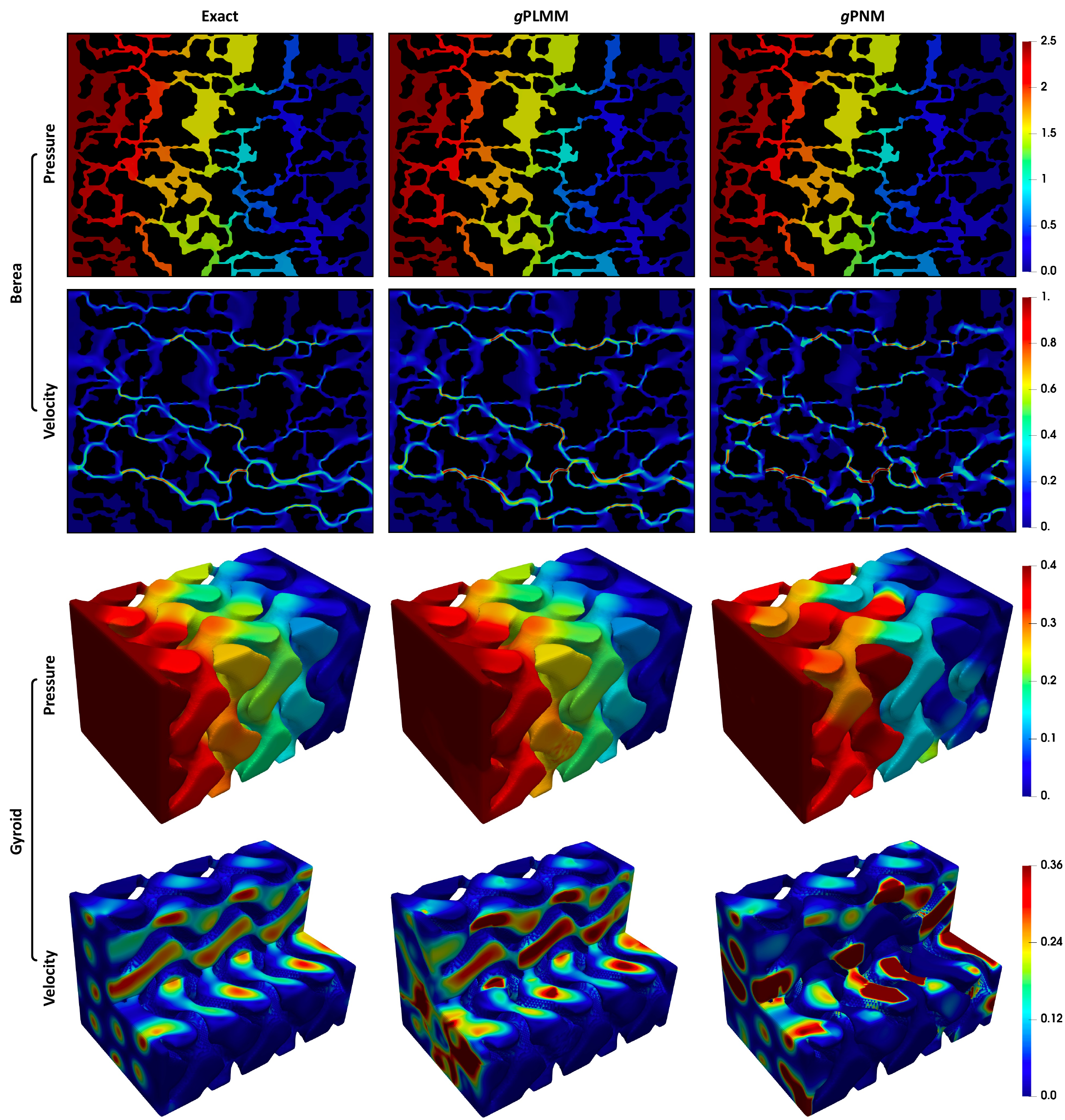}}
  \caption{Comparison of approximate pressure and velocity-magnitude fields obtained from the coarse-scale solver in Algorithm \ref{alg:coarse_scale} against exact solutions. Results correspond to the \textit{steady-state Navier-Stokes equations} on Berea at $Re$ = $1.6\times10^4$ and Gyroid at $Re$ = $4.7\times10^3$.}
\label{fig:coarse_sol2}
\end{figure}

Figs.\ref{fig:coarse_sol1}--\ref{fig:coarse_sol2} compare approximate solutions from the coarse-scale solver to exact solutions at the highest $Re$ reachable by the coarse-scale solver. Results are shown for the PD4, GLD4, Berea and Gyroid domains. Approximate solutions from both $g$PLMM and $g$PNM are included, where the corresponding $\mathrm{M_G}$ (or equivalently $\mathrm{\hat{R}}$ and $\mathrm{\hat{R}}$) is utilized in Algorithm \ref{alg:coarse_scale}. We observe very good agreement between approximate pressure and velocity-magnitudes from $g$PLMM versus the exact solutions. Approximate solutions via $g$PNM are significantly inferior in accurately capturing the flow field. These observations are confirmed further by Fig.\ref{fig:coarse_plot}, where comparisons of $K$ versus $Re$ are presented.

To quantify errors in the approximate solutions, we use the following metrics by \cite{mehmani2018mult, mehmani2025multiscale}:
\begin{equation} \label{eq:error_L2}
E^\xi_2 \equiv \left(\frac{1}{|\Omega|} \int_{\Omega} \left( E^\xi_p \right)^2 d\Omega \right)^{1/2}
\qquad
E^\xi_p \equiv \frac{|\,\xi_{t} - \xi_{c}|}{\mathrm{sup}_{\Omega} |\,\xi_{t}|} \times 100
\qquad
E_K \equiv \frac{|\,K_t - K_c\,|}{|\,K_t\,|} \times 100
\end{equation}
where $\xi$ is a placeholder for either pressure, $p$, or velocity magnitude, $|\bs{u}|$. Subscripts $c$ and $t$ denote the approximate solution obtained from the coarse-scale solver and the ``true'' (or exact) solution, respectively. From left to right, the metrics in Eq.\ref{eq:error_L2} yield the $L_2$, pointwise, and permeability errors; all expressed in percentages of the true value.

\begin{figure} [t!]
  \centering
  \centerline{\includegraphics[scale=0.42, trim={0cm 0.0cm 0cm 0cm},clip]{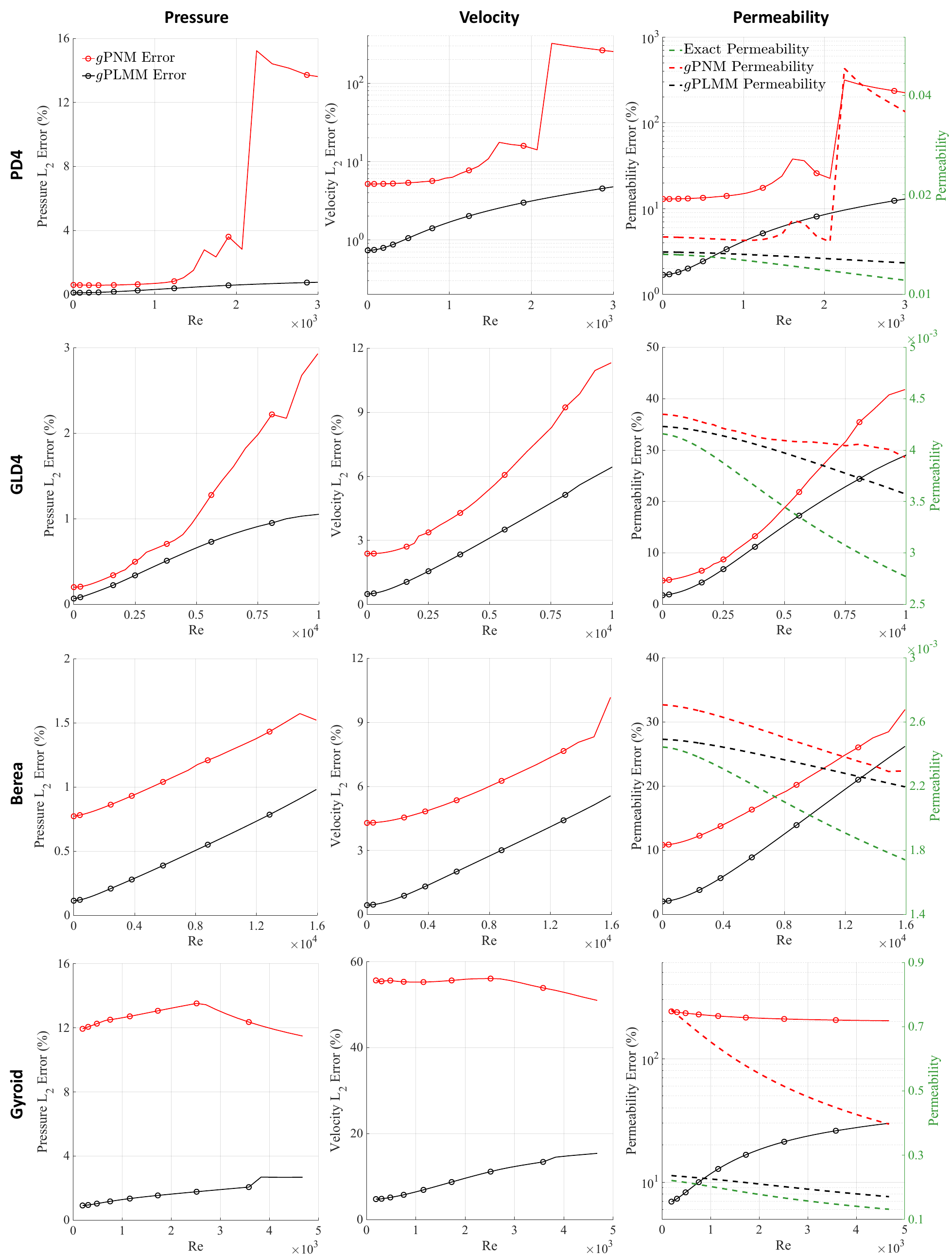}}
  \caption{Errors in pressure ($E^p_2$), velocity magnitude ($E^{|\bs{u}|}_2$), and permeability ($E_K$) versus Reynolds number ($Re$) for approximate solutions from the coarse-scale solver in Algorithm \ref{alg:coarse_scale} based on $g$PLMM and $g$PNM. Results correspond to the PD4, GLD4, Berea, and Gyroid domains. The approximate and exact permeabilities ($K$) are plotted with dashed lines on the right. The former overestimates $K$ due to vortex suppression.}
\label{fig:coarse_plot}
\end{figure}

Fig.\ref{fig:coarse_plot} plots $E^p_2$, $E^{|\bs{u}|}_2$, and $E_K$ versus $Re$ for all domains, where we see these errors increase with Reynolds number. Moreover, errors associated with $g$PNM are much larger (up to 20 times) than those with $g$PLMM. For the coarse-solver based on $g$PLMM, errors are bounded as follows for all domains: $E^p_2\!<\!3$\%, $E^{|\bs{u}|}_2\!<\!6$\% in 2D and $<\!18$\% in 3D, and $E_K\!<\!30$\%, which are acceptable in many engineering applications. We note the approximate velocity field of the coarse-scale solver is conservative in a ``weak'' sense, i.e., integrated flux along each interface is continuous. Fig.\ref{fig:coarse_err} depicts pointwise errors of the pressure and velocity-magnitude in Figs.\ref{fig:coarse_sol1}--\ref{fig:coarse_sol2} for the PD4 and Gyroid domains, which confirm the observations just described. $g$PLMM errors are lower than $g$PNM, and under $1$\% for much of the domain. Other domains behave similarly and omitted for brevity. Unlike Stokes flow, where errors localize more strongly near contact interfaces \cite{mehmani2025multiscale} (high-frequency), Navier-Stokes exhibits errors that are comparatively global (low-frequency).

An interesting observation from Fig.\ref{fig:coarse_plot} is that the permeability predicted by the $g$PLMM coarse-scale solver is always larger than the true permeability computed from the exact solution. This is because the imposition of the closure BC $\partial_n\bs{u}\!=\!\bs{0}$ suppresses vortex formation, which is a key mechanism for dissipating energy, hence reducing flowrate. This is illustrated by Fig.\ref{fig:coarse_streamline}, where streamlines at low and high $Re$ are drawn inside a portion of the domain shown in Fig.\ref{fig:main_sketch}a. The suppression of vortices is, incidentally, the reason why errors in Fig.\ref{fig:coarse_plot} grow with $Re$.

\begin{figure} [h!]
  \centering
  \centerline{\includegraphics[scale=0.39, trim={0cm 0.0cm 0cm 0cm},clip]{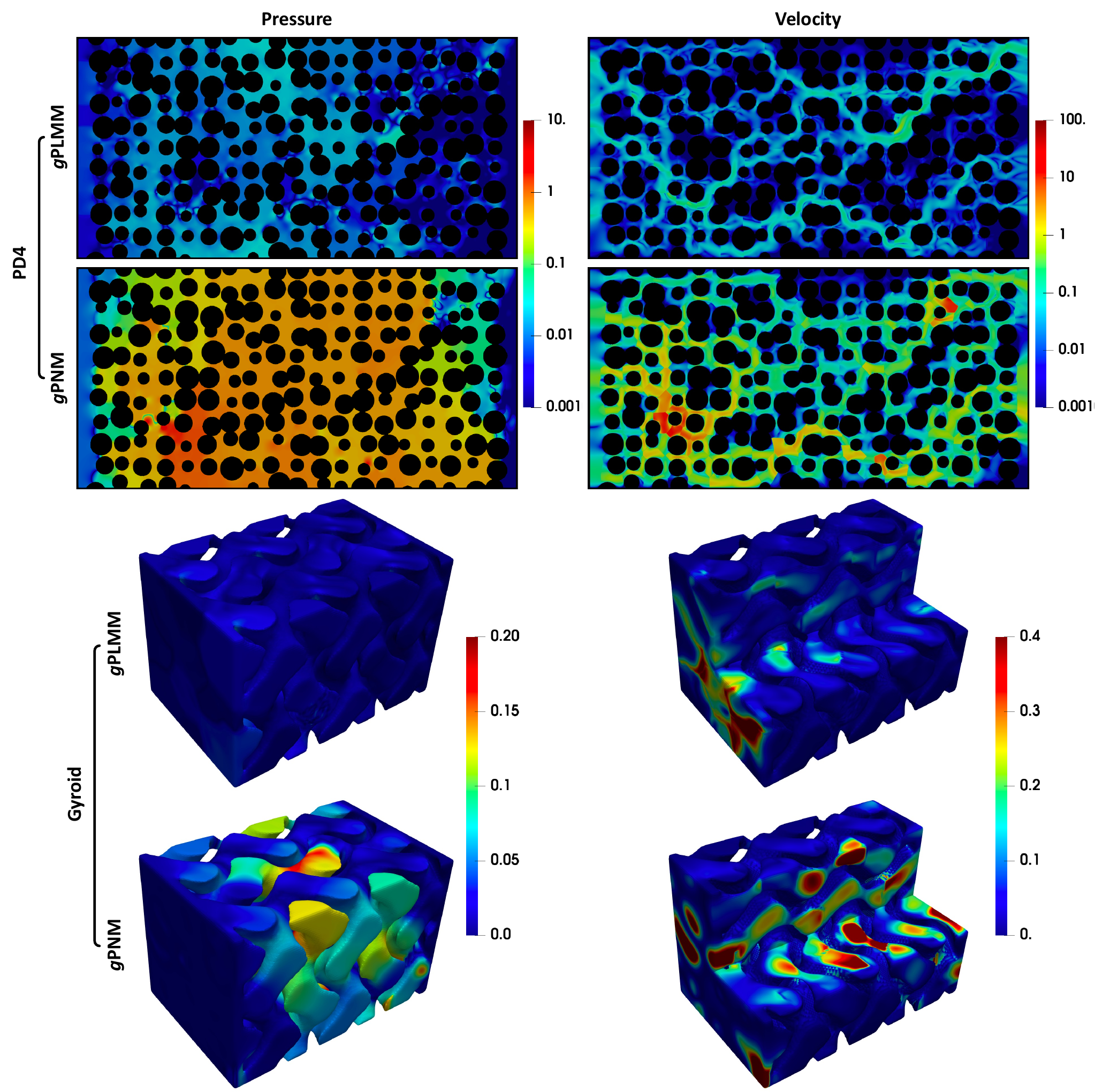}}
  \caption{Pointwise errors in pressure ($E^p_p$) and velocity magnitude ($E^{|\bs{u}|}_p$) for approximate solutions from the coarse-scale solver in Algorithm \ref{alg:coarse_scale} based on $g$PLMM and $g$PNM. Results correspond to PD4 at $Re$ = $3.1\times 10^3$ and Gyroid at $Re$ = $4.7\times10^3$, whose field plots are shown in Figs.\ref{fig:coarse_sol1}--\ref{fig:coarse_sol2}.}
\label{fig:coarse_err}
\end{figure}

\begin{figure} [h!]
  \centering
  \centerline{\includegraphics[scale=0.43, trim={0cm 0.0cm 0cm 0cm},clip]{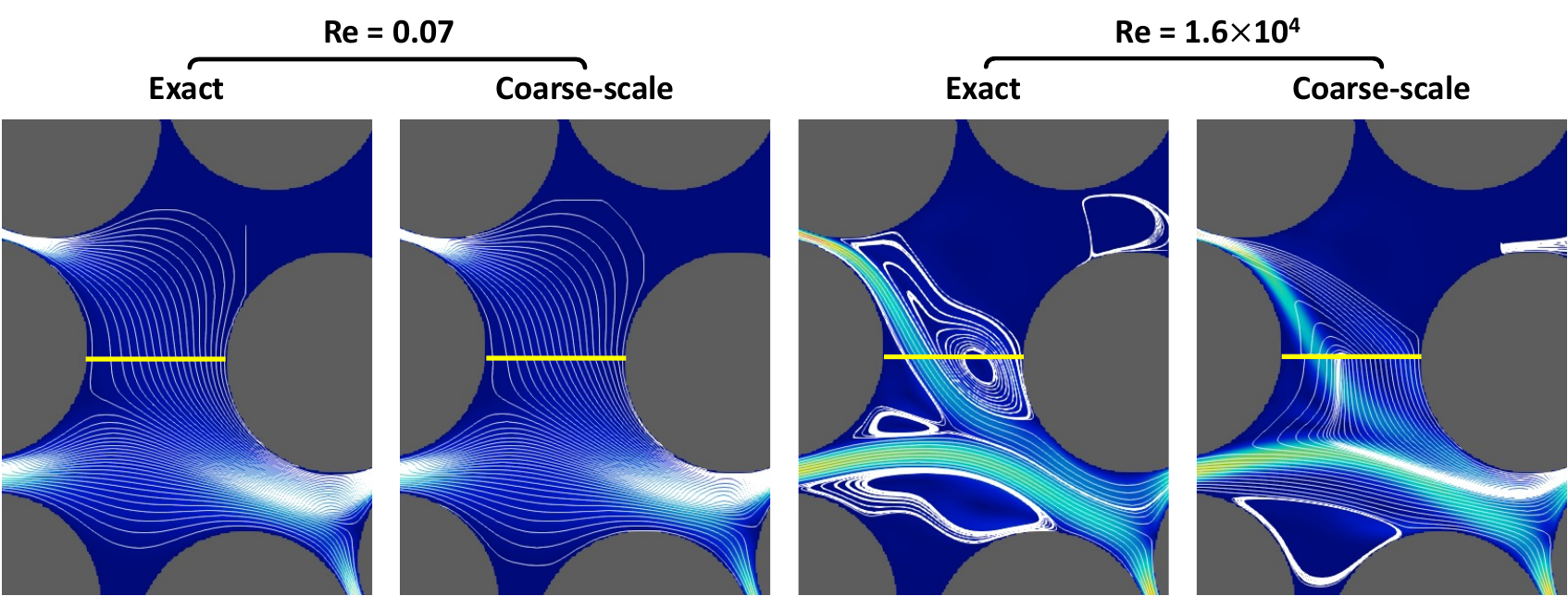}}
  \caption{Schematic of flow streamlines (white lines) near a contact interface (yellow line). The velocity field is computed from the $g$PLMM coarse-scale solver and the exact solution. Results are shown for low and high $Re$ corresponding to a portion of the domain in Fig. \ref{fig:main_sketch}.}
\label{fig:coarse_streamline}
\end{figure}

\section{Testing the preconditioners} \label{sec:test_prec}
We apply all preconditioners, including $a$PLMM/$a$PNM, $g$PLMM/$g$PNM, $b$PLMM/$b$PNM, and $b$AMG, within right-preconditioned GMRES. We declare convergence when the normalized residual satisfies $\| \hat{A}\hat{x} - \hat{b} \| / \|\hat{b}_{0}\|  \!<\! 10^{-9}$, where $\hat{b}_{0}$ is the residual of the steady-state (or time-dependent if velocity at prior time step is set to zero) Navier-Stokes equations at the origin, i.e., $r(\bs{u} = \bs{0}, p = 0)$. This captures the magnitude of nonhomogeneous BCs. We use $\|\hat{b}_{0}\|\!=\!\|r^0 \|$ for normalization instead of $\| \hat{b} \|\!=\!\| r^k\|$ because the latter decreases with Newton iterations, leading to excess GMRES iterations at large $k$. We run simulations in series on an Intel\textsuperscript{\textregistered} Xeon\textsuperscript{\textregistered} Gold 6248R CPU @ 3.00 GHz machine.

\subsection{Steady-state Navier-Stokes equations} \label{sec:test_prec_st}
Similar to Section \ref{sec:test_coarse}, we increment $Re$ by progressively increasing the inlet pressure $p_{in}$ from 1 dynes/cm\textsuperscript{2} using constant increments of $\Delta p_{in}\!=\!1$ for PD4, 0.5 for GLD4 and Berea, and 2 for the Granular domain, while keeping the outlet pressure at $p_{out}\!=\!0$. Unlike Section \ref{sec:test_coarse}, fine-scale (not coarse-scale) solutions are computed at each $Re$ by performing Newton iterations on the fine grid, and GMRES iterations at each Newton iteration. The initial Newton guess at any given $Re$ is the solution from the previous $Re$ increment. To probe the necessity of rebuilding the $a$PLMM/$a$PNM, $g$PLMM/$g$PNM, and $b$PLMM/$b$PNM preconditioners as $Re$ evolves, we test two variants of each: (1) \textit{Stokes-based}, where the preconditioners are built from the Stokes equations once and reused across all subsequent $Re$; and (2) \textit{Navier-Stokes-based}, where the preconditioners are rebuilt from scratch at the start of every $Re$ increment; specifically, before the Newton loop is entered, but not during. Following the notation in Section \ref{sec:intro}, we denote Stokes-based preconditioners with subscript ``S'' and Navier-Stokes-based preconditioners with ``NS.'' For example, $b\mathrm{PLMM_S}$ is the block PLMM preconditioner based on the Stokes equations, and $a\mathrm{PNM_{NS}}$ is the algebraic monolithic PNM preconditioner that is continuously updated based on the Navier-Stokes equations. Since $g\mathrm{PLMM_{S}}$=$a\mathrm{PLMM_{S}}$ and $g\mathrm{PNM_{S}}$=$a\mathrm{PNM_{S}}$, we use $g\mathrm{PLMM}$ and $g\mathrm{PNM}$ (no subscript) to denote Navier-Stokes-based preconditioners.

\begin{figure} [t!]
  \centering
  \centerline{\includegraphics[scale=0.39, trim={0cm 0.0cm 0cm 0cm},clip]{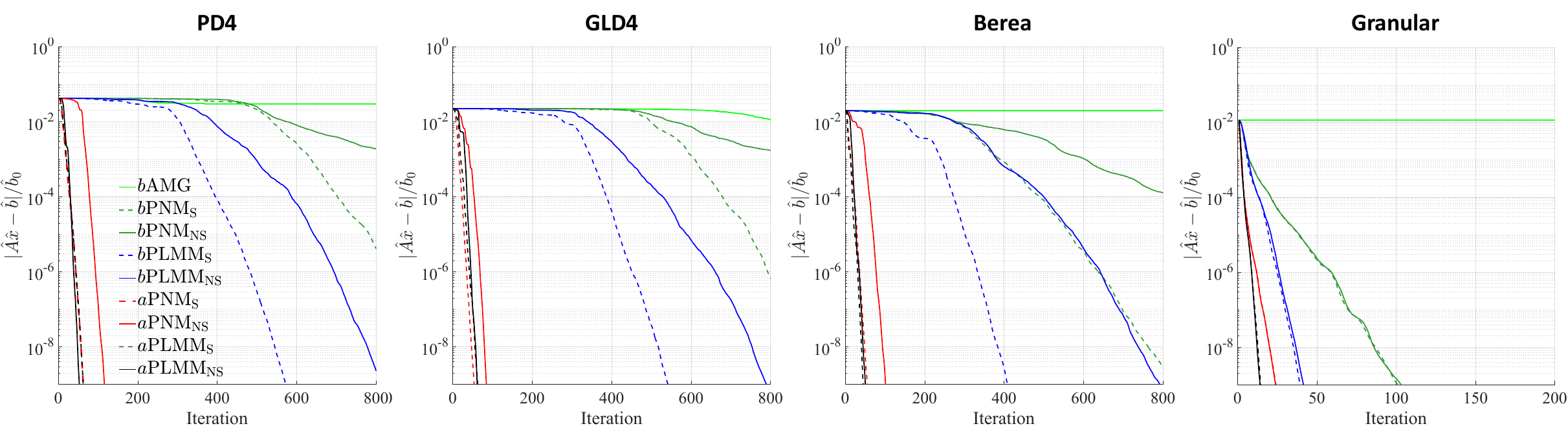}}
  \caption{Normalized linear-system residual versus GMRES iterations for the first Newton step of the first $Re$ increment ($Re_1$ = 9.0$\times 10^3$). Results correspond to different block and monolithic preconditioners applied to the \textit{steady-state Navier-Stokes equations} on all domains.}
\label{fig:qa_result_conv}
\end{figure}
 
\begin{table}[b!]\centering
	\caption{Total wall-clock times (seconds) associated with the first three $Re$ increments in solving the steady-state Navier-Stokes equations on the PD4 domain. They include the costs of constructing $\mathrm{M_G}$, constructing $\mathrm{M_L}$, and the self-time of GMRES for different preconditioners in GMRES.}
	{\small
  	\begin{tabular}{ P{2.2cm} | P{1.8cm} | P{1.8cm} | P{1.8cm} | P{1.8cm} | P{1.8cm} | P{1.8cm} | P{1.8cm} | P{1.8cm} |}
  	\cline{2-9}
{} & $b\mathrm{PNM_S}$ & $b\mathrm{PNM_{NS}}$  & $b\mathrm{PLMM_S}$  & $b\mathrm{PLMM_{NS}}$  & $a\mathrm{PNM_S}$  & $a\mathrm{PNM_{NS}}$ & $a\mathrm{PLMM_S}$ & $a\mathrm{PLMM_{NS}}$ \\
	\hline
	\rowcolor[gray]{0.9}\multicolumn{1}{|c|}{$Re_1$ = 9.0$\times 10^3$} 
	     & 40,088 & 87,345 & 22,295 &  44,478 & 658 & 1,158 & 566 & 758 \\
	\hline
	\multicolumn{1}{|c|}{$Re_2$ = 1.2$\times 10^4$} 
	     & 48,796 & 110,934 & 31,011 & 65,789 & 674 & 1,800 & 563 & 741 \\
	\hline
	\rowcolor[gray]{0.9}\multicolumn{1}{|c|}{$Re_3$ = 1.4$\times 10^4$}
	     & 62,347 & 137,529 & 49,125 & 102,886 & 824 & 2,405 & 643 & 769\\
	\hline
 	\end{tabular}
	}
	\label{tab:pd4_time}
\end{table}

First, let us demonstrate that block preconditioners are disqualified from further consideration due to their severely poor performance. Fig.\ref{fig:qa_result_conv} compares the GMRES convergence observed for the monolithic and block preconditioners in the first Newton iteration of the first $Re$ increment. This is the simplest system to solve, as higher $Re$ poses further challenges to GMRES (shown later). Compared to $a$PLMM and $a$PNM, the block preconditioners $b$PLMM, $b$PNM, and $b$AMG are extremely slow to converge with long periods of initial stagnation. For $b$AMG, stagnation is permanent. By contrast, monolithic preconditioners converge robustly and at least an order magnitude faster. This observation is insensitive to whether block preconditioners are Stokes-based or Navier-Stokes-based. In \cite{mehmani2025multiscale}, the authors observed better performance by $b$PLMM and $b$PNM for the Stokes equations than what is observed here for the Navier-Stokes equations. However, even there, $a$PLMM and $a$PNM were superior, and $b$AMG was prone to stagnation, consistent with Fig.\ref{fig:qa_result_conv}. Table \ref{tab:pd4_time} lists the total wall-clock times (WCTs) associated with the first three $Re$ increments of the PD4 domain, using different preconditioners. The WCTs include contributions from the costs of building $\mathrm{M_G}$, building $\mathrm{M_L}$, and the self-time of GMRES. We see that block preconditioners take roughly 1-2 orders of magnitude longer than the monolithic preconditioners. Hence, we exclude $b$PLMM, $b$PNM, and $b$AMG from all subsequent analysis.

\begin{figure} [b!]
  \centering
  \centerline{\includegraphics[scale=0.39, trim={0cm 0.0cm 0cm 0cm},clip]{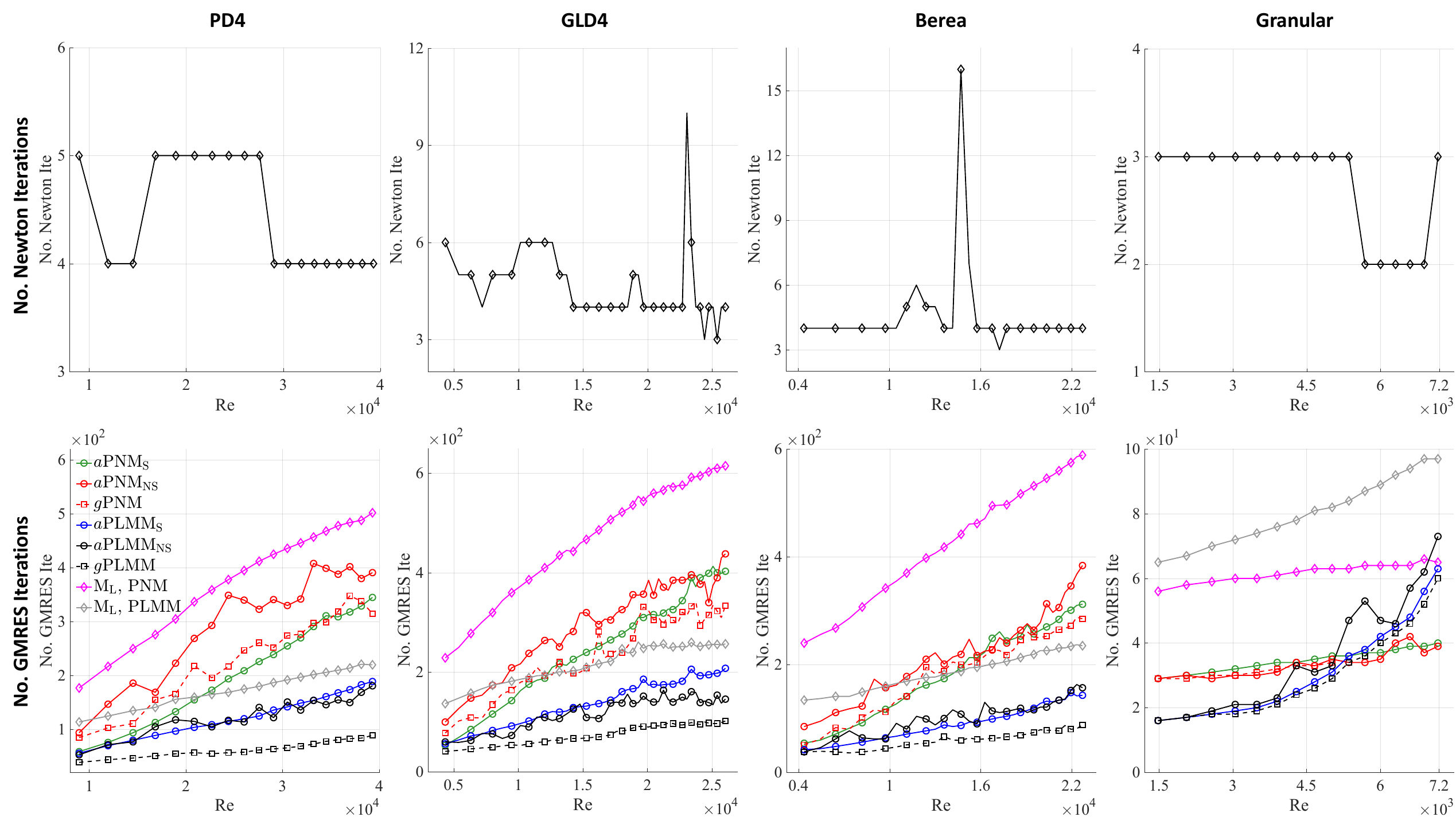}}
  \caption{(Top row) Number of Newton iterations versus $Re$. (Bottom row) Number of GMRES iterations of the first Newton step versus $Re$. Results belong to different preconditioners based on PLMM and PNM applied to the \textit{steady-state Navier-Stokes equations} in all domains. Prefix $a$ means ``algebraic,'' and $g$ means ``geometric.'' Subscripts ``S'' and ``NS'' imply preconditioners are built from Stokes or Navier Stokes, respectively.}
\label{fig:qa_result_ite}
\end{figure}

\begin{figure} [t!]
  \centering
  \centerline{\includegraphics[scale=0.39, trim={0cm 0.0cm 0cm 0cm},clip]{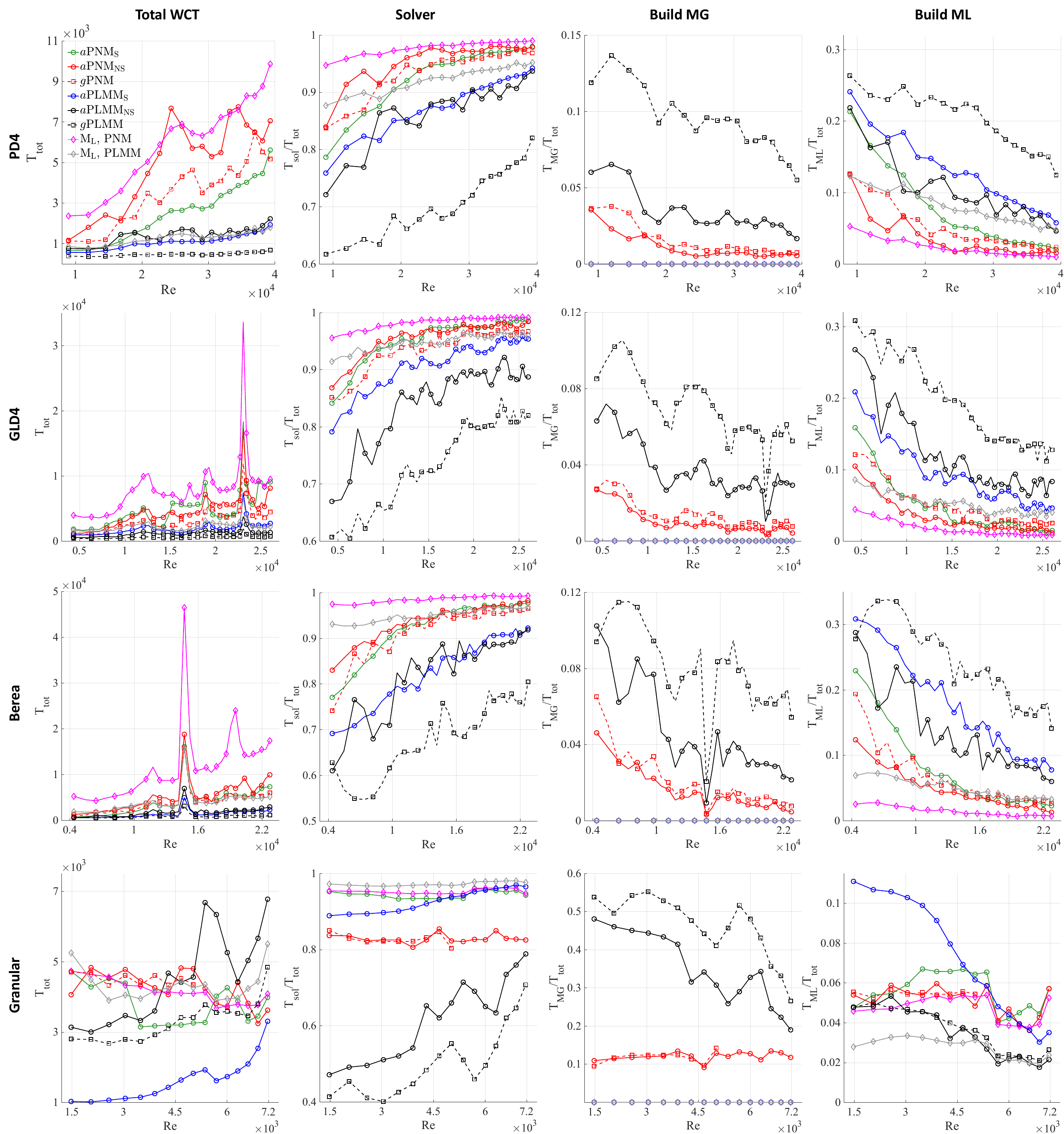}}
  \caption{Total wall-clock times (WCT; in seconds) spent solving the \textit{steady-state Navier-Stokes equations} ($\mathrm{T_{tot}}$) versus $Re$, for different preconditioners based on PLMM and PNM (built from Stokes and Navier Stokes) in all domains. $\mathrm{T_{tot}}$ consists of the costs of building $\mathrm{M_G}$ ($\mathrm{T_{MG}}$; if monolithic), building $\mathrm{M_L}$ ($\mathrm{T_{ML}}$), and self-time of GMRES ($\mathrm{T_{sol}}$). The breakdown of $\mathrm{T_{tot}}$ into $\mathrm{T_{MG}}$, $\mathrm{T_{ML}}$, and $\mathrm{T_{sol}}$ is illustrated in fractions of $\mathrm{T_{tot}}$.}
\label{fig:qa_result_wct}
\end{figure}

\begin{figure} [t!]
  \centering
  \centerline{\includegraphics[scale=0.4, trim={0cm 0.0cm 0cm 0cm},clip]{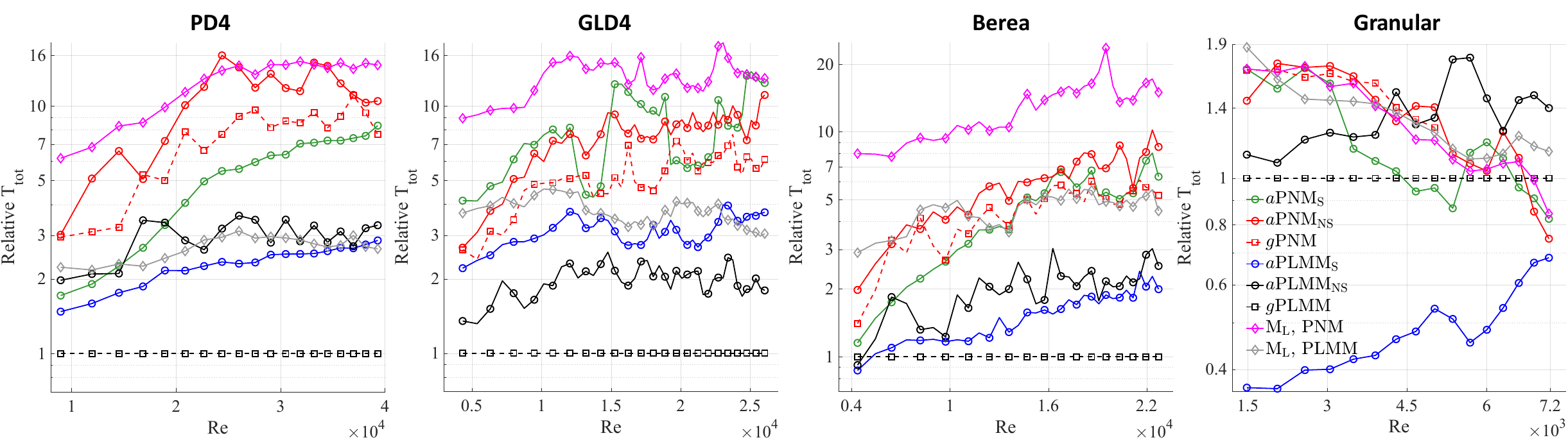}}
  \caption{Relative performance of all monolithic and smoother-only preconditioners in Fig.\ref{fig:qa_result_wct} measured against $g$PLMM for the \textit{steady-state Navier-Stokes equations}. Relative performance is defined as the ratio of total WCT ($\mathrm{T_{tot}}$) for the preconditioner in question over that of $g$PLMM.}
\label{fig:qa_result_performance}
\end{figure}

Fig.\ref{fig:qa_result_ite} shows the number of Newton and GMRES iterations as a function of $Re$ for the PD4, GLD4, Berea, and Granular domains. GMRES iterations correspond to the first Newton step of each $Re$, where GMRES struggles the most compared to subsequent Newton steps. Notice that the number of Newton iterations (between 2 to 6) is almost constant with respect to $Re$, which is a direct result of using the converged solution from the prior $Re$ as the initial guess for the current $Re$. By contrast, GMRES iterations increase steadily with $Re$, indicative of $\mathrm{\hat{A}}\hat{x}\!=\!\hat{b}$ becoming more challenging to solve at higher Reynolds numbers. \textit{The best-performing preconditioner is consistently $g$PLMM across all domains}, outperforming $a\mathrm{PLMM_{S/NS}}$ by up to a factor of $\sim$2, and $a\mathrm{PNM_{S/NS}}$ by up to a factor of $\sim$4. The only exception is Granular, where $g$PLMM and $a\mathrm{PLMM_{S/NS}}$ are comparable, and $a\mathrm{PNM_{S/NS}}$ slightly outperforms $g$PLMM at high (but not low) $Re$. Interestingly, $g$PNM performs as poorly as $a\mathrm{PNM_{S/NS}}$, implying that the closure BC $\partial_n\bs{u}\!=\!\bs{0}$ is only meaningful for $g$PLMM, but not $g$PNM. Moreover, comparing $a\mathrm{PLMM_{S}}$ to $a\mathrm{PLMM_{NS}}$, we see little difference in performance, implying the Stokes-based preconditioner is more attractive because it is built only once. As a general rule, preconditioners based on PLMM are seen to perform far better than those based on PNM.

Fig.\ref{fig:qa_result_ite} also includes the performance of GMRES preconditioned by only the smoothers of PLMM and PNM, i.e., $\mathrm{M_L}$ not paired with $\mathrm{M_G}$. We see the PLMM smoother performs better than the PNM smoother, but neither beats their respective monolithic preconditioners, where $\mathrm{M_L}$ is paired with $\mathrm{M_G}$. This confirms that such pairing is necessary.

Fig.\ref{fig:qa_result_wct} depicts total wall-clock times (WCT), $\mathrm{T_{tot}}$, associated with computations in Fig.\ref{fig:qa_result_ite} versus $Re$. This includes costs corresponding to building $\mathrm{M_G}$ ($\mathrm{T_{MG}}$), building $\mathrm{M_L}$ ($\mathrm{T_{ML}}$; here LU-decomposition of local systems in the additive-Schwarz smoothers), and self-time of GMRES ($\mathrm{T_{sol}}$). Results belong to the $a$PLMM, $a$PNM, $g$PLMM, $g$PNM, and $\mathrm{M_L}$ preconditioners. The breakdown of $\mathrm{T_{tot}}$ into $\mathrm{T_{MG}}$, $\mathrm{T_{ML}}$, and $\mathrm{T_{sol}}$ is shown as fractions of $\mathrm{T_{tot}}$ in Fig.\ref{fig:qa_result_wct}. We see in all cases, $\mathrm{T_{sol}}$ dominates $\mathrm{T_{tot}}$, making the relative ranking among preconditioners in terms of $\mathrm{T_{tot}}$ similar to that already discussed in Fig.\ref{fig:qa_result_ite} in terms of GMRES iterations. To quantify speedup, we normalize the $\mathrm{T_{tot}}$ of all preconditioners against that of $g$PLMM in Fig.\ref{fig:qa_result_performance}. Compared to $g$PLMM, $a\mathrm{PLMM_{S/NS}}$ are 2-3 times slower, $a\mathrm{PNM_{S/NS}}$ are $\sim$10 times slower, the $\mathrm{M_L}$ of PLMM is 3-5 times slower, and the $\mathrm{M_L}$ of PNM is 10-20 times slower. The only exception is 3D Granular, whose $\mathrm{T_{MG}}$ constitutes a sizable fraction of $\mathrm{T_{tot}}$ due to the presence of a few large primary grids, making the continuous updating of $\mathrm{M_G}$ versus $Re$ in $g$PLMM more costly than the one-time cost of construction in $a\mathrm{PLMM_{S}}$. Once again, $a\mathrm{PLMM_S}$ and $a\mathrm{PLMM_{NS}}$ have comparable $\mathrm{T_{tot}}$ across all domains, except Granular for reasons just stated. This makes $a\mathrm{PLMM_S}$ more attractive than $a\mathrm{PLMM_{NS}}$ as it is not rebuilt at every $Re$ increment.

To understand why $g$PLMM outperforms its algebraic counterparts $a\mathrm{PLMM_{S/NS}}$, and why PLMM-based preconditioners outperform PNM-based ones, we visualize the spectrum of the error-propagation matrix $\mathrm{E}$:
\begin{equation} \label{eq:err_prop}
	\mathrm{E_G} = \mathrm{I} - \mathrm{M_G^{-1}} \hat{\mathrm{A}}, \quad \mathrm{E_L} = \mathrm{I} - \mathrm{M_L^{-1}} \hat{\mathrm{A}} \quad \Rightarrow \quad \mathrm{E} = \mathrm{E_L} \mathrm{E_G}
\end{equation}
for $g\mathrm{PLMM}$, $g\mathrm{PNM}$, $a\mathrm{PLMM_{S/NS}}$, and $a\mathrm{PNM_{S/NS}}$ in Fig.\ref{fig:qa_result_spec}. Each spectrum consists of the 500 largest eigenvalues of $\mathrm{E}$ for the PD4 domain at $Re$ = $3.9\times 10^4$ (other domains yield similar results). We see that the spectrum of $g$PLMM is clustered closest to the origin of the complex plane and lies inside the unit circle, explaining its superior GMRES convergence. Next in ranking are $a\mathrm{PLMM_{S/NS}}$, which have similar performance. While $a\mathrm{PLMM_{NS}}$ is more clustered near the origin than $a\mathrm{PLMM_{S}}$, the former is plagued by rogue eigenvalues outside the unit circle slowing down convergence. Lastly, $a\mathrm{PNM_{S/NS}}$ and $g$PNM have the largest spectral radii, thus the slowest GMRES convergence.

In addition to $\mathrm{E}$'s spectrum, Fig.\ref{fig:qa_result_spec} shows the spectra of $\mathrm{E_G}$ and $\mathrm{E_L}$ defined by Eq.\ref{eq:err_prop}. These correspond to the error-propagation matrices of $\mathrm{M_G}$ and $\mathrm{M_L}$, respectively. Eigenvalues of $\mathrm{E_L}$ lie inside the unit circle for both the PLMM and PNM smoothers, with many very close to one, confirming the smoothers alone are inadequate as preconditioners. For $a\mathrm{PLMM_{NS}}$ and $a\mathrm{PNM_{NS}}$, Fig.\ref{fig:qa_result_spec} shows the eigenvalues of $\mathrm{E_G}$ consist of zero or one, because $\mathrm{E_G}$ is a projection by Proposition \ref{prp:project} and our simplification for $\mathrm{M_G}$ in Remark 1. For $g$PLMM, $g$PNM, $a\mathrm{PLMM_S}$, and $a\mathrm{PNM_S}$, $\mathrm{E_G}$'s eigenvalues lie outside the unit circle, because the $\mathrm{\hat{A}}$ used in Eq.\ref{eq:err_prop} is different from the Jacobian used to construct $\mathrm{M_G}$.
What is interesting is that the spectral radius of $\mathrm{E}$ for $g$PLMM and $a\mathrm{PLMM_S}$ (but not $g$PNM and $a\mathrm{PNM_S}$) is smaller than one, despite the spectral radius of $\mathrm{E_G}$ being larger than one. The opposite is true for $a\mathrm{PLMM_{NS}}$ and $a\mathrm{PNM_{NS}}$, where the spectral radius of $\mathrm{E}$ is larger than one despite both $\mathrm{E_G}$ and $\mathrm{E_L}$ having spectra contained within the unit disk. The latter is possible because both $\mathrm{E_G}$ and $\mathrm{E_L}$ are non-normal matrices, with non-orthogonal eigenvectors \cite{trefethen2020spectra}.

\begin{figure} [t!]
  \centering
  \centerline{\includegraphics[scale=0.42, trim={0cm 0.0cm 0cm 0cm},clip]{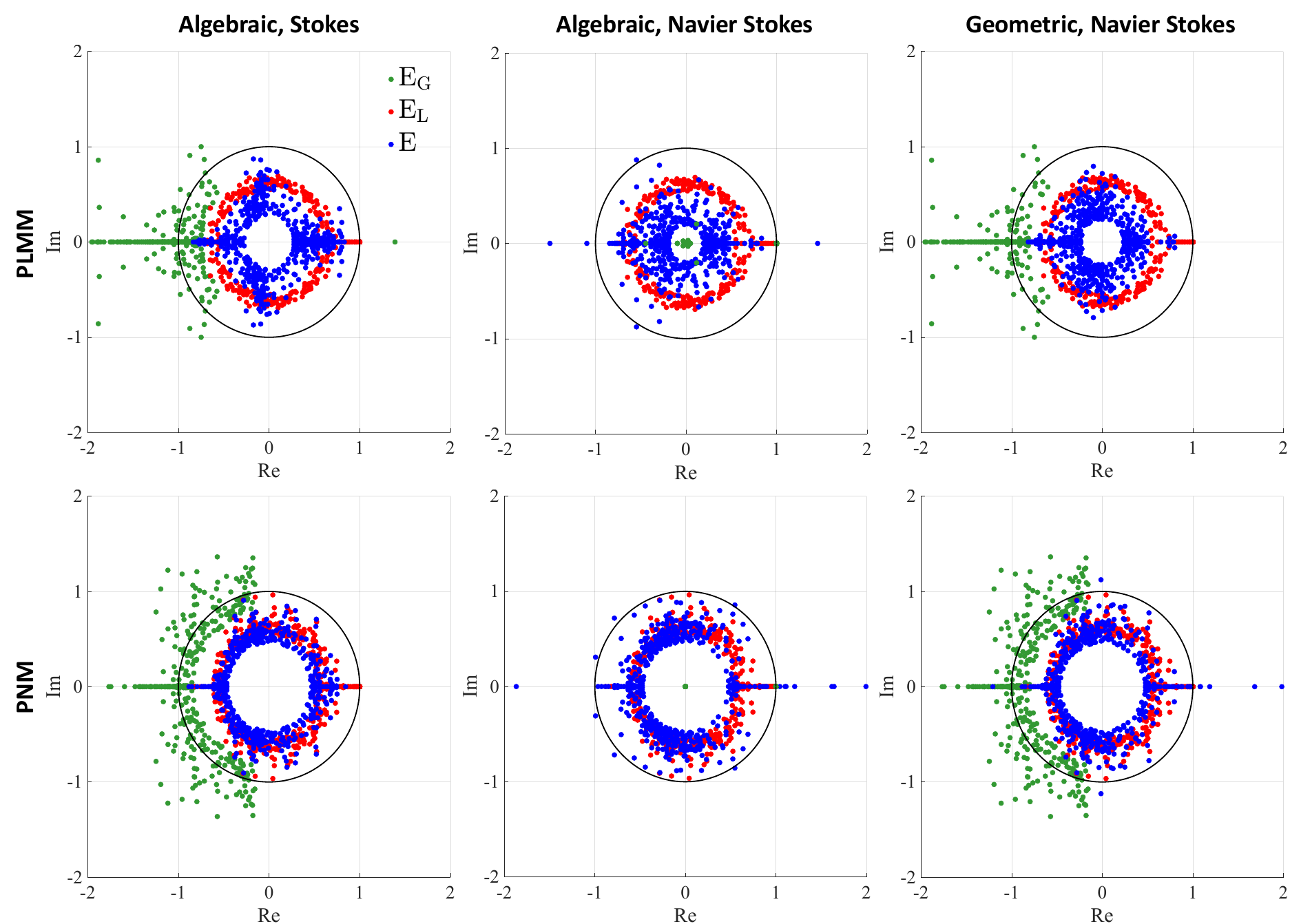}}
  \caption{Spectra (500 largest eigenvalues) of the error-propagation matrices $\mathrm{E}$ (blue dots), $\mathrm{E_G}$ (green dots), and $\mathrm{E_L}$ (red dots) in Eq.\ref{eq:err_prop}. Results belong to preconditioners $a\mathrm{PLMM_{S}}$ and $a\mathrm{PNM_{S}}$ (left column), $a\mathrm{PLMM_{NS}}$ and $a\mathrm{PNM_{NS}}$ (middle column), and $g$PLMM and $g$PNM (right column) applied to the \textit{steady-state Navier-Stokes equations} on the PD4 domain at $Re\!=\!3.9\times10^4$.}
\label{fig:qa_result_spec}
\end{figure}

\subsection{Unsteady-state Navier-Stokes equations} \label{sec:test_prec_dy}
We focus next on the unsteady-state Navier-Stokes equations and the turbulent flow regimes that result from it. To ensure flow is turbulent at all $t$, we set the IC equal to the steady-state solution from the previous section at the highest simulated $Re$. The BCs are identical to the steady-state simulations in Section \ref{sec:test_prec_st}, where the inlet pressure $p_{in}$ is incrementally increased while the outlet pressure is $p_{out}\!=\!0$. The only difference here is that a timescale is associated with the growth of $p_{in}$. Specifically, we take one $\Delta p_{in}$ increment per $\Delta t\! =\! 0.2$ sec, where $\Delta p_{in}\!=\!20$ dynes/cm\textsuperscript{2} in all 2D domains and $\Delta p_{in}\!=\!30$ dynes/cm\textsuperscript{2} in the 3D Granular domain. Within each 0.2 sec interval, during which $p_{in}$ is constant, smaller time steps $dt$ are taken to capture the turbulent details of flow. The $dt$ are taken adaptively, based on how fast the residual norm evolves. We test the performance of the $g$PLMM, $a\mathrm{PLMM_{S/NS}}$, and $a\mathrm{PNM_{S/NS}}$ preconditioners on the PD4, GLD4, Berea, and Granular domains (see Fig.\ref{fig:domains}). We exclude $g$PNM from consideration based on its poor performance in Section \ref{sec:test_prec_st}. Navier-Stokes-based preconditioners (all except $a\mathrm{PLMM}_{S}$ and $a\mathrm{PNM}_{S}$) are rebuilt at the \textit{beginning} of every $Re$ (or $p_{in}$) increment, not in between $dt$ steps or Newton iterations. Since dynamic simulations are computationally expensive, underperforming perconditioners are terminated prematurely whenever it is clear they are not competitive. Finally, GMRES iterations preconditioned by $a\mathrm{PNM_{S/NS}}$ diverge or stagnate in the Granular domain, even for very small $dt$. This is why $a\mathrm{PNM_{S/NS}}$ results for Granular appear absent from Figs.\ref{fig:dy_result_ite}-\ref{fig:dy_result_performance}.

To visualize the complex nature of turbulent flow, Fig.\ref{fig:dy_result_sol} depicts snapshots of fluid pressure, velocity magnitude, and vorticity magnitude ($|\nabla\times\bs{u}|$) at the highest simulated $Re$ in all domains. Compared to steady-state flow in Figs.\ref{fig:coarse_sol1}-\ref{fig:coarse_sol2}, we see many more circulation zones and chaotic patterns. Fig.\ref{fig:dy_result_ite} shows that the number of $dt$ time steps, taken within each $Re$ increment, increases with $Re$, indicating progressively smaller $dt$ is required to capture the growing intricacies of turbulent flow. The number of Newton and GMRES iterations per $dt$ step (averaged over each $Re$ increment) are also depicted as a function of $Re$. GMRES iterations correspond to the first Newton step of each $Re$, where GMRES struggles the most compared to later Newton steps. Similar to Section \ref{sec:test_prec_st}, Newton iterations are constant ($\sim$5) versus $Re$. However, unlike Section \ref{sec:test_prec_st} where GMRES iterations grow with $Re$, they remain nearly constant here. We attribute this to the small $dt$ used to discretize $\partial_t\bs{u}$ in Eq.\ref{eq:governing_eqs}, rendering the (1,1)-block of $\hat{\mathrm{A}}$ in Eq.\ref{eq:Ax=b} diagonally dominant. At the limit of large $dt$, the $\partial_t\bs{u}$ term becomes negligible and the steady-state results of Fig.\ref{fig:qa_result_ite} are recovered.

\begin{figure} [t!]
  \centering
  \centerline{\includegraphics[scale=0.43, trim={0cm 0.0cm 0cm 0cm},clip]{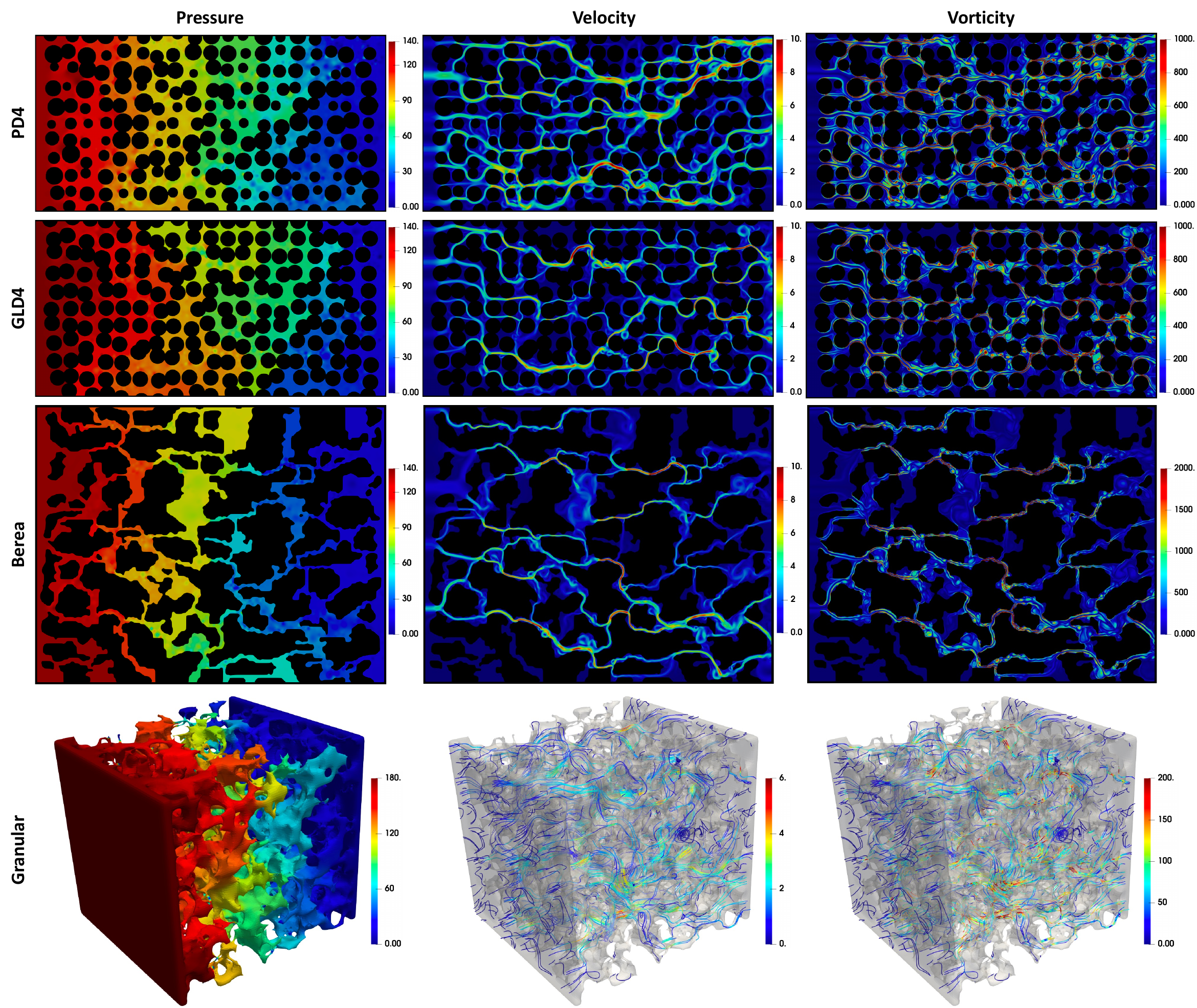}}
  \caption{Spatial distributions of fluid pressure, velocity magnitude, and vorticity magnitude for the \textit{unsteady-state Navier-Stokes equations} solved on the domains in Fig.\ref{fig:domains}, namely, PD4 at $Re$ = $1.3\times 10^5$, GLD4 at $Re \!=\! 5.5\times10^4$ , Berea at $Re \!=\! 4.4\times10^4$ , and Granular at $Re \!=\! 1.8\times10^4$.}
\label{fig:dy_result_sol}
\end{figure}

\begin{figure} [t!]
  \centering
  \centerline{\includegraphics[scale=0.39, trim={0cm 0.0cm 0cm 0cm},clip]{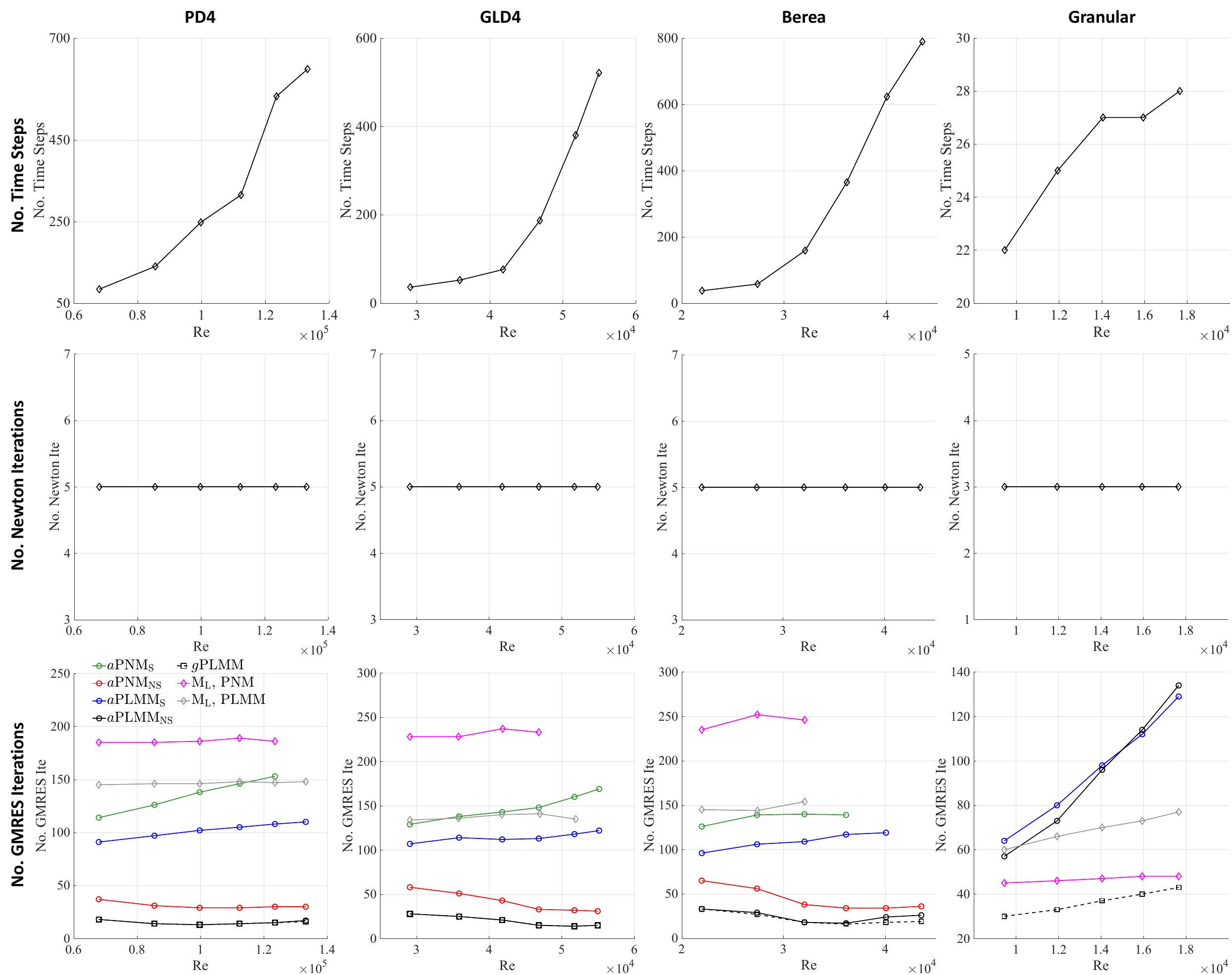}}
  \caption{(Top row) Number of time steps $dt$, taken within each $Re$ increment, versus $Re$. (Middle row) Number of Newton iterations versus $Re$, averaged over each $Re$ increment. (Bottom row) Number of GMRES iterations of the first Newton step versus $Re$, averaged over each $Re$ increment. Results belong to different preconditioners based on PLMM and PNM applied to the \textit{unsteady-state Navier-Stokes equations} in all domains. Prefix $a$ means ``algebraic,'' and $g$ means ``geometric.'' Subscripts ``S'' and ``NS'' imply preconditioners are built from Stokes or Navier Stokes, respectively.}
\label{fig:dy_result_ite}
\end{figure}

Similar to Section \ref{sec:test_prec_st}, Fig.\ref{fig:dy_result_ite} suggests that the \textit{best-performing preconditioner is $g$PLMM across all domains}. However, unlike Section \ref{sec:test_prec_st}, \textit{$aPLMM_{NS}$ is a very close second}, slower by only <1.5 times in 2D and <2.8 times in the 3D Granular. By contrast, $g$PLMM is faster than $a\mathrm{PNM_{NS}}$ by <2 times in 2D domains (3D diverged), and than $a\mathrm{PLMM_S}$, $a\mathrm{PNM_S}$, and PLMM/PNM smoothers by >10 times in 2D and <2 times in 3D Granular. Clearly, preconditioners based on Navier Stokes ($g$PLMM, $a\mathrm{PLMM_{NS}}$, $a\mathrm{PNM_{NS}}$) outperform those based on Stokes ($a\mathrm{PLMM_{S}}$, $a\mathrm{PNM_{S}}$) by about an order of magnitude. This implies that rebuilding the monolithic preconditioners at the start of every $Re$ increment is necessary to efficiently model turbulent flow. Another clear observation is that preconditioners based on PLMM perform twice as fast as those based on PNM. Lastly, smoother-only preconditioners are extremely slow (except in Granular) indicating a pairing with a coarse preconditioner $\mathrm{M_G}$ is necessary for rapid convergence.

Fig.\ref{fig:dy_result_wct} plots the total wall-clock times ($\mathrm{T_{tot}}$) spent on computations versus $Re$ for all domains. The WCTs of all $dt$ steps taken within each $Re$ increment are summed for better representation of cost. The overall rise in $\mathrm{T_{tot}}$ with $Re$ is due to the need for smaller $dt$ in simulating turbulent flow at higher $Re$. By the same token, $\mathrm{T_{sol}}$ also increases with $Re$ but the associated cost-per-$dt$ is roughly constant, as evidenced by the constant number of GMRES iterations versus $Re$ in Fig.\ref{fig:dy_result_ite}. Unlike Section \ref{sec:test_prec_st}, $\mathrm{T_{tot}}$ here is dominated by the cost of building the smoother ($\mathrm{T_{ML}}$; LU-decomposition of local systems in the additive-Schwarz smoothers), especially for $g$PLMM. The cost of rebuilding $\mathrm{M_G}$ is negligible in contrast. To better visualize speedup, Fig.\ref{fig:dy_result_performance} normalizes the $\mathrm{T_{tot}}$ of all preconditioners against that of $g$PLMM. Once again, we see that the performance of $a\mathrm{PLMM_{NS}}$ is on par with $g$PLMM in 2D but slower by <3.4 times in 3D Granular. Moreover, $g$PLMM is faster than $a\mathrm{PNM_{NS}}$ by <3 times, $a\mathrm{PLMM_{S}}$ by <6 times, $a\mathrm{PNM_{S}}$ by <9 times, PLMM smoother by <15 times, and the PNM smoother by <30 times. Other trends regarding the relative ranking of preconditioners are consistent with observations already discussed in relation to GMRES iterations in Fig.\ref{fig:dy_result_ite}.

Fig.\ref{fig:dy_result_spec} visualizes the 500 largest eigenvalues of the error-propagation matrix $\mathrm{E}$ in Eq.\ref{eq:err_prop} for the PD4 domain at $Re\!=\!1.3\times 10^5$. 
Included is also the spectrum of $\mathrm{E_L}$, but not $\mathrm{E_G}$ because most of its eigenvalues for $g$PLMM, $g$PNM, $a\mathrm{PLMM_{S}}$ and $a\mathrm{PNM_{S}}$ lie far outside the unit circle. However, the $\mathrm{E_G}$ of $a\mathrm{PLMM_{NS}}$ and $a\mathrm{PNM_{NS}}$ has eigenvalues that are either zero or one (see discussion in Section \ref{sec:test_prec_st}).
The spectral radius of $\mathrm{E}$ for $g$PLMM and $a\mathrm{PLMM_{NS}}$ are similar, which explains their comparable GMRES convergence in Fig.\ref{fig:dy_result_ite}. The same observation holds for $g$PNM versus $a\mathrm{PNM_{NS}}$. Notice also that the Navier-Stokes-based preconditioners ($a\mathrm{PLMM_{NS}}$/$a\mathrm{PNM_{NS}}$) have significantly reduced spectral radii than Stokes-based preconditioners ($a\mathrm{PNM_S}$/$a\mathrm{PLMM_S}$), resulting in faster GMRES convergence. Finally, the spectral radius of $\mathrm{E_L}$ is almost one, confirming smoothers alone are inadequate for preconditioning GMRES.

\begin{figure} [h!]
  \centering
  \centerline{\includegraphics[scale=0.4, trim={0cm 0.0cm 0cm 0cm},clip]{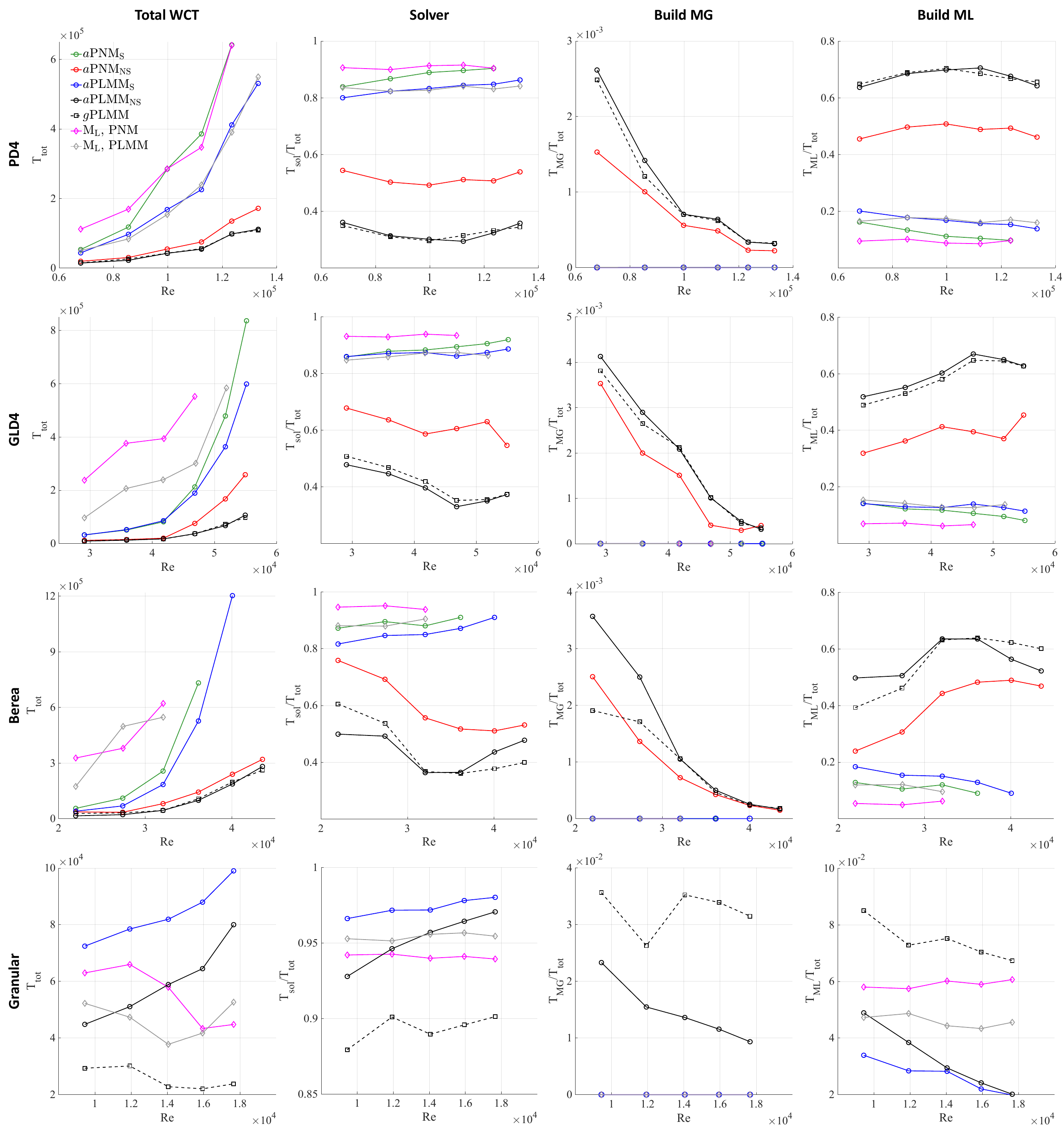}}
  \caption{Total wall-clock times (WCT; in seconds) spent solving the \textit{unsteady-state Navier-Stokes equations} ($\mathrm{T_{tot}}$) versus $Re$, for different preconditioners based on PLMM and PNM (built from Stokes and Navier Stokes) in all domains. $\mathrm{T_{tot}}$ consists of the costs of building $\mathrm{M_G}$ ($\mathrm{T_{MG}}$; if monolithic), building $\mathrm{M_L}$ ($\mathrm{T_{ML}}$), and self-time of GMRES ($\mathrm{T_{sol}}$). The breakdown of $\mathrm{T_{tot}}$ into $\mathrm{T_{MG}}$, $\mathrm{T_{ML}}$, and $\mathrm{T_{sol}}$ is illustrated in fractions of $\mathrm{T_{tot}}$.}
\label{fig:dy_result_wct}
\end{figure}

\begin{figure} [h!]
  \centering
  \centerline{\includegraphics[scale=0.4, trim={0cm 0.0cm 0cm 0cm},clip]{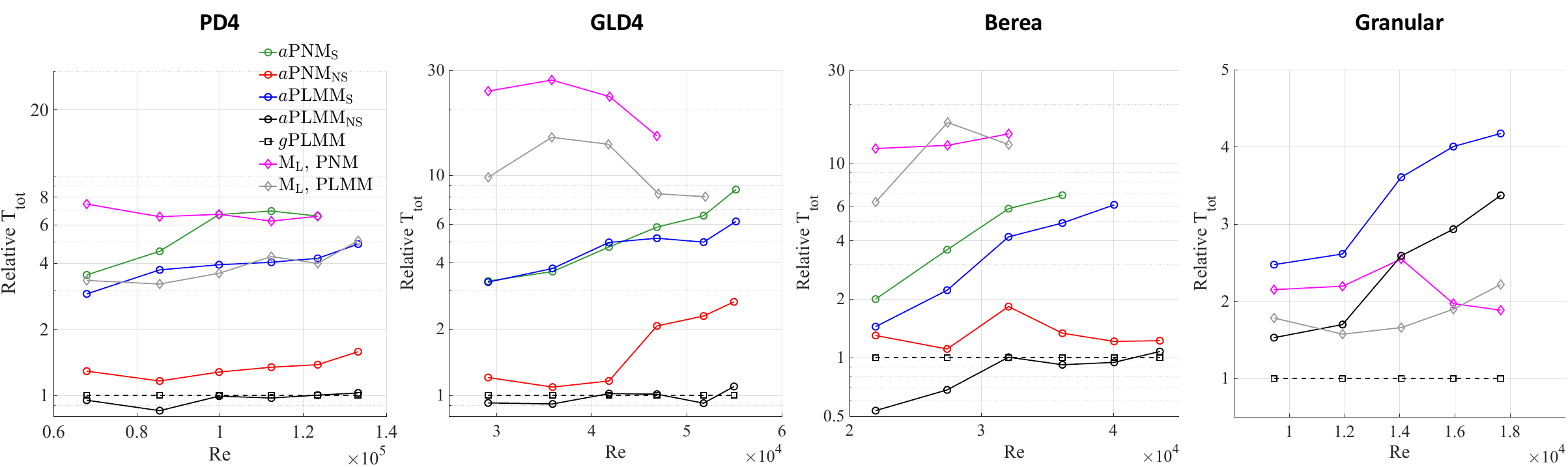}}
  \caption{Relative performance of all monolithic and smoother-only preconditioners in Fig.\ref{fig:dy_result_wct} measured against $g$PLMM for the \textit{unsteady-state Navier-Stokes equations}. Relative performance is defined as the ratio of total WCT ($\mathrm{T_{tot}}$) for the preconditioner in question over that of $g$PLMM.}
\label{fig:dy_result_performance}
\end{figure}

\begin{figure} [h!]
  \centering
  \centerline{\includegraphics[scale=0.41, trim={0cm 0.0cm 0cm 0cm},clip]{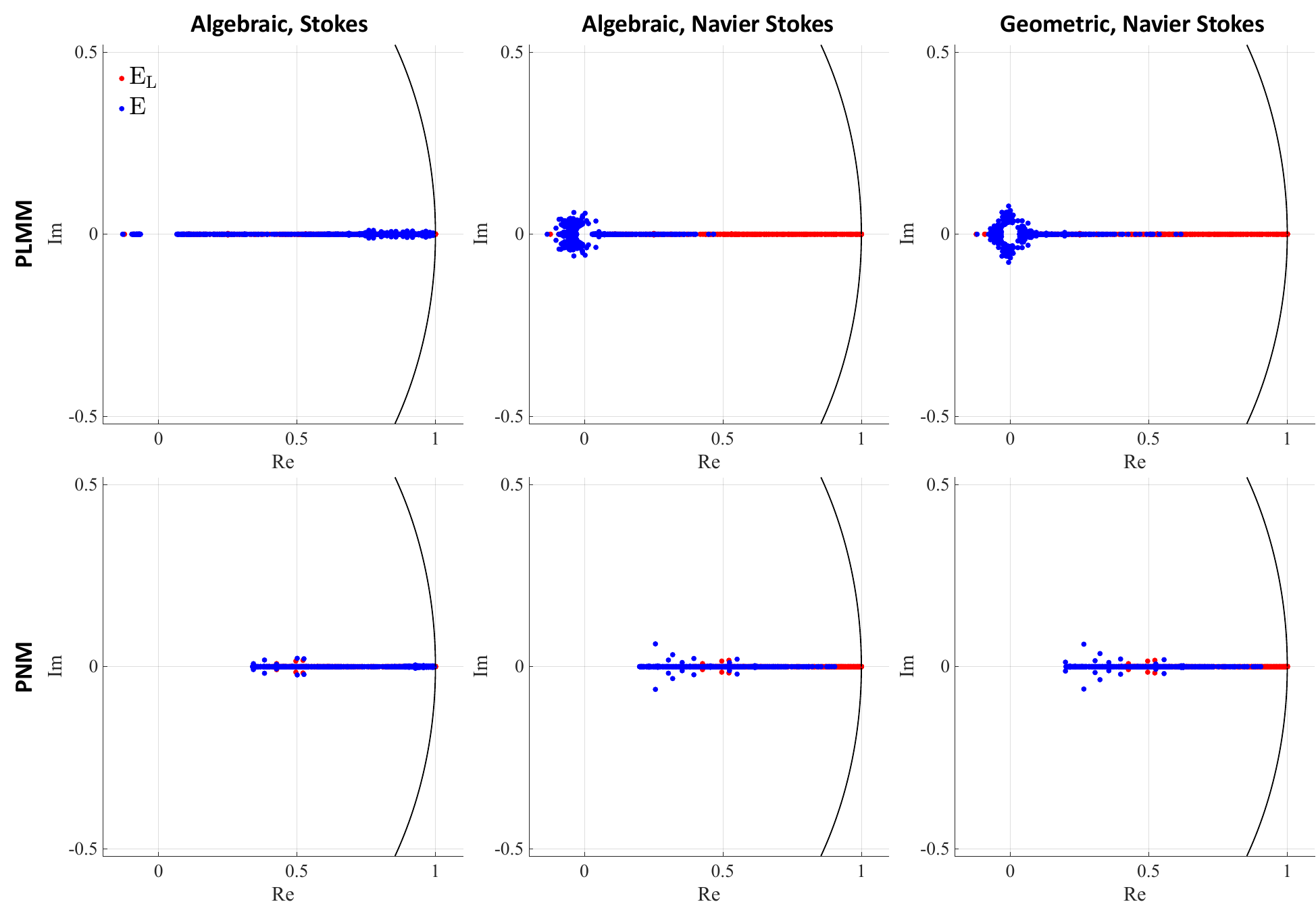}}
  \caption{Spectra (500 largest eigenvalues) of the error-propagation matrices $\mathrm{E}$ (blue dots) and $\mathrm{E_L}$ (red dots) in Eq.\ref{eq:err_prop}. Results belong to preconditioners $a\mathrm{PLMM_{S}}$ and $a\mathrm{PNM_{S}}$ (left column), $a\mathrm{PLMM_{NS}}$ and $a\mathrm{PNM_{NS}}$ (middle column), and $g$PLMM and $g$PNM (right column) applied to the \textit{unsteady-state Navier-Stokes equations} on the PD4 domain at $Re\!=\!1.3\times10^5$.}
\label{fig:dy_result_spec}
\end{figure}

\section{Discussion} \label{sec:discussion}
\subsection{Coarse-scale solver}
We presented a coarse-scale solver for computing approximate solutions to the steady-state Navier-Stokes equations in Algorithm \ref{alg:coarse_scale}, utilizing operators that comprise $\mathrm{M_G}$. We found that the $\mathrm{M_G}$ of $g$PLMM yields far more accurate results than that of $g$PNM. Approximate solutions so obtained are sufficiently accurate ($<30\%$ error in permeability) for use in a wide range of manufacturing (e.g., flow battery) and subsurface (e.g., CO\textsubscript{2} storage) applications. But as $Re$ grows, accuracy degrades because the closure $\partial_n\bs{u}\!=\!\bs{0}$ becomes invalid at contact interfaces. Fig.\ref{fig:coarse_streamline} demonstrates this, where unlike Stokes flow, streamlines no longer strike normal to the interface at high $Re$ due to the formation of vortices. Recall a key step in imposing the velocity closure was to remove the inertia term from the residual of faces near each interface (Fig.\ref{fig:coarse_sketch}), so the column-sum operator $\mathcal{C}(\cdot)$ (in building $\mathrm{\hat{P}}$) imposes $\partial_n\bs{u}\!=\!\bs{0}$ exactly. Since all computations in Algorithm 1 involve only coarse-scale pressure unknowns defined on contact interfaces, with no recourse to the fine grid, the coarse-scale solver is much cheaper than computing solutions directly on the fine grid. Specifically, the $\mathrm{T_{sol}}$ and $\mathrm{T_{ML}}$ costs in Fig.\ref{fig:qa_result_wct} are almost absent, rendering the coarse-scale solver faster by a factor of $\sim$10 in 2D and $\sim$2.5 in 3D Gyroid. However, Newton iterations of the coarse solver are larger by a factor of 2--5.

Two open questions remain: (1) Is there a more accurate closure at high $Re$ and how does one impose it? (2) How can the coarse solver be generalized to time-dependent flows? The answer to (1) is unknown to to the authors, but we suspect there is an opportunity to use mortars (similar to \cite{khan2025hPLMM}) to increase the flexibility of capturing arbitrary velocities along each interface. This would require building additional shape vectors (Eq.\ref{eq:prolong_1}) and solving a larger coarse system (Eq.\ref{eq:coarse_scale_sys}), as was done in \cite{khan2025hPLMM} for elliptic equations. Whether the associated costs would render such an approach prohibitive remains unknown. One answer to (2) involves discretizing $\partial_t\bs{u}$ in the Navier-Stokes equations with a fixed time step $dt$, before building $\mathrm{M_G}$. This embeds $dt$ into the shape and correction vectors, which in turn can be used to march forward in time in increments of $dt$ (similar to \cite{mehmani2018mult}). The drawback is committing to one $dt$ for all $t$.

\subsection{Recommended preconditioners}
The best-performing preconditioner for both steady-state and unsteady-state Navier-Stokes equations is $g$PLMM, consistently across all domains studied in Section \ref{sec:test_prec}. This, incidentally, coincides with the more accurate coarse-scale solver based on $g$PLMM. The downside of $g$PLMM, however, is its need for geometric information, or more specifically, access to the code that can assemble the modified residual and Jacobian ($\tilde{r}$ and $\tilde{J}$ from Eq.\ref{eq:mod_res}). Such access is needed to remove the inertia term from face-residuals near contact interfaces to impose $\partial_n\bs{u}\!=\!\bs{0}$. In the absence of such access, algebraic preconditioners are the only option, with $a\mathrm{PLMM_S}$ recommended for the steady-state and $a\mathrm{PLMM_{NS}}$ for the unsteady-state Navier-Stokes equations. The former has the benefit of being built only once and reused for all $Re$. We remind that block preconditioners ($b$AMG, $b$PLMM, and $b$PNM) were extremely prohibitive and eliminated immediately from further consideration. PNM-based preconditioners ($g$PNM, $a\mathrm{PNM_{S/NS}}$) also exhibited very poor performance, indicating PLMM-based decomposition and closure are more accurate. Finally, smoothers alone are inadequate as preconditioners regardless of the steady-state or time-dependent form of the equations.

\subsection{Computational complexity}
While a serial machine was used here to perform all numerical tests, building and applying the coarse preconditioner $\mathrm{M_G}$ and the smoother $\mathrm{M_L}$ in all the proposed PLMM/PNM-based preconditioners are parallelizable. The cost of building $\mathrm{M_G}$ is dominated by computing shape vectors, $p^{p_i}_{c_j}$, and correction vectors, $c^{p_i}$ in Eq.\ref{eq:prolong_1}, whose computations are fully decoupled across all primary grids. Similarly, applying the additive-Schwarz smoother $\mathrm{M_L}$ requires solving a set of decoupled local systems on primary and dual grids, both amenable to parallelism. In the following, we briefly analyze the computational complexity and parallel scalability of constructing and applying $\mathrm{M_G}$ and $\mathrm{M_L}$.

Let $\Omega$ consist of $N_f$ fine-scale unknowns, $N^p$ primary grids, and $N^d$ (= $N^c$) dual grids. Let all linear(ized) systems (global and local) be solved using a linear solver that scales as $O(n^{\beta})$, where $n$ is the size of the linear system and $\beta \in (1,3)$. The wall-clock times (WCTs) associated with building and applying $\mathrm{M_G}$ and $\mathrm{M_L}$ are:
\begin{subequations}
\label{eq:cost_analysis}
\begin{align}
	\mathcal{T}_{build}^{\mathrm{M_G}} & = O(N_f / N^p)^{\beta} \times (2N^d+N^p) / N_{prc} \label{eq:cost_build_mg}\\
	\mathcal{T}_{apply}^{\mathrm{M_G}} &= O(N^d+N^p)^{\beta} \label{eq:cost_apply_mg}\\
	\mathcal{T}_{apply}^{\mathrm{M_L}} &= O\left[(N_f / N^p)^{\beta} \times N^p / N_{prc} + (N_f f^d/ N^d)^{\beta} \times N^d / N_{prc}\right] \label{eq:cost_apply_ml}
\end{align}
\end{subequations}
where $N_{prc}$ is the number of processors employed and $f^d$ is the fraction of fine-scale unknowns associated with dual grids. Recall dual grids cover only a small portion of $\Omega$ (10--30\%). Focusing on Eq.\ref{eq:cost_build_mg}, the average cost of solving one local system on a primary grid is $O(N_f / N^p)^{\beta}$, of which $2N^d$ must be solved for shape vectors and $N^p$ for correction vectors to build $\mathrm{M_G}$. These computations can all be performed in parallel, entailing $N_{prc}$ can equal $2N^d + N^p$ at maximum. Focusing next on Eq.\ref{eq:cost_apply_mg}, one application of $\mathrm{M_G}$ requires solving a small coarse system of size $N^d+N^p$ (Eq.\ref{eq:coarse_prob}), whose cost is negligible compared to other overheads. As for Eq.\ref{eq:cost_apply_ml}, the cost of applying $\mathrm{M_L}$ consists of solving $N^p$ systems on primary grids and $N^d$ systems on dual grids. If no pre-processing is done, the cost of building $\mathrm{M_L}$ is zero and all cost is due to application. But since applying $\mathrm{M_L}$ in iterative solvers requires repeated solves the same local systems, LU-decomposition during a pre-processing stage can significantly speed up $\mathrm{M_L}$'s application. The construction cost of $\mathrm{M_L}$ then corresponds to these LU-decompositions, and the application cost in Eq.\ref{eq:cost_apply_ml} is lowered (by reducing $\beta$). Both the construction and application of $\mathrm{M_L}$ can be parallelized with $N_{prc}$ up to $\max\{N^p,N^d\}$.

\subsection{Challenges beyond linear preconditioning}
Solving linear systems like Eq.\ref{eq:Ax=b} during Newton iterations is the slowest step in solving the Navier-Stokes equations, which we focused on herein.  With our best preconditioner ($g$PLMM), the cost of steady-state simulations is dominated by GMRES's self-time ($\mathrm{T_{sol}}$), whereas that of dynamic simulations is dominated by $\mathrm{M_L}$'s build-time ($\mathrm{T_{ML}}$). Further speedup requires advances in both solvers and smoothers, not coarse preconditioners. Potential smoothers that remain untested include block ILU($k$), block Gauss-Seidel, and block additive-Schwarz based on PLMM or other.

Beyond linear preconditioning, the biggest challenge with solving the Navier-Stokes equations lies with Newton convergence. If the initial guess is far from the solution, Newton is very prone to divergence at high $Re$. What we have done here to reach high $Re$ is essentially a method called \textit{homotopy continuation} \cite{watson1990homotopy}, where instead of solving the residual $r(\hat{x}; Re)$ in Eq.\ref{eq:Ax=b} directly at $Re$ (the semicolon separates variable from parameter), an augmented residual $H(\hat{x},\lambda;Re)\!:=\!r(\hat{x};\lambda Re)$ is solved for gradually increasing values of $\lambda\!\in\!(0,1)$. Notice $H(\hat{x},0;Re)$ is an easy problem because it corresponds to Stokes flow, whereas $H(\hat{x},1;Re)\!=\!r(\hat{x}; Re)$ is difficult as it corresponds to Navier-Stokes flow at the desired high $Re$. The approach we took to solve $H(\hat{x},\lambda;Re)$ and to increase $\lambda$ was to increment the inlet pressure $p_{in}$ gradually from a low to a high value. This, in turn, implied an increment in $Re$, or equivalently $\lambda$. After solving $H(\hat{x},\lambda;Re)$ for a given $\lambda$, the solution was used as the initial guess for Newton in the next $\lambda$. While effective, our approach is computationally wasteful because $H(\hat{x},\lambda;Re)$ is solved to completion at each fixed $\lambda$. Better schemes, where $\lambda$ is updated simultaneously as $\hat{x}$ in a single Newton loop, exist and would likely be much more efficient \cite{jiang2018homotopy, younis2010homotopy}.

The above continuation approach was found essential for the steady-state flow equations. However, for unsteady-state flow, it is needed only if the time step $dt$ used to discretize the $\partial_t\bs{u}$ term in Eq.\ref{eq:governing_eqs} is large. Note that at the limit $dt\!\rightarrow\!\infty$, the discrete unsteady-state system reduces to state steady. But increasing $dt$ too much introduces numerical diffusion into the solution, resulting in the inability to capture intricate details of the turbulent velocity field. In other words, the physics of the problem itself places an upper bound on $dt$, which is stricter at higher $Re$. If this bound is to be honored in favor of capturing the solution details accurately, then $dt$ must necessarily be small and the system $\hat{\mathrm{A}}\hat{x}\!=\!\hat{b}$ in Eq.\ref{eq:Ax=b} becomes diagonally dominant with respect to its (1,1)-block. The latter helps Newton converge with almost \textit{any} initial guess, but practically, it is natural to use the solution from the prior time step, as was done here.

\subsection{Future extensions}
While our preconditioners target the Navier-Stokes equations for Newtonian fluid flow in a fixed geometry, their applicability extends to more challenging problems. In immiscible two-phase flow, governing equations consist of a momentum balance and a phase-evolution equation, solved in staggered fashion to march forward in time. The former is essentially Navier Stokes plus a source term for interfacial tension, whose solution is the main computational bottleneck \cite{mehmani2019mult}. We expect our preconditioners to accelerate solving the momentum equation without any modification to $\mathrm{M_G}$, as the source term will likely introduce high-frequency variations to the flow field. The phase-evolution equation is hyperbolic and discretized explicitly in time. It is, therefore, much cheaper to solve and errors are high-frequency in nature, which can be removed by a smoother alone (no $\mathrm{M_G}$; similar to the phase-field equation in fracture mechanics \cite{li2023smooth}). Another class of candidate problems exhibit evolving geometry, due either to mineral dissolution/precipitation \cite{noiriel2016precip}, biofilm growth \cite{aufrecht2019biofilm}, or phase change \cite{moure2024phase}. If the representation of $\Gamma_w$ is sharp (as is here), the decomposition of $\Omega$ into primary and dual grids must be updated periodically, which can be done locally and adaptively. But if $\Gamma_w$ is diffuse, as in the Brinkman-Stokes model \cite{soulaine2016dbs}, extra source terms in Eq.\ref{eq:navstokes_mom} capture geometric changes that get embedded into the shape/correction vectors of $\mathrm{M_G}$. Finally, the flow of non-Newtonian fluids (e.g., blood, polymer) may benefit from our preconditioners if applied to linearized forms of the governing equations \cite{crochet2012nonNewt}.

\section{Conclusions} \label{sec:conclusion}
We present monolithic preconditioners to accelerate the convergence of linear solvers applied to the Navier-Stokes equations in arbitrarily complex porous microstructues. They are based on the pore-level multiscale method (PLMM) and the pore network model (PNM), and they generalize previously proposed preconditioners for the Stokes equations \cite{mehmani2025multiscale}. The preconditioners are two-level, consisting of a coarse preconditioner $\mathrm{M_G}$ and a smoother $\mathrm{M_L}$, designed to attenuate low- and high-frequency errors, respectively. To build both, the computational domain is decomposed into subdomains (primary and dual grids), over which local problems are solved. We formulate algebraic ($a$PLMM and $a$PNM) and geometric ($g$PLMM and $g$PNM) versions of each preconditioner, with the key difference being that the closure BC $\partial_n\bs{u}\!=\!\bs{0}$ is imposed exactly on subdomain interfaces in the geometric form but only approximately in the algebraic form. Hence, geometric preconditioners perform better but require access to code assembling the Jacobian to modify it slightly. Algebraic preconditioners need no such access and operate directly on the linear system.

We test our preconditioners against a wide range of 2D and 3D porous geometries, different $Re$, and steady-state and time-dependent forms of the Navier-Stokes equations.We also benchmark them against state-of-the-art block preconditioners $b$AMG, $b$PLMM, and $b$PNM. The best-performing preconditioner across the board was $g$PLMM. But since it requires intrusive access to code, comparable performance is possible with $a\mathrm{PLMM_S}$ for steady-state Navier Stokes and $a\mathrm{PLMM_{NS}}$ for time-dependent Navier Stokes. Subscripts ``S'' and ``NS'' indicate the preconditioners are built from the Jacobians of the Stokes and Navier-Stokes equations. Given the Stokes equations are linear, $a\mathrm{PLMM_S}$ need only be built once for usage at other $Re$, whereas $a\mathrm{PLMM_{NS}}$ (and $g$PLMM) must be built periodically at a nearby $Re$ if $Re$ changes, to retain optimal performance. Block and PNM-based preconditioners are far from competitive.

Lastly, we formulate a coarse-scale solver based on $g$PLMM for the steady-state Navier-Stokes equations, which computes approximate solutions rapidly. Acceleration is achieved by conducting all computations on coarse-scale pressure unknowns defined at subdomain interfaces, without any recourse to the fine grid. The coarse solution can be mapped onto the fine grid using a prolongation operator inherent to $g$PLMM. Approximate solutions so obtained have acceptable accuracy for a range of applications (permeability error $<30$\%). Our coarse solver is more accurate than classical PNMs designed to capture inertial flow in pore-scale domains, simplified by a ball-and-stick representation.

\section*{Acknowledgments}
This material is based upon work supported by the National Science Foundation under Grant No. CMMI-2145222. We acknowledge the Institute for Computational and Data Sciences (ICDS) at Penn State University for access to computational resources.

\appendix
\setcounter{figure}{0}
\setcounter{table}{0}
\section{Validating the fine-grid solver} \label{sec:appendix_a}

We validate the fine-grid solver used to assemble the residual and Jacobian of the Navier-Stokes equations, thus produce all results in this work. We consider fluid flow around circular and square obstacles at different $Re$. Fig.\ref{fig:benchmark_sol} shows the 2D domains, which consist of rectangles of length $L\!=\! 3$ cm and height $H\!=\!2$ cm. The centroids of the circle and square are positioned at half-height and a distance of 1 cm from the left boundary. The side length of the square and diameter of the circle are both $D\!=\!0.2$ cm. The BCs consists of: (1) inflow on the left ($\bs{u}\!=\!(10,0)$ and $\partial_n p\!=\!0$); (2) outflow on the right ($\partial_n \bs{u}\!=\!\bs{0}$ and $p\!=\!0$); (3) slip on top and bottom ($\partial_n\bs{u}\!=\!\bs{0}$ and $\partial_n p\!=\!0$); and (4) no-slip on circle and square perimeters ($\bs{u}\!=\!\bs{0}$ and $\partial_n p\!=\!0$). Fig.\ref{fig:benchmark_sol} shows the pressure, velocity magnitude, and vorticity magnitude over the whole domain at a low and high $Re$. A von Karman vortex street is seen to form in the latter.

To validate the fine-scale solver, we compute drag coefficient:
\begin{equation} \label{eq:drag}
	C_D = \frac{2 F_d}{\rho u_{in}^2 D}
\end{equation}
on the obstacles at different $Re\!=\!\rho u_{in} D/\mu$ and compare them to literature values in Table \ref{tab:drag_cof}. In Eq.\ref{eq:drag}, $F_d$ is the drag force exerted on the obstacles, and $u_{in}\!=\!10$ is the inlet velocity imposed in the horizontal direction. Table.\ref{tab:drag_cof} shows the $C_D$ computed here agrees with many other reported values in the literature for $Re$ between 20 and 300.

\begin{figure} [h!]
  \centering
  \centerline{\includegraphics[scale=0.43, trim={0cm 0.0cm 0cm 0cm},clip]{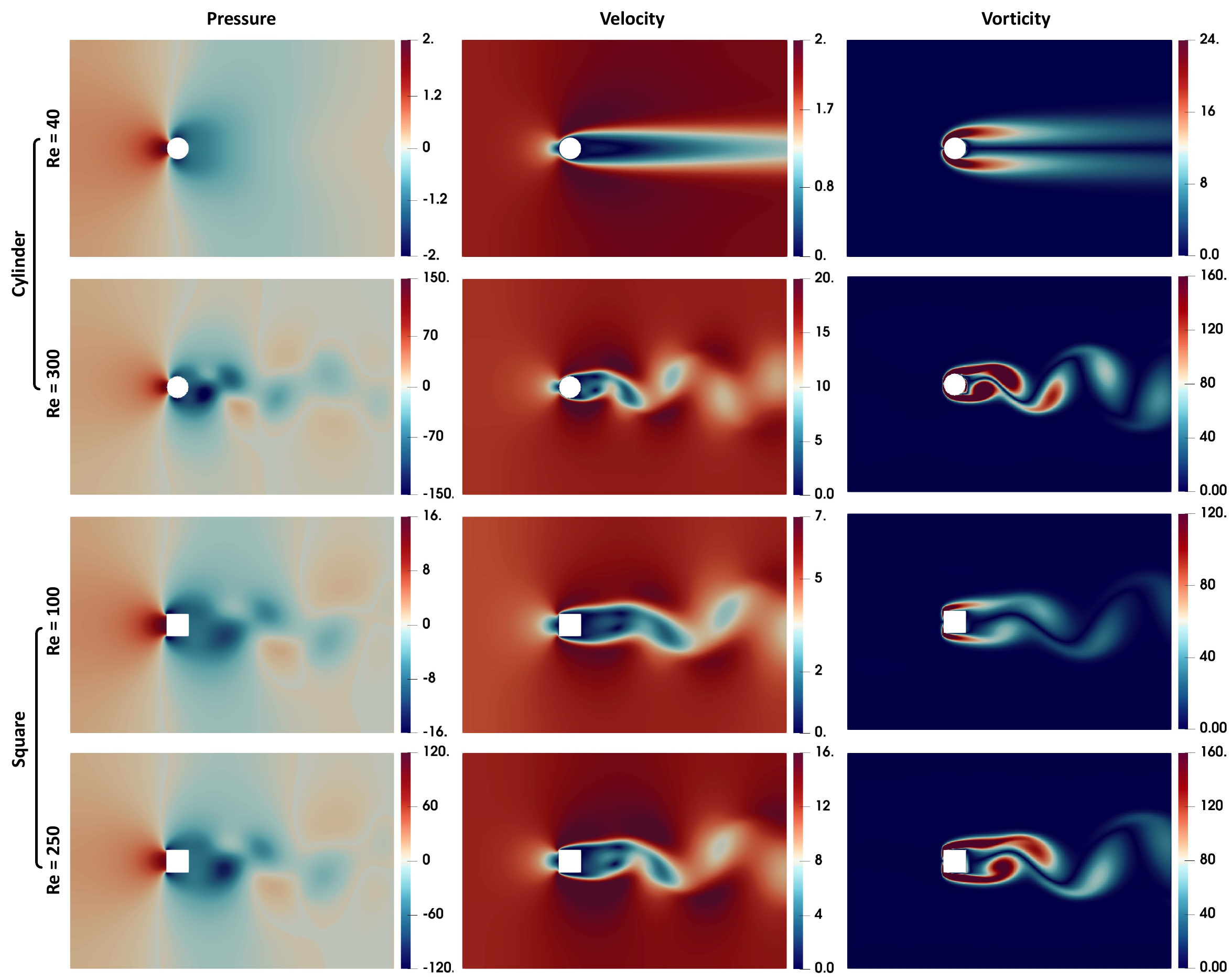}}
  \caption{Pressure, velocity magnitude, and vorticity magnitude for flow around circular and square obstacles at different $Re$.}
\label{fig:benchmark_sol}
\end{figure}

\begin{table}[h!]\centering
	\caption{Comparison of drag coefficient ($C_D$) computed here against literature for flow around a circle and square at different $Re$.}
	{\small
  	\begin{tabular}{ P{2cm} | P{3cm} | P{3cm} | P{5cm} |}
  	\cline{2-4}
	{} & Reynolds number &  This work, $C_D$ & Literature, $C_D$\\
	\hline
	\rowcolor[gray]{0.9}
	\multicolumn{1}{|c|}{\cellcolor{white}} & $Re = 20$ & 2.18 & 2.00~\cite{fornberg1980numerical}, 2.09~\cite{sheard2005computations}, 2.20~\cite{wieselsberger1921neuere} \\
	\multicolumn{1}{|c|}{\cellcolor{white}} & $Re = 40$ & 1.63 & 1.62~\cite{calhoun2002cartesian}, 1.53~\cite{sheard2005computations}, 1.60~\cite{russell2003cartesian} \\
	\rowcolor[gray]{0.9}
	\multicolumn{1}{|c|}{\cellcolor{white}} & $Re = 100$ & 1.37 & 1.35~\cite{beaudan1995numerical}, 1.39~\cite{silva2003numerical}, 1.35~\cite{main2018shifted}\\
	\multicolumn{1}{|c|}{\cellcolor{white}\multirow{-4}{*}{Cylinder}} & $Re = 300$ & 1.32 & 1.38~\cite{main2018shifted}, 1.28~\cite{wieselsberger1921neuere}, 1.37~\cite{rajani2009numerical} \\
	\hline
	\rowcolor[gray]{0.9}
	\multicolumn{1}{|c|}{\cellcolor{white}} & $Re = 100$ & 1.57 & 1.59~\cite{khademinezhad2015numerical}, 1.51~\cite{lam2012numerical} \\
	\multicolumn{1}{|c|}{\cellcolor{white}\multirow{-2}{*}{Square}} & $Re = 250$ & 1.71 & 1.66~\cite{main2018shifted}, 1.69-1.72~\cite{saha2003three} \\
	\hline
 	\end{tabular}
	}
	\label{tab:drag_cof}
\end{table}

\section*{References}
\bibliographystyle{unsrtnat}
\bibliography{./References.bib}

\end{document}